\documentclass[11pt,reqno]{amsart}
\usepackage[a4paper, hmargin={2.7cm,2.7cm},vmargin={3.3cm,3.3cm}]{geometry}
\usepackage[T1]{fontenc}
\usepackage{lmodern}
\usepackage{amsmath,amssymb,amsthm,mathtools}
\mathtoolsset{showonlyrefs}
\usepackage{todonotes}
\usepackage{comment}
\usepackage[normalem]{ulem}

\usepackage[noadjust]{cite}
\usepackage{bm}
\usepackage{booktabs,longtable,array}
\usepackage{enumitem}
\usepackage{xcolor}
\usepackage{calc}

\usepackage{hyperref}
\usepackage{cleveref}
\usepackage{setspace}
\usepackage{mathrsfs}
\usepackage{listings}
\usepackage{seqsplit}
\usepackage{needspace}
\usepackage{datetime}

\lstdefinestyle{pythonappendix}{
  language=Python,
  basicstyle=\ttfamily\scriptsize,
  keywordstyle=\color{deepblue},
  commentstyle=\color{darkgreen},
  stringstyle=\color{warningred},
  numbers=left,
  numberstyle=\tiny\color{gray},
  numbersep=8pt,
  frame=single,
  backgroundcolor=\color{softgray},
  breaklines=true,
  breakatwhitespace=false,
  columns=fullflexible,
  keepspaces=true,
  showstringspaces=false,
  tabsize=4,
  captionpos=b
}

\definecolor{warningred}{RGB}{145,25,25}
\definecolor{deepblue}{RGB}{24,64,118}
\definecolor{darkgreen}{RGB}{30,105,60}
\definecolor{softgray}{RGB}{246,247,249}
\hypersetup{
 colorlinks=true,
 linkcolor=deepblue,
 citecolor=deepblue,
 urlcolor=deepblue,
 pdftitle={Linear dependence of time-frequency shifts of Schwartz functions},
 pdfauthor={}
}

\numberwithin{equation}{section}

\newtheorem*{conjecture}{Conjecture}

\newtheorem{theorem}{Theorem}[section]
\newtheorem{lemma}[theorem]{Lemma}
\newtheorem{proposition}[theorem]{Proposition}
\newtheorem{corollary}[theorem]{Corollary}

\theoremstyle{definition}
\newtheorem{definition}[theorem]{Definition}
\newtheorem{remark}[theorem]{Remark}
\newtheorem*{remark*}{Remark}

\newcommand{\eps}{\varepsilon}

\newcommand{\Z}{\mathbb Z}
\newcommand{\R}{\mathbb R}
\newcommand{\Q}{\mathbb Q}
\newcommand{\C}{\mathbb C}
\newcommand{\N}{\mathbb N}
\newcommand{\ii}{\mathrm i}
\newcommand{\e}{\mathrm e}
\newcommand{\dd}{\,\mathrm d}

\newcommand{\norm}[1]{\left\lVert #1\right\rVert}
\newcommand{\abs}[1]{\left\lvert #1\right\rvert}

\newcommand{\supp}{\operatorname{supp}}

\newcommand{\Log}{\operatorname{Log}}

\newcommand{\op}{\mathrm{op}}
\newcommand{\Zak}{\mathcal{Z}}

\newcommand{\FrakI}{\mathfrak{I}}
\newcommand{\BoldEta}{\boldsymbol{\eta}}
\newcommand{\CalX}{\mathcal{X}}
\newcommand{\CalS}{\mathcal{S}}
\newcommand{\Ball}{\mathcal{B}}

\newcommand{\PhysicalVariable}{t}
\newcommand{\PhasePoint}{z}
\newcommand{\TranslationParameter}{x}
\newcommand{\ModulationParameter}{\omega}

\newcommand{\TimeFrequencyShift}{\pi}
\newcommand{\WeylShift}{\rho}
\newcommand{\GenericWeylPolynomial}{\mathcal P}

\newcommand{\CubicIrrational}{\vartheta}
\newcommand{\TranslationAlpha}{\alpha}
\newcommand{\TranslationBeta}{\beta}
\newcommand{\WeylOffset}{\zeta}
\newcommand{\ZakPoint}{z}
\newcommand{\z}{\ZakPoint}
\newcommand{\TorusShiftVector}{\tau}
\newcommand{\TranslateZak}{T_{\TorusShiftVector}}

\newcommand{\CoefficientIndexSet}{\mathcal I}
\newcommand{\RealCoefficientNumerator}{p}
\newcommand{\ImagCoefficientNumerator}{q}
\newcommand{\FixedWeylCoefficient}{a}
\newcommand{\aaa}{\FixedWeylCoefficient}

\newcommand{\ScalarZak}{\mathcal Z}
\newcommand{\VectorZak}{\mathcal Z_2}
\newcommand{\ZakPosition}{x}
\newcommand{\ZakFrequency}{\omega}
\newcommand{\FirstSewingMatrix}{U_1}
\newcommand{\SecondSewingMatrix}{U_2}
\newcommand{\SewingMatrix}{U}

\newcommand{\SmoothStep}{s_0}
\newcommand{\AngleFunction}{\theta}
\newcommand{\ReferenceSine}{a}
\newcommand{\ReferenceCosine}{b}
\newcommand{\ReferenceWindow}{f_0}
\newcommand{\ReferenceSection}{\chi}
\newcommand{\ReferenceProjection}{P}
\newcommand{\ComplementProjection}{Q}

\newcommand{\LatticeWeylMatrix}{L}

\newcommand{\OffsetPhase}{\eta}
\newcommand{\TranslatedReferenceSection}{w}
\newcommand{\IdealMatrixField}{A_0}
\newcommand{\IdealCocycle}{B_0}
\newcommand{\FixedLatticeOperator}{\mathcal P_*^{(0)}}
\newcommand{\FixedMatrixField}{A_*}
\newcommand{\FixedWeylOperator}{\mathcal P_*}
\newcommand{\FixedCocycle}{B_*}

\newcommand{\MatrixErrorField}{\mathcal E_A}
\newcommand{\LaurentVariable}{U}
\newcommand{\LaurentExponent}{\ell}
\newcommand{\LaurentSupport}{\mathcal L}
\newcommand{\FixedLaurentCoefficient}{\gamma}
\newcommand{\IdealLaurentCoefficient}{d}
\newcommand{\ErrorLaurentCoefficient}{e}
\newcommand{\LaurentMajorant}{\mathfrak B}
\newcommand{\CertifiedError}{\delta}
\newcommand{\DerivativeBound}{D_\ast}

\newcommand{\CorrectionField}{X}
\newcommand{\CandidateInvariantVector}{v_X}
\newcommand{\NormalizationScalar}[1]{q_{#1}}
\newcommand{\CocycleErrorField}{\mathcal E_B}
\newcommand{\CorrectionMap}{\Phi}
\newcommand{\FixedCorrection}{X_*}
\newcommand{\FixedInvariantVector}{v_*}
\newcommand{\InvariantMultiplier}{q_*}
\newcommand{\NormalSection}{\nu}
\newcommand{\UnitaryFrame}{U_\chi}
\newcommand{\ProjectiveCoordinate}{\xi}
\newcommand{\vv}{\ProjectiveCoordinate}
\newcommand{\LocalMatrix}{\mathscr M}
\newcommand{\LocalMatrixEntry}{\mu}
\newcommand{\ProjectiveMap}{\mathfrak F}

\newcommand{\ProjectiveContractionConstant}{\lambda_0}

\newcommand{\LogMultiplier}{\varphi}
\newcommand{\CohomologySolution}{u}
\newcommand{\ScalarCorrection}{h}
\newcommand{\FixedEigenvalue}{c_*}
\newcommand{\FinalZakEigenfunction}{F_*}
\newcommand{\FinalWindow}{f_*}

\newcommand{\GridSize}{G}
\newcommand{\GridInvariantVector}{v_G}
\newcommand{\GridMultiplier}{q_G}
\newcommand{\GridLogMultiplier}{\varphi_G}
\newcommand{\GridCohomologySolution}{u_G}
\newcommand{\GridScalarCorrection}{h_G}
\newcommand{\GridZakEigenfunction}{F_G}
\newcommand{\GridEigenvalue}{c_G}
\newcommand{\GridWindow}{f_G}
\newcommand{\GridScalarZak}{(\ScalarZak\GridWindow)}
\newcommand{\GridResidual}{R_G}
\newcommand{\GridSummand}{S}

\renewcommand{\Re}{\operatorname{Re}}
\renewcommand{\Im}{\operatorname{Im}}

\title[Linear dependence of time-frequency shifts  of a Schwartz function]
{Linear dependence of time-frequency shifts \\ of a Schwartz function}

\author[M. Faulhuber]{Markus Faulhuber}
\address{Faculty of Mathematics,
University of Vienna, 
Oskar-Morgenstern-Platz 1,
1090 Vienna, Austria}

\email{markus.faulhuber@univie.ac.at}

\author[P. Petersen]{Philipp Petersen}
\address{Faculty of Mathematics,
University of Vienna, 
Kolingasse 14-16,
1090 Vienna, Austria}
\email{philipp.petersen@univie.ac.at}

\author[J.T. van Velthoven]{Jordy Timo van Velthoven}
\address{Faculty of Mathematics,
University of Vienna, 
Oskar-Morgenstern-Platz 1,
1090 Vienna, Austria}
\email{jordy-timo.van-velthoven@univie.ac.at}

\author[F. Voigtlaender]{Felix Voigtlaender}
\address{
Mathematical Institute for Machine Learning and Data Science (MIDS),
Catholic University of Eichstätt–Ingolstadt (KU),
Auf der Schanz 49, 85049 Ingolstadt, Germany
}

\email{felix.voigtlaender@ku.de}

\thanks{All authors have contributed equally}

\subjclass[2020]{42C15}

\keywords{Linear dependence, Time-frequency shifts, Zak transform}

\begin{document}

\begin{abstract}
    We show that a finite number of time-frequency shifts of a Schwartz function can be linearly dependent.
    This disproves the so-called HRT conjecture of Heil, Ramanathan, and Topiwala.
    In particular, we provide an example consisting of 12 time-frequency shifts.
\end{abstract}

\maketitle

\section{Introduction}\label{sec:claim}

For a function $f \in L^2 (\R)$ and a point $ \PhasePoint=(\TranslationParameter,\ModulationParameter)\in\R^2$,
its associated \emph{time-frequency shift} is defined by
\begin{equation}
    \TimeFrequencyShift(\PhasePoint)f(\PhysicalVariable)
    =
    \e^{2\pi\ii\PhysicalVariable\ModulationParameter}
    f(\PhysicalVariable-\TranslationParameter),
    \qquad
    \PhasePoint=(\TranslationParameter,\ModulationParameter)\in\R^2.
\end{equation}
The following conjecture of Heil, Ramanathan, and Topiwala \cite{HRT} is widely known as the \emph{HRT conjecture}. 

\begin{conjecture}[\cite{HRT}]
    For a nonzero function $f \in L^2(\R)$ and a finite number of pairwise distinct points
$\PhasePoint_1,\ldots,\PhasePoint_n \in \R^2$, the vectors 
    $
       \TimeFrequencyShift(\PhasePoint_1)f,\ldots,
        \TimeFrequencyShift(\PhasePoint_n)f $
    are linearly independent over $\C$.
\end{conjecture}

The HRT conjecture is known to hold for various classes of special functions (see, e.g., \cite{bownik2016linear, bownik2013linear, romanos2025linear, guan2006HRT, benedetto2015linear, kreisel2019letter, HRT}) and special point configurations (see, e.g., \cite{okoudjou2025letter, demeter2012proof, Linnell, demeter2010linear, liu2019letter, Rzeszotnik2025point}); see also the surveys \cite{heil2006linear, heil2015hrt}. Among the most fundamental contributions to the HRT conjecture is Linnell's theorem \cite{Linnell}, which asserts that the conjecture is true for any finite subset of a discrete subgroup of $\R^2$; see also \cite{DemeterGautam, enstad2025linear, bownik2010linear, antezana2020linear} for various more elementary proofs of this theorem. 
We also mention \cite{grochenig2015linear} on the asymptotics of lower Riesz bounds of finite subsets of Gabor frames. The main result in that paper demonstrates that one may numerically observe linear dependencies even when the opposite holds analytically, e.g., in the lattice case.

In this paper, we show that the HRT conjecture is false in general. In particular, we show that a twelve-point counterexample exists. The following theorem is the main result of this paper.

\begin{theorem}\label{thm:main}
    There exist $\WeylOffset \in (\R\setminus\Q)^2$, nonzero  $\alpha_1, ..., \alpha_{12} \in\C $, pairwise distinct points $\PhasePoint_1, ..., \PhasePoint_{12} \in\R^2 $,
    and a nonzero Schwartz function $\FinalWindow \in \mathcal{S} (\R)$ such that
    $\PhasePoint_1,\ldots,\PhasePoint_{11}\in(\Z\times\frac12\Z)+\WeylOffset$,
    $\PhasePoint_{12}=0$, and
    \begin{equation}
        \sum_{k=1}^{12}
        \alpha_k\,\TimeFrequencyShift(\PhasePoint_k)
        \FinalWindow =0.
    \end{equation}
\end{theorem}

We show a numerical approximation to the function $\FinalWindow$ of Theorem~\ref{thm:main} in Appendix~\ref{sec:numerical-illustrations}.

The proof of Theorem~\ref{thm:main} relies on a certified numerical computation. To isolate the underlying mechanism and make the argument more transparent, we provide a separate, purely analytic proof of the following qualitative result.

\begin{theorem}\label{thm:qualitative}
There exists $\zeta \in (\R \setminus \Q)^2$, a finite subset $\Lambda \subseteq (\Z \times \frac{1}{2} \Z) + \zeta$ and a nonzero $\FinalWindow \in \mathcal{S}(\R)$ such that the system $\big \{ \pi(\lambda ) \FinalWindow : \lambda \in \Lambda \cup \{0\} \big\}$ is linearly dependent over $\C$.
\end{theorem}

Both theorems disprove the general HRT conjecture and its well-known formulation restricted to functions in the Schwartz space $\mathcal{S}(\R)$; see, e.g., \cite[Conjecture~2]{Okoudjo}. 
The addition of the origin to the translated half-integer lattice is essential in view of Linnell's theorem \cite{Linnell}. Moreover, the conclusion $\FinalWindow\in\mathcal S(\R)$ cannot be strengthened to arbitrarily fast decay, since the HRT conjecture holds under certain stronger decay conditions than rapid decay; see, e.g., \cite{benedetto2015linear, bownik2013linear, bownik2016linear}.

\subsection{Proof outline}
\label{sec:outline}
We first sketch the proof of the qualitative Theorem~\ref{thm:qualitative}, which isolates the mechanism of the construction. The proof of Theorem~\ref{thm:main} follows the same general principle and adds the certified numerical step needed to obtain exactly twelve shifts.

To prove Theorem~\ref{thm:qualitative}, we explicitly construct a Schwartz eigenfunction for a 
so-called \emph{Weyl polynomial}, i.e., a finite linear combination of Weyl shifts
\(
  \WeylShift(\PhasePoint)
  = e^{-\pi \ii \TranslationParameter \ModulationParameter} \TimeFrequencyShift(\PhasePoint),
\)
where $\PhasePoint = (\TranslationParameter, \ModulationParameter) \in \R^2$. We note that we use Weyl shifts (rather than time-frequency shifts) because they yield more symmetric formulae. Specifically, in order to prove Theorem~\ref{thm:qualitative}, we construct a nonzero $c \in\C $ and a nonzero $f \in\mathcal S(\R) $ satisfying the eigenvalue problem
\begin{equation}\label{eq:eigenfunction}
    \mathcal{P} f
    =c f
\end{equation}
for a Weyl polynomial 
\begin{equation} \label{eq:fixed_weyl}
    \mathcal{P}
    =\sum_{k=1}^{n}c_k\,\WeylShift(\PhasePoint_k),
\end{equation}
where $c_1, ..., c_n \in \C$ are nonzero coefficients and $\PhasePoint_1, ..., \PhasePoint_n \in \R^2$ are pairwise distinct points. 
Setting $\PhasePoint_{n+1} = 0 \in \R^2$ reduces this to
$\mathcal{P} f - c \WeylShift (z_{n+1}) f = 0$,
that is, the linear dependence of the vectors $\WeylShift(\PhasePoint_k) f$, $ k = 1, \dots, n + 1$.
In turn, this implies that also the vectors $\TimeFrequencyShift(\PhasePoint_k) f $, $ k = 1, \dots , n + 1 $, are linearly dependent,
settling Theorem~\ref{thm:qualitative}.

We mention that a spectral approach to the HRT conjecture was developed in \cite{Balan}. In particular, by \cite[Theorem~3.10]{Balan}, it is known that an eigenvalue for a finite linear combination of time-frequency shifts cannot be both isolated and of finite multiplicity. The present counterexample, however, shows that $0$ can occur in the point spectrum of such an operator. 

Instead of the general Weyl polynomials of the form \eqref{eq:fixed_weyl}, we consider Weyl polynomials of the more restricted form
\begin{equation}
\label{eq:Intro-StructuredWeylPolynomial}
  \FixedWeylOperator
  = \sum_{(m,n) \in \Z^2}
      \aaa_{m,n} \, \WeylShift \left( m + \TranslationAlpha, \frac{n + \TranslationBeta}{2} \right)
  ,
\end{equation}
with a finitely supported coefficient sequence $\aaa = (\aaa_{m,n})_{(m,n) \in \Z^2} \in \C^{\Z \times \Z}$ and a  translation parameter $ (\TranslationAlpha,{\TranslationBeta})\in(\R\setminus\Q)^2$. This particular choice allows us to rewrite the action of a Weyl polynomial, as in \eqref{eq:Intro-StructuredWeylPolynomial}, using the \emph{vector Zak transform} $\VectorZak f$ of a function $f \in L^2(\R)$, given by
\begin{equation}
    \VectorZak f (\ZakPosition,\ZakFrequency)
    = 2^{-1/2}
    \begin{pmatrix}
        \ScalarZak f \left(\ZakPosition, \ZakFrequency / 2 \right)\\
        \ScalarZak f \left(\ZakPosition,(\ZakFrequency+1) / 2\right)
    \end{pmatrix}.
\end{equation}
The vector Zak transform $\VectorZak$ is a unitary operator from $L^2 (\R)$ onto a Hilbert space of measurable functions $F: \R^2 \to \C^2$ satisfying certain quasi-periodicity relations. We call the elements of this space \emph{vector Zak functions}; see Section~\ref{sec:zak} for precise details. Using the vector Zak transform $\VectorZak$, the action of the Weyl polynomial $\FixedWeylOperator$ defined in \eqref{eq:Intro-StructuredWeylPolynomial} on a vector Zak function $F : \R^2 \to \C^2$ is given by
\begin{equation}\label{eq:B*}
    \left(\VectorZak\FixedWeylOperator\VectorZak^{-1} \right)F (\ZakPoint)
    =
    \FixedCocycle(\ZakPoint) F (\ZakPoint - \TorusShiftVector),
    \qquad \ZakPoint\in\R^2,
\end{equation}
for $\TorusShiftVector := (\TranslationAlpha, \TranslationBeta) \in \R^2$ and some explicit matrix-valued function $\FixedCocycle : \R^2 \to \C^{2 \times 2}$ depending on the coefficient sequence $\aaa = (\aaa_{m,n})_{m,n \in \Z^2}$ in \eqref{eq:Intro-StructuredWeylPolynomial}, cf.\ Section~\ref{sec:AdmissibleSequences}. As such, the  eigenvalue problem \eqref{eq:eigenfunction} for $\FixedWeylOperator$ can alternatively be solved by finding a nonzero $\FixedEigenvalue \in \C$ and a nonzero vector Zak function $\FinalZakEigenfunction : \R^2 \to \C^2$ satisfying
\begin{equation} \label{eq:intro_reformulation}
    \FixedCocycle(\ZakPoint)
    \FinalZakEigenfunction(\ZakPoint - \TorusShiftVector)
    =
    \FixedEigenvalue\FinalZakEigenfunction(\ZakPoint), \qquad \ZakPoint\in\R^2.
\end{equation}
We mention that solving an equation as in \eqref{eq:intro_reformulation} is the key reason why we work with a vector-valued Zak transform rather than the classical (scalar-valued) Zak transform, because a scalar-valued version of \eqref{eq:intro_reformulation} would not admit a nonzero scalar-valued function as a solution; see Lemma~\ref{lem:zak-zero}.

To solve \eqref{eq:intro_reformulation}, we construct an explicit smooth vector Zak function $\ReferenceSection : \R^2 \to \C^2$ satisfying
\begin{equation}
    \ReferenceSection(\ZakPoint)^*\ReferenceSection(\ZakPoint)
    = \| \ReferenceSection(z)\|^2 = 1,
    \qquad \ZakPoint\in\R^2,
\end{equation}
which naturally gives rise to the rank-one matrix function
\begin{equation}
    \IdealCocycle(\ZakPoint)
    =
    \ReferenceSection(\ZakPoint)
    \ReferenceSection(\ZakPoint - \TorusShiftVector)^*, \qquad z \in \R^2.
\end{equation}
For our purposes, the matrix function $\IdealCocycle$ is more convenient to work with than the general matrix function $\FixedCocycle$ appearing in \eqref{eq:intro_reformulation}. For this reason, we show that the set of possible matrix-valued functions $\FixedCocycle : \R^2 \to \C^{2 \times 2}$ in \eqref{eq:intro_reformulation} obtained by varying the coefficients $(\aaa_{m,n})_{(m,n) \in \Z^2}$ is sufficiently large to ensure for a suitable choice of the coefficients that
\begin{equation}
    \sup_{\ZakPoint\in \R^2}
    \left\|
      \FixedCocycle(\ZakPoint)-\IdealCocycle(\ZakPoint)
    \right\|_{\op}
    <\frac{1}{3};
    \label{eq:IntroCertifiedEstimate}
\end{equation}
we refer to Section~\ref{sub:LatticeWeylDensity} and Section~\ref{sec:AdmissibleSequences} for the details. For such a choice of coefficients $\aaa \in \C^{\Z \times \Z}$, we consider the associated matrix function $\FixedCocycle$ and fix the associated Weyl polynomial $\FixedWeylOperator$ as in \eqref{eq:Intro-StructuredWeylPolynomial}.
The close approximation \eqref{eq:IntroCertifiedEstimate} allows us in various arguments
to work with the more convenient matrix function $\IdealCocycle$ instead of  $\FixedCocycle$.
In particular, this estimate is crucial in a contraction argument to show the existence of a smooth,
nowhere-vanishing vector Zak function $\FixedInvariantVector : \R^2 \to \C^2$
and a smooth, $\Z^2$-periodic function $\InvariantMultiplier : \R^2 \to \C$ satisfying
\begin{equation} \label{eq:intro_scalarmultiplier}
    \FixedCocycle(\ZakPoint)
    \FixedInvariantVector(\ZakPoint - \TorusShiftVector)
    =
    \InvariantMultiplier(\ZakPoint)\FixedInvariantVector(\ZakPoint), \qquad z \in \R^2;
\end{equation}
see Section~\ref{sec:graph}. Note that the function $\FixedInvariantVector$ would solve \eqref{eq:intro_reformulation} if $\InvariantMultiplier$ was a constant function, which, however, is generally not the case. 

To solve \eqref{eq:intro_reformulation}, and hence the eigenvalue problem \eqref{eq:eigenfunction}, it remains to solve the equation
\begin{equation} \label{eq:difference_equation}
    \InvariantMultiplier(\ZakPoint)
    \ScalarCorrection(\ZakPoint - \TorusShiftVector)
    = \FixedEigenvalue\ScalarCorrection(\ZakPoint)
\end{equation}
by finding a smooth solution $h : \R^2 \to \C$. Indeed, defining $\FinalZakEigenfunction := \ScalarCorrection \cdot \FixedInvariantVector$, it would then follow from a combination of \eqref{eq:intro_scalarmultiplier} and \eqref{eq:difference_equation} that
\begin{align*}
\FixedCocycle(\ZakPoint)
    \FinalZakEigenfunction (\ZakPoint - \TorusShiftVector) &= 
 \FixedCocycle(\ZakPoint) \ScalarCorrection(\ZakPoint - \TorusShiftVector) \FixedInvariantVector(\ZakPoint - \TorusShiftVector) \\
 &= \InvariantMultiplier(\ZakPoint) \ScalarCorrection(\ZakPoint - \TorusShiftVector) \FixedInvariantVector(\ZakPoint) \\
 &= \FixedEigenvalue \FinalZakEigenfunction (\ZakPoint),
\end{align*}
which shows that the smooth vector Zak function $\FinalZakEigenfunction : \R^2 \to \C^2$ satisfies \eqref{eq:intro_reformulation}, cf. Section~\ref{sec:lifting} for precise details.
We show the existence of a solution to \eqref{eq:difference_equation}
in Section~\ref{sec:cohomology} using a Fourier series argument.
Therefore, mapping the vector Zak function $\FinalZakEigenfunction$
satisfying \eqref{eq:intro_reformulation} back via the vector Zak transform yields that
$\FinalWindow=\VectorZak^{-1}\FinalZakEigenfunction$
is a Schwartz function solving the eigenvalue problem \eqref{eq:eigenfunction} for $\FixedWeylOperator$ with eigenvalue $\FixedEigenvalue$, thus settling Theorem~\ref{thm:qualitative}.

\subsection{Usage of Large Language Models}
\label{sec:llm}

The counterexample and its proof strategy were developed by Large Language Models (LLMs), specifically ChatGPT (GPT-5.6 Pro), in a dialogue with the authors. We initiated the counterexample search, suggested increasing the number of points when an initial search got stuck, urged replacing abstract methods with more elementary arguments, and asked that an at-first-only-existential construction be converted into a fixed finite configuration.

The original counterexample consisted of the final configuration and a $C^3$ function. Going through the argument, we found that the regularity of the function can be increased, potentially to $C^\infty$ if an initially too strict version of Theorem~\ref{thm:certificate} was replaced by a looser version ($\CertifiedError < 1/4$ replaced by $\CertifiedError < 1/3$). This led to a concrete Schwartz-class counterexample.

We replaced several of ChatGPT's complex arguments with elementary proofs and filled numerous obvious gaps by providing the necessary computations or the relevant reference. From our perspective, the gaps in the proof were occasionally relatively substantial. For example, the entire argument of Section~\ref{sec:regOfZakFunction} and Appendix~\ref{sec:FaaDiBrunoByHand} was originally given with just the following two (ultimately correct) sentences: ``\emph{The fixed point inherits the finite regularity of the data. One may prove this by differentiating the fixed-point equation: at each derivative order, the highest derivative occurs linearly with the same fiber-contraction coefficient $< 1$, while lower-order terms are already known.}'' 

We checked whether the vector Zak construction can be simplified to a scalar version, which is not possible. We provide arguments for this fact, underlining that the dimensionality of the vector Zak transform is both necessary and minimal. We identified a nontrivial mistake in an initial proof of Theorem~\ref{lem:regularity}, which we then repaired with a new argument. 

In the end, we followed the proposed proof strategy, outlined in Section~\ref{sec:outline}, which remained valid.

The authors have independently verified, validated, and rewritten all parts of the paper influenced
by LLM-generated material and take full responsibility for the mathematical content of the paper.

\begin{remark*}
    The first version of this paper contained only the current Theorem~\ref{thm:main}, although an argument for an at-first-only-existential construction was already suggested by ChatGPT in the initial search for a counterexample (see Section~\ref{sec:llm}). After the first version of our paper appeared, this observation was also made by Tao~\cite{taoblog, supplement}, who pointed out that a merely qualitative counterexample (i.e., the current Theorem~\ref{thm:qualitative}) can be obtained without numerical arguments. In the present version, we formulate such a qualitative result separately as Theorem~\ref{thm:qualitative}. We decided to do so because we believe that it clarifies the main ideas of the construction and, moreover, requires no computer-assisted numerical verification. The original quantitative proof of Theorem~\ref{thm:main} is retained in the final section (Section~\ref{sec:certificate}), where certified numerics are used to obtain the explicit twelve-point counterexample.
\end{remark*}

\subsection{Notation}

We write $\N=\{1,2,\ldots\}$ for the natural numbers and set $\N_0=\N \cup \{0\}$. For $n \in \N$, the Euclidean norm on $\R^n$ or $\C^n$ is denoted by $\|\cdot\|$. The operator norm of a matrix will be denoted by $\|\cdot\|_{\op}$. For a matrix $A \in \C^{n \times n}$, we denote its conjugate transpose by $A^*$.

We use the usual notation for multi-indices and (partial) derivatives.
Specifically, for multi-indices $\alpha=(\alpha_1,\alpha_2)$ and
$\gamma=(\gamma_1,\gamma_2)$ in $\N_0^2$, we write
\[
|\gamma|=\gamma_1+\gamma_2,\qquad
\partial^\gamma
=\partial_{\ZakPosition}^{\gamma_1}
 \partial_{\ZakFrequency}^{\gamma_2},
\]
and $\alpha\leq\gamma$ means
$\alpha_j\leq\gamma_j$ for $j=1,2$.

\section{\texorpdfstring{Weyl-polynomial reformulation of the qualitative theorem}{Weyl-polynomial reformulation of the qualitative theorem}}
\label{sec:EigenvalueReformulation}

In this section, we reformulate the qualitative Theorem~\ref{thm:qualitative} as an eigenvalue problem for a specific Weyl polynomial, which is the problem we will actually solve.

\subsection{Weyl operators}
For a  point
$\PhasePoint=(\TranslationParameter,\ModulationParameter)\in\R^2$,
the Weyl shift
$\WeylShift(\PhasePoint)
 =\WeylShift(\TranslationParameter,\ModulationParameter)$
is the unitary operator on $L^2 (\R)$ given by
\begin{equation}\label{eq:Weyl}
 \WeylShift(\TranslationParameter,\ModulationParameter)
 f(\PhysicalVariable)
 =
 \e^{2\pi\ii\ModulationParameter
       (\PhysicalVariable-\TranslationParameter/2)}
 f(\PhysicalVariable-\TranslationParameter).
\end{equation}
The term phase space reflects that $\PhasePoint$ records both a 
translation by $\TranslationParameter$ in the physical variable and a
modulation by $\ModulationParameter$ in the dual frequency variable.
Thus, $\R^2$ parametrizes the combined time--frequency position
of a Weyl shift.
Note that $\WeylShift(\PhasePoint)$ is a symmetric
version of $\TimeFrequencyShift(\PhasePoint)$, and the two operators
coincide up to a phase factor. We have
\begin{equation}\label{eq:Weyl_vs_TF-shifts}
 \WeylShift(\PhasePoint)
 =
 \e^{-\pi\ii\TranslationParameter\ModulationParameter}
 \TimeFrequencyShift(\PhasePoint).
\end{equation}
Thus, the HRT conjecture can equivalently be formulated for Weyl shifts by absorbing the occurring phase factors into the coefficients $\{c_k\}$. The Weyl shifts closely follow the composition law of the Heisenberg group; see, e.g., \cite{Fol89, Gro01}.

The composition of two Weyl shifts yields
\begin{equation}\label{eq:Weyl-shift_composition}
    \WeylShift(\TranslationParameter,\ModulationParameter)
    \WeylShift(\TranslationParameter',\ModulationParameter')
    =  \e^{\pi\ii(\ModulationParameter\TranslationParameter' -\TranslationParameter\ModulationParameter')}
    \WeylShift(
    \TranslationParameter+\TranslationParameter',
    \ModulationParameter+\ModulationParameter')
\end{equation}
for $(\TranslationParameter, \ModulationParameter), (\TranslationParameter', \ModulationParameter') \in \R^2$.

\begin{definition}
For coefficients $c_1, ..., c_n \in \C$ and pairwise distinct points
$\PhasePoint_1,\ldots,\PhasePoint_n \in \R^2$, a
\emph{finite Weyl polynomial} is an operator of the form
\begin{equation}
    \GenericWeylPolynomial
    =
    \sum_{k=1}^n c_k\WeylShift(\PhasePoint_k),
\end{equation}
 The \emph{support} of a finite Weyl polynomial is defined as
$
    \operatorname{supp} (\GenericWeylPolynomial)
    =
    \{\PhasePoint_k:c_k\ne0\}
    \subset\R^2.
$
\end{definition}
The HRT conjecture is equivalent to the statement that a finite Weyl polynomial with a nonempty support cannot annihilate a nonzero function. Thus, a counterexample consists of a nonzero $f\in L^2(\R)$, pairwise distinct points $\PhasePoint_k$, and coefficients $c_k$, not all zero, such that
\begin{equation}
    \GenericWeylPolynomial f
    =
    \sum_{k=1}^n c_k\WeylShift(\PhasePoint_k) f
    =0.
\end{equation}
The following subsection describes the specific point configuration we will work with.

\subsection{Choice of point configuration}\label{sec:Counterexample}
 Throughout this paper, we set
\begin{equation}\label{eq:arith-data}
    \CubicIrrational := \sqrt[3]{2},\qquad
    \TranslationAlpha := \CubicIrrational-1,\qquad
    \TranslationBeta := \CubicIrrational^2-1,\qquad
    \WeylOffset
    := \left(\TranslationAlpha,\frac{\TranslationBeta}{2}\right),
    \qquad
    \TorusShiftVector := (\TranslationAlpha,\TranslationBeta).
\end{equation}
Note that $0 < \TranslationAlpha < \TranslationBeta < 1$.
For a finitely supported sequence $(\FixedWeylCoefficient_{m,n})_{(m,n) \in \Z^2} \in \C$,  we consider the operators
\begin{equation}
    \sum_{(m,n) \in \Z^2}
    \FixedWeylCoefficient_{m,n} \WeylShift\left(m, \frac{n}{2}\right)
    \qquad \text{and} \qquad 
    \WeylShift(\WeylOffset) = \WeylShift \left(\TranslationAlpha, \frac{\TranslationBeta}{2}\right),
\end{equation}
and note that their composition is given by
\begin{equation}\label{eq:Tstar-preview}  
    \sum_{(m,n) \in \Z^2}
    \FixedWeylCoefficient_{m,n} \WeylShift\left(m, \frac{n}{2}\right)
    \WeylShift(\WeylOffset) =
    \sum_{(m,n)\in \Z^2}
    \FixedWeylCoefficient_{m,n}
    \e^{\frac{\pi\ii}{2}(n\TranslationAlpha-m\TranslationBeta)} 
    \WeylShift\!\left(m+\TranslationAlpha,\frac{n+\TranslationBeta}{2}\right).
\end{equation}
The operator in \eqref{eq:Tstar-preview} is  a Weyl polynomial with coefficients $c_{m,n} =  \FixedWeylCoefficient_{m,n}\e^{\frac{\pi\ii}{2}(n\TranslationAlpha-m\TranslationBeta)}$ and points $z_{m,n} = (m + \alpha, (n + \beta)/2) \in \Z \times \frac{1}{2} \Z + \WeylOffset$ for $(m,n) \in \Z^2$. The points $z_{m,n}$ are clearly pairwise distinct. As such, for any nonzero $c \in \C$, the support of the operator
\[
\sum_{(m,n)\in \Z^2}
        \FixedWeylCoefficient_{m,n}\e^{\frac{\pi\ii}{2}(n\TranslationAlpha-m\TranslationBeta)} 
        \WeylShift\!\left(m+\TranslationAlpha,\frac{n+\TranslationBeta}{2}\right) - cI
\]
is given by
$
    \{(0,0)\}\cup
    \left\{z_{m,n} : a_{m,n} \neq 0 \right\}.
$

For proving the qualitative Theorem~\ref{thm:qualitative}, we will show the following theorem:

\begin{theorem}\label{thm:main_Weyl}
There exist a nonzero scalar $\FixedEigenvalue \in \C$, a nonzero finitely supported sequence $(a_{m,n} )_{(m,n) \in \Z^2} \in \C^{\Z \times \Z}$
and a nonzero $\FinalWindow\in\mathcal S(\R)$ such that
\begin{equation}\label{eq:fixed-dependence}
    -\FixedEigenvalue\FinalWindow+
    \sum_{(m,n)\in \Z^2}
      \FixedWeylCoefficient_{m,n} \, \e^{\frac{\pi\ii}{2}(n\TranslationAlpha-m\TranslationBeta)}
      \WeylShift\!\left(m+\TranslationAlpha,\frac{n+\TranslationBeta}{2}\right)\FinalWindow=0.
\end{equation}
\end{theorem}

This is the reformulation of the qualitative Theorem~\ref{thm:qualitative}, using Weyl shifts $\WeylShift$ instead of time-frequency shifts~$\TimeFrequencyShift$. We will provide a proof of Theorem~\ref{thm:main_Weyl} in Section~\ref{sec:lifting}.

\section{Zak transforms and Zak functions}\label{sec:zak}
This section recalls the definition of the Zak transform and defines the notion of vector Zak functions, which will play an important role throughout the paper.

\subsection{Zak transforms}
For $f\in\mathcal S(\R)$ and
$\ZakPoint=(\ZakPosition,\ZakFrequency)\in\R^2$, define the scalar
Zak transform by
\begin{equation}\label{eq:scalar-zak}
    \ScalarZak f(\ZakPosition,\ZakFrequency)
    = \sum_{k\in\Z} f(\ZakPosition-k)\e^{2 \pi \ii k \ZakFrequency}.
\end{equation}
The sum converges absolutely because of the rapid decay of $f$, and
$\ScalarZak$ extends to a unitary operator from $L^2(\R)$ onto $L^2 ([0,1]^2)$, see, e.g., \cite[Theorem 8.2.3]{Gro01}. It is
straightforward to check that $\ScalarZak f$ satisfies the \emph{quasi-periodicity relations}
\begin{align}
    \ScalarZak f(\ZakPosition+1,\ZakFrequency)
    & = \e^{2\pi\ii\ZakFrequency}
    \ScalarZak f(\ZakPosition,\ZakFrequency),
    \label{eq:Zak_quasiperiodic}
    \\
    \ScalarZak f(\ZakPosition,\ZakFrequency+1)
    & = \ScalarZak f(\ZakPosition,\ZakFrequency).
    \label{eq:Zak_periodic}
\end{align}
By using the explicit action \eqref{eq:Weyl} of $\WeylShift$ on $f$ and elementary calculations, we obtain the following covariance principle for Weyl shifts (cf.\ \cite[Chap.~8.2]{Gro01}) and the Zak transform.
\begin{align}\label{eq:zak-covariance}
    \Zak\bigl(\rho(x_0,\omega_0)f\bigr)(x,\omega)
    & = \sum_{k\in\Z}
    \bigl(\rho(x_0,\omega_0)f\bigr)(x-k) \e^{2\pi \ii k\omega}
    \\
    & = \sum_{k\in\Z}
    \e^{2\pi \ii \omega_0(x-k-x_0/2)}
    f(x-k-x_0) \e^{2 \pi \ii k\omega}
    \\
    & = \e^{2 \pi \ii \omega_0(x-x_0/2)}
    \sum_{k\in\Z} f(x-x_0-k) \e^{2 \pi \ii k(\omega-\omega_0)}
    \\
    & = \e^{2 \pi \ii\omega_0(x-x_0/2)}
    \Zak f(x - x_0, \omega - \omega_0) \\
    &  = \e^{2 \pi \ii\omega_0(x-x_0/2)}
    \Zak f(T_{(x_0, \omega_0)} (x, \omega)),
\end{align}
where we use the notation
\begin{equation} \label{eq:shift}
    T_{z_0} z = z - z_0
\end{equation}
for points $z, z_0 \in \R^2$.

The following observation explains why the scalar Zak transform cannot be used in our construction. Every continuous function satisfying the quasi-periodicity conditions \eqref{eq:Zak_quasiperiodic} and \eqref{eq:Zak_periodic} must vanish somewhere in a fundamental square, see, e.g., \cite[Chapter~8.4]{Gro01}. This statement is included in the following lemma. 

\begin{lemma}\label{lem:zak-zero}
    Let $F\colon \R^2 \to \C$ be continuous and satisfy the quasi-periodicity relations
    \eqref{eq:Zak_quasiperiodic} and \eqref{eq:Zak_periodic}. Then $F$ has a zero in $[0,1]^2$.
    In addition, suppose that
    \begin{equation}\label{eq:scalar-zak-functional-equation}
        b(\ZakPoint)F(\TranslateZak\ZakPoint)
        =
        cF(\ZakPoint)
        \quad \text{for all} \quad
        \ZakPoint \in \R^2,
    \end{equation}
    for some $b\colon \R^2 \to \C$ and $c\neq 0$. Then $F \equiv 0$.
\end{lemma}

\begin{proof}
    The fact that $F$ has a zero in $[0,1]^2$ can be found in, e.g., \cite[Lemma 8.4.2]{Gro01}.
    
    Moreover, we note that the quasi-periodicity relations
    \eqref{eq:Zak_quasiperiodic} and \eqref{eq:Zak_periodic} imply that the zero set of $F$
    is invariant under translations by $\Z^2$. Indeed, for every $\ZakPoint\in\R^2$ and
    $m\in\Z^2$,
    \begin{equation}
        F(\ZakPoint+m)
        =
        \gamma_m(\ZakPoint)F(\ZakPoint)
    \end{equation}
    for some nonzero $\gamma_m(\ZakPoint)\in\C$.

    For the second assertion, let $\ZakPoint_0\in [0,1]^2$ satisfy $F(\ZakPoint_0)=0$. Evaluating
    \eqref{eq:scalar-zak-functional-equation} at
    $\TranslateZak^{-1}\ZakPoint_0$ gives
    \begin{equation}
        b(\TranslateZak^{-1}\ZakPoint_0)F(\ZakPoint_0)
        =
        cF(\TranslateZak^{-1}\ZakPoint_0).
    \end{equation}
    Since the left-hand side is zero by assumption and $c\neq0$, we obtain
    $
        F(\TranslateZak^{-1}\ZakPoint_0)=0.
    $
    Iterating this argument yields
    $
        F(\TranslateZak^{-k}\ZakPoint_0)=0,
   $ for $k \in \N_0$.
   
    Finally, let $\ZakPoint\in\R^2$ be arbitrary. Since $1, \TranslationAlpha, \TranslationBeta$ are linearly independent over $\Q$ and $\TorusShiftVector := (\TranslationAlpha, \TranslationBeta)$, it follows from Kronecker's theorem that $\{\TranslateZak^{-k}\ZakPoint:k\in\N_0\} \subset \R^2$  is dense modulo $\Z^2$, i.e., its image under the projection $p : \R^2 \to \R^2 / \Z^2$ is dense. Therefore, there exist sequences $(k_j)_{j\in\N}\subset\N_0$ and $(m_j)_{j\in\N}\subset\Z^2$ such that
    $
        \TranslateZak^{-k_j}\ZakPoint_0-m_j
        \to
        \ZakPoint
    $
    as $j \to \infty$. Since the zero set of $F$ is invariant under translations by $\Z^2$, we thus obtain
    \begin{equation}
        F(\TranslateZak^{-k_j}\ZakPoint_0-m_j)=0
        \qquad \text{for all} \qquad j\in\N.
    \end{equation}
    By continuity of $F$, it follows that $F(\ZakPoint)=0$. This shows that $F \equiv 0$.
\end{proof}

\subsection{The vector Zak transform and Zak functions}
For $f\in L^2(\R)$ and $\ZakPoint=(\ZakPosition,\ZakFrequency)\in\R^2$, define the two-component vector Zak transform by
\begin{equation}\label{eq:vector Zak}
    (\VectorZak f)_{r+1}(\ZakPosition,\ZakFrequency)
    = 2^{-1/2} \, \ScalarZak f
    \left(\ZakPosition,\frac{\ZakFrequency+r}{2}\right),
    \quad r\in\{0,1\}.
\end{equation}
The vector Zak transform is thus given by 
\begin{equation}
    \VectorZak f (\ZakPosition,\ZakFrequency)
    = 2^{-1/2}
    \begin{pmatrix}
        \ScalarZak f \left(\ZakPosition,\frac{\ZakFrequency}{2}\right)\\
        \ScalarZak f \left(\ZakPosition,\frac{\ZakFrequency+1}{2}\right)
    \end{pmatrix}.
\end{equation}
The target  space $\mathscr{H}_\Zak$ consists of (equivalence classes of) measurable functions $F : \R^2\to\C^2$ satisfying, for a.e. $(\ZakPosition, \ZakFrequency) \in \R^2$,
\begin{equation}\label{eq:x-sewing}
 F (\ZakPosition+1,\ZakFrequency) = \FirstSewingMatrix(\ZakFrequency) F (\ZakPosition,\ZakFrequency),
 \qquad
 \FirstSewingMatrix(\ZakPosition, \ZakFrequency)
 := \FirstSewingMatrix(\ZakFrequency) 
 := \e^{\pi\ii\ZakFrequency}
 \begin{pmatrix}
    1&0\\
    0&-1
 \end{pmatrix},
\end{equation}
\begin{equation}\label{eq:omega-sewing}
 F (\ZakPosition,\ZakFrequency+1)
 =
 \SecondSewingMatrix F (\ZakPosition,\ZakFrequency),
 \qquad
 \SecondSewingMatrix (\ZakPosition, \ZakFrequency)
 := \SecondSewingMatrix
 := \begin{pmatrix}0&1\\1&0\end{pmatrix},
\end{equation}
and having a finite norm
\begin{equation}
 \norm{F}_{\mathscr H_\Zak}^2
 =\int_{[0,1]^2}
 F (z)^*
 F (z)
 \,\dd z.
\end{equation}
We remark that $\FirstSewingMatrix(\ZakPosition, \ZakFrequency)$ is independent of $\ZakPosition$, whereas $\SecondSewingMatrix(\ZakPosition, \ZakFrequency)$ is independent of $\ZakPosition, \ZakFrequency$.
Nevertheless, this notation will prove convenient for many proofs.

\begin{definition}
    A measurable function $F: \R^2 \to \C^2$ satisfying the relations \eqref{eq:x-sewing} and \eqref{eq:omega-sewing} is said to be a \emph{vector Zak function}. The relations \eqref{eq:x-sewing} and \eqref{eq:omega-sewing} are sometimes referred to as \emph{quasi-periodicity relations} for $\C^2$-valued functions.
    
    In addition, a matrix-valued function $C:\R^2\to \C^{2 \times 2}$ is said to satisfy the (matrix) \emph{quasi-periodicity relations} if, for $(\ZakPosition, \ZakFrequency) \in \R^2$,
\begin{equation}\label{eq:sewing}
    C(\ZakPosition+1,\ZakFrequency)
    =
    \FirstSewingMatrix(\ZakFrequency)
    C(\ZakPosition,\ZakFrequency)
    \FirstSewingMatrix(\ZakFrequency)^{-1},
    \qquad
    C(\ZakPosition,\ZakFrequency+1)
    =
    \SecondSewingMatrix
    C(\ZakPosition,\ZakFrequency)
    \SecondSewingMatrix^{-1}.
\end{equation}
 where the matrices $\FirstSewingMatrix$ and $\SecondSewingMatrix$ are as in \eqref{eq:x-sewing} and \eqref{eq:omega-sewing}, respectively.
\end{definition}

Note that if a measurable map $C : \R^2 \to \C^{2 \times 2}$ satisfies the quasi-periodicity relations \eqref{eq:sewing},
then pointwise multiplication by $C$ maps vector Zak functions to vector Zak functions.

Lastly, we mention the following lemma. Although it is expected to be part of the folklore (see, e.g., \cite[Section 2.2]{ZibZee97}), we provide a short proof for completeness.

\begin{lemma} \label{lem:vector_zak_unitary}
    The vector Zak transform $\VectorZak$ forms an isometry from $L^2 (\R)$ onto $\mathscr{H}_\Zak$. In particular, the space $\mathscr{H}_\Zak$ is a Hilbert space.
\end{lemma}
\begin{proof}
The quasi-periodicity relations of the Zak transform $\ScalarZak f$ show that $\VectorZak f$ satisfies the quasi-periodicity relations \eqref{eq:x-sewing} and \eqref{eq:omega-sewing}. Moreover, splitting the frequency integral into two halves gives
\begin{align*}
    \norm{\VectorZak f}_{\mathscr H_\Zak}^2
    &=\frac12\int_0^1\int_0^1
    \left(
    \left|\ScalarZak f\left(\ZakPosition,\frac{\ZakFrequency}{2}\right)\right|^2
    +
    \left|\ScalarZak f\left(\ZakPosition,\frac{\ZakFrequency+1}{2}\right)\right|^2
    \right)
    \,\dd\ZakPosition\,\dd\ZakFrequency\\
    &=\int_0^1\int_0^1
    |\ScalarZak f(\ZakPosition,\nu)|^2
    \,\dd\ZakPosition\,\dd\nu
    =\norm{f}_{L^2(\R)}^2.
\end{align*}
Thus $\VectorZak$ is an isometry. To prove surjectivity, let
$F = (f_1,f_2)^\mathsf{T}\in\mathscr H_\Zak$ and define
$\widetilde{F} \in L^2([0,1]^2)$ by
\begin{equation}\label{eq:inverse-vector Zak-folding}
 \widetilde{F}(\ZakPosition,\nu)
 :=
 \begin{cases}
  \sqrt2\,f_1(\ZakPosition,2\nu),
      &0\leq\nu<\frac12,\\
  \sqrt2\,f_2(\ZakPosition,2\nu-1),
      &\frac12\leq\nu\leq1.
 \end{cases}
\end{equation}
Since the Zak transform $\ScalarZak : L^2 (\R) \to L^2 ([0,1]^2)$ is unitary, there exists $f\in L^2(\R)$
such that $\ScalarZak f=\widetilde{F}$ on $[0,1]^2$. The definition of $\widetilde{F}$ then gives
$\VectorZak f=F$ almost everywhere on $[0,1]^2$. Both sides satisfy the
quasi-periodicity relations, so this equality extends almost everywhere to $\R^2$.
Consequently, $\VectorZak:L^2(\R)\to\mathscr H_\Zak$ is unitary.
\end{proof}

\subsection{An explicit vector Zak function}
Define the smooth step function $s_0 : \R \to \R$ by
\begin{equation}\label{eq:smooth-step}
    \SmoothStep(\PhysicalVariable)
    :=
    \begin{cases}
        0,&\PhysicalVariable\le0,\\
        \displaystyle
        \frac{\exp(-1/\PhysicalVariable)}{\exp(-1/\PhysicalVariable)+\exp(-1/(1-\PhysicalVariable))},
        &0<\PhysicalVariable<1,\\
        1,&\PhysicalVariable\ge1,
    \end{cases}
    \qquad
    \AngleFunction(\PhysicalVariable) := \frac\pi2\SmoothStep(\PhysicalVariable),
\end{equation}
and let $\ReferenceWindow:\R\to\R$ be defined by
\begin{equation}\label{eq:g0}
    \ReferenceWindow(\PhysicalVariable)
    :=
    \begin{cases}
        \sin (\AngleFunction(\PhysicalVariable)),
        &0\le\PhysicalVariable\le1,\\
        \cos (\AngleFunction(\PhysicalVariable-1)),
        &1\le\PhysicalVariable\le2,\\
        0,&\text{otherwise}.
    \end{cases}
\end{equation}
We first note that $\ReferenceWindow \in C_c^{\infty} (\R)$.

\begin{lemma}\label{lem:g-smooth}
The function $\ReferenceWindow$ belongs to $C_c^\infty(\R)$.
\end{lemma}

\begin{proof}
    The function $\PhysicalVariable\mapsto\exp(-1/\PhysicalVariable)$, extended by zero for $\PhysicalVariable\le0$, is smooth and flat at the origin: every derivative there is zero. It follows from \eqref{eq:smooth-step}, and by the symmetry $\SmoothStep(1-\PhysicalVariable)=1-\SmoothStep(\PhysicalVariable)$, that $\SmoothStep$ is smooth, is flat at $0$, and that $1-\SmoothStep$ is flat at $1$. Consequently, the first branch in \eqref{eq:g0} and all of its derivatives vanish at $\PhysicalVariable=0$, while the second branch and all of its derivatives vanish at $\PhysicalVariable=2$. At $\PhysicalVariable=1$, both branches equal $1$; their nonconstant parts are flat there, so their derivatives of every positive order agree and vanish. Hence $\ReferenceWindow\in C^\infty(\R)$ and $\supp \ReferenceWindow\subset[0,2]$.
\end{proof}

For proving the main Theorem~\ref{thm:main}, we will also use the following estimate for the derivative of $\SmoothStep$. This lemma is not needed for the qualitative Theorem~\ref{thm:qualitative}.

\begin{lemma}\label{lem:s0-derivative-estimate}
    The function $\SmoothStep$ defined in \eqref{eq:smooth-step} satisfies $0 \leq \SmoothStep' (x) \leq 2$ for all $x \in \R$.
\end{lemma}

\begin{proof}
    On $(-\infty, 0)$ and $(1,\infty)$, we clearly have $\SmoothStep' (\ZakPosition) = 0$. By continuity of $\SmoothStep'$, it is thus enough to consider the case $\ZakPosition \in (0,1)$. To this end, put
    \begin{equation}
    t := 2\ZakPosition-1 \in (-1,1),
    \qquad
    r(t) := \frac{2t}{1-t^2} \in \R.
    \label{eq:s0-derivative-auxiliary-quantities}
    \end{equation}
    
    Writing $\SmoothStep$ in logistic form and differentiating, as detailed in Section~\ref{sub:smooth-step-derivative}, gives
    \begin{equation}
    \SmoothStep'(\ZakPosition)
    =
    2\frac{1+t^2}{(1-t^2)^2}\operatorname{sech}^2(r(t)) \geq 0.
    \label{eq:s0-derivative-sech-formula}
    \end{equation}
    Recall the well-known identity $\operatorname{sech}(r) = \frac{1}{\cosh(r)}$,
    where $\cosh(r) = \frac{1}{2}(e^r + e^{-r})$ satisfies
    \[
      \cosh(r)
      = \frac{1}{2}
        \sum_{n=0}^\infty \frac{r^n + (-r)^n}{n!}
      = \sum_{n \in \N_0, \, n \text{ even}}
          \frac{r^n}{n!}
      \geq \frac{r^0}{0!} + \frac{r^2}{2!}
      =  1 + \frac{r^2}{2}
    \]
    and hence $\cosh^2(r) \geq (1 + \frac{r^2}{2})^2 = 1 + r^2 + \frac{r^4}{4} \geq 1 + r^2$.
    This implies
    \begin{equation}
    \frac{1}{\operatorname{sech}^2(r(t))}
    =\cosh^2(r(t))
    \geq
    1+r(t)^2
    =
    \frac{(1+t^2)^2}{(1-t^2)^2}
    \geq
    \frac{1+t^2}{(1-t^2)^2}.
    \end{equation}
    Hence, $\operatorname{sech}^2(r(t)) \leq \frac{(1-t^2)^2}{1+t^2}$.
    Plugging this inequality into \eqref{eq:s0-derivative-sech-formula},
    we obtain $\SmoothStep'(\ZakPosition)\leq2$, as claimed.
\end{proof}

For $\PhysicalVariable\in \R$, put
\begin{equation}
    \ReferenceSine(\PhysicalVariable)
    := \sin (\AngleFunction(\PhysicalVariable)),
    \qquad
    \ReferenceCosine(\PhysicalVariable)
    := \cos (\AngleFunction(\PhysicalVariable)).
    \label{eq:ab-definition}
\end{equation}
Then $\ReferenceSine(\PhysicalVariable)^2+\ReferenceCosine(\PhysicalVariable)^2=1$.
Since $\ReferenceWindow$ is supported in $[0,2]$, for $\ZakPosition \in [0,1]$
only the terms $k=-1,0$ contribute in the Zak transform of $\ReferenceWindow$, i.e.,
\begin{equation}
    (\ScalarZak\ReferenceWindow)(\ZakPosition,\ZakFrequency)
    =
    \ReferenceSine(\ZakPosition)
    +\ReferenceCosine(\ZakPosition)\e^{-2\pi\ii\ZakFrequency}
    \qquad \text{for } \ZakPosition \in [0,1], \ZakFrequency \in \R.
\end{equation}
Consequently, the vector Zak function
\begin{align} \label{eq:explicit_vector_zak}
  \ReferenceSection := \VectorZak\ReferenceWindow
\end{align}
satisfies
\begin{equation}\label{eq:chi}
    \ReferenceSection(\ZakPosition,\ZakFrequency)
    = \VectorZak \ReferenceWindow(\ZakPosition, \ZakFrequency) = 2^{-1/2}
    \begin{pmatrix}
        \ReferenceSine(\ZakPosition)
        +\ReferenceCosine(\ZakPosition)\e^{-\pi\ii\ZakFrequency}\\[1mm]
        \ReferenceSine(\ZakPosition)
        -\ReferenceCosine(\ZakPosition)\e^{-\pi\ii\ZakFrequency}
    \end{pmatrix},
  \quad \text{for} \quad
  \ZakPosition \in [0,1] \text{ and } \ZakFrequency \in \R
  .
\end{equation}
Here, we used $\e^{-\pi\ii(\ZakFrequency + 1)} = \e^{-\pi\ii\ZakFrequency}\e^{-\pi\ii} = - \e^{-\pi\ii\ZakFrequency}$.

\begin{lemma}\label{lem:unit}
    The function $\ReferenceSection$ defined in \eqref{eq:explicit_vector_zak}
    belongs to $C^{\infty}(\R^2; \C^2)$ and satisfies the quasi-periodicity relations \eqref{eq:x-sewing} and \eqref{eq:omega-sewing}.
    Moreover, it satisfies
    \begin{equation}\label{eq:chi-unit}
        \ReferenceSection(\PhasePoint)^* \ReferenceSection(\PhasePoint) = 1
        \qquad\text{for all }
        \PhasePoint \in \R^2.
    \end{equation}
\end{lemma}

\begin{proof}
Smoothness follows from
$\ReferenceWindow\in C_c^\infty(\R)$ and the locally finite Zak sum.
Validity of  \eqref{eq:x-sewing} and \eqref{eq:omega-sewing} follows
since $\ReferenceSection = \VectorZak \ReferenceWindow \in \mathscr{H}_\Zak$.

To prove \eqref{eq:chi-unit}, first consider the case $\ZakPoint = (\ZakPosition, \ZakFrequency) \in [0,1)^2$.
In this case, \eqref{eq:chi} shows
\begin{align}
    \ReferenceSection^*(\ZakPoint) \ReferenceSection(\ZakPoint)
    &= \norm{\ReferenceSection(\ZakPoint)}^2
    = 2^{-1}\left(
    \abs{\ReferenceSine(\ZakPosition)
        +\ReferenceCosine(\ZakPosition) \e^{-\pi\ii\ZakFrequency}}^2
    +\abs{\ReferenceSine(\ZakPosition)
         -\ReferenceCosine(\ZakPosition) \e^{-\pi\ii\ZakFrequency}}^2
    \right) \\
    &=\ReferenceSine(\ZakPosition)^2+\ReferenceCosine(\ZakPosition)^2=1.
\end{align}
The sewing matrices $\FirstSewingMatrix(\ZakFrequency)$ and $\SecondSewingMatrix$ are unitary for real $\ZakFrequency$, so the identity extends to every fundamental cell.
\end{proof}

Lastly, we prove the following lemma on the partial derivative of $\ReferenceSection$ with respect to its first variable. This result is not needed for the proof of the qualitative Theorem~\ref{thm:qualitative}, but will be used in our proof of the main Theorem~\ref{thm:main}.

\begin{lemma}\label{lem:ChiDerivativeEstimate}
    For the function $\ReferenceSection : \R^2 \to \C^2$ introduced in
    \eqref{eq:explicit_vector_zak}, we have
    \[
      \| \partial_\ZakPosition \ReferenceSection (\ZakPosition, \ZakFrequency) \|
      \leq \pi
      \qquad \text{for all } (\ZakPosition, \ZakFrequency) \in \R^2
      .
    \]
\end{lemma}

\begin{proof}
  We first consider the case $x \in (\ell, \ell+1)$ with $\ell \in \Z$. In this case, if $k \in \Z$ with $\ReferenceWindow(\ZakPosition - k) \neq 0$, then necessarily $0 < \ZakPosition - k < 2$ and hence $k < \ZakPosition < \ell+1$ and $\ell < \ZakPosition < k + 2$, which implies $k \in \{\ell - 1, \ell \}$. Therefore, by the definition of the Zak transform, we see
  \begin{align*}
    \Zak \ReferenceWindow (\ZakPosition, \ZakFrequency)
    &= \sum_{k \in \Z} \ReferenceWindow(x - k) e^{2 \pi i k \ZakFrequency} \\
    &= e^{2 \pi i (\ell-1) \omega} \ReferenceWindow(x - (\ell-1))
      + e^{2 \pi i \ell \omega} \ReferenceWindow (x - \ell) \\
    &= e^{2 \pi i \ell \omega} \bigl( \ReferenceSine(x - \ell) + e^{-2 \pi i \omega} \ReferenceCosine(x - \ell) \bigr)
    ,
  \end{align*}
  where we used the definition \eqref{eq:g0} of $\ReferenceWindow$,
  the definition \eqref{eq:ab-definition} of $\ReferenceSine,\ReferenceCosine$,
  and the fact that $x - \ell \in (0,1)$ and $x - \ell + 1 \in (1,2)$.
  This implies
  \begin{align*}
    \partial_\ZakPosition (\Zak \ReferenceWindow) (\ZakPosition, \ZakFrequency)
    &= e^{2 \pi i \ell \omega}
       \bigl( \ReferenceSine'(\ZakPosition - \ell) + e^{-2 \pi i \omega} \ReferenceCosine'(\ZakPosition - \ell) \bigr) \\
    &= e^{2 \pi i \ell \omega}
       \bigl(
             \AngleFunction'(\ZakPosition-\ell) \cdot \ReferenceCosine(\ZakPosition - \ell)
             - e^{-2 \pi i \omega} \AngleFunction'(\ZakPosition-\ell) \ReferenceSine(\ZakPosition - \ell)
       \bigr) \\
    &= e^{2 \pi i \ell \omega} \cdot \AngleFunction'(\ZakPosition-\ell) \cdot
       \bigl(
             \ReferenceCosine(\ZakPosition - \ell)
             - e^{-2 \pi i \omega} \ReferenceSine(\ZakPosition - \ell)
       \bigr)
  \end{align*}
for $\ZakPosition \in (\ell, \ell+1), \ZakFrequency \in \R$.

  In view of the definition of the vector Zak transform $\VectorZak$, we have
  \[
    \ReferenceSection (\ZakPosition, \ZakFrequency)
    = \VectorZak \ReferenceWindow (\ZakPosition, \ZakFrequency)
    = 2^{-1/2} \begin{pmatrix}
        \Zak \ReferenceWindow (\ZakPosition, \frac{\ZakFrequency}{2}) \\[0.2cm]
        \Zak \ReferenceWindow (\ZakPosition, \frac{\ZakFrequency + 1}{2})
    \end{pmatrix}
  \]
  and hence
  \begin{align*}
    \| \partial_\ZakPosition \ReferenceSection (\ZakPosition, \ZakFrequency) \|
    &= 2^{-1/2} |\AngleFunction' (x - \ell)|
       \cdot
       \left\|
         \begin{pmatrix}
              \ReferenceCosine(\ZakPosition - \ell)
              - e^{-\pi i \omega} \ReferenceSine(\ZakPosition - \ell)
             \\[0.2cm]
              \ReferenceCosine(\ZakPosition - \ell)
              - e^{-\pi i (\omega+1)} \ReferenceSine(\ZakPosition - \ell)
         \end{pmatrix}
       \right\| \\
    &\leq 2^{-1/2} \pi
       \left\|
         \begin{pmatrix}
              \ReferenceCosine(\ZakPosition - \ell)
              - e^{-\pi i \omega} \ReferenceSine(\ZakPosition - \ell)
             \\[0.2cm]
              \ReferenceCosine(\ZakPosition - \ell)
              + e^{-\pi i \omega} \ReferenceSine(\ZakPosition - \ell)
         \end{pmatrix}
       \right\|
      = \pi
  \end{align*}
  where we used the estimate $|\AngleFunction'(t)| = \frac{\pi}{2} |\SmoothStep'(t)| \leq \pi$
  implied by Lemma~\ref{lem:s0-derivative-estimate},
  and the parallelogram identity $|A + B|^2 + |A - B|^2 = 2 (|A|^2 + |B|^2)$
  for $A,B \in \C$, as well as the fact that
  \[
    \ReferenceCosine^2 (\ZakPosition - \ell) + \ReferenceSine^2 (\ZakPosition-\ell)
    = \cos^2 (\AngleFunction(\ZakPosition-\ell))
      + \sin^2 (\AngleFunction(\ZakPosition-\ell))
    = 1
    .
  \]
  We have thus verified the claimed estimate on $\bigcup_{\ell \in \Z} (\ell,\ell+1)$.
  Since this is a dense set and $\partial_\ZakPosition \ReferenceSection$ is continuous,
  we are done.
\end{proof}

\section{The matrix field associated with the Weyl operator}
\label{sec:cocycles}

In this section, we provide a convenient representation of a Weyl operator using the vector Zak transform.

\subsection{Lattice Weyl matrices}

In accordance with the vector Zak transform $\VectorZak$, we will consider Weyl shifts $\WeylShift(m,n/2)$ along the half-integer lattice $\Z \times \frac{1}{2}\Z$. Our first goal is to show that applying $\VectorZak \, \WeylShift(m,n/2) \, \VectorZak^{-1}$ is multiplication by a $2 \times 2$ matrix $\LatticeWeylMatrix_{m,n}$.

Specifically, for $m,n\in\Z$, define
\begin{equation}\label{eq:Lmn-factorization}
    \LatticeWeylMatrix_{m,n}(\ZakPosition,\ZakFrequency)
    :=
    \e^{\pi\ii\left(n\ZakPosition-m\ZakFrequency+\frac{mn}{2}\right)}
    \LatticeWeylMatrix_{m,n}^{(0)}
    \in\C^{2\times2},
    \qquad
    (\ZakPosition,\ZakFrequency)\in\R^2,
\end{equation}
where
\begin{equation}\label{eq:Lmn-constant}
  \LatticeWeylMatrix_{m,n}^{(0)}
  :=
  \begin{cases}
    \left(
      \begin{smallmatrix}
        1&0\\
        0&\e^{-\pi\ii m}
      \end{smallmatrix}
    \right)
    = \left(
        \begin{smallmatrix}
          1 & 0\\
          0 & (-1)^m
        \end{smallmatrix}
      \right),
    & \text{if } n\equiv0\pmod 2,\\[3mm]
    \left(
     \begin{smallmatrix}
        0&1\\
        \e^{-\pi\ii m}&0
      \end{smallmatrix}
    \right)
    = \left(
       \begin{smallmatrix}
          0      & 1\\
          (-1)^m & 0
        \end{smallmatrix}
      \right)
    ,
    & \text{if } n\equiv1\pmod 2.
  \end{cases}
\end{equation}
We will refer to a matrix $\LatticeWeylMatrix_{m,n} (\ZakPosition,\ZakFrequency)$ as defined in \eqref{eq:Lmn-factorization} a \emph{lattice Weyl matrix}. 

In what follows, we will prove various basic properties of lattice Weyl matrices. We start by proving the aforementioned relation with the Weyl shifts under the vector Zak transform in the following lemma.

\begin{lemma}\label{lem:vector_Zak-matrix_multiplication}
    For every $m,n\in\Z$ and $F \in\mathscr H_\Zak$,
    \begin{equation}
      \left(
        \VectorZak\,\WeylShift\left(m,\frac n2\right)\VectorZak^{-1}
      \right) F (\ZakPoint)
      =
      \LatticeWeylMatrix_{m,n}(\ZakPoint)F (\ZakPoint)
      \qquad\text{for almost every }\ZakPoint\in\R^2.
    \end{equation}
    In addition, if $F$ is continuous, then the identity holds for every $z \in \R^2$.
\end{lemma}

\begin{proof}
    As $F \in \mathscr{H}_\Zak$, we have $F = \VectorZak f$ for some $f \in L^2(\R)$. We compute
    \begin{align}
        \left(
            \VectorZak \, \WeylShift(m,n/2) \, \VectorZak^{-1}
        \right) \,
        \VectorZak f(z)
        & = \VectorZak (\WeylShift(m,n/2) f)(z)\\
        & = 2^{-1/2} \left(
            \e^{2 \pi \ii \left( \TranslationParameter n/2 - \frac{m n/2}{2}\right)} 
            \Zak f\left(\TranslationParameter-m,\frac{\ModulationParameter-n+r}{2}\right)
        \right)_{r=0,1},
    \end{align}
    by \eqref{eq:zak-covariance}. Next, we use the quasiperiodicity \eqref{eq:Zak_quasiperiodic} of $\Zak$, and the fact that $m \in \Z$ to obtain
    \begin{align}
        \left(
            \VectorZak \, \WeylShift(m,n/2) \, \VectorZak^{-1}
        \right) \,
        \VectorZak f(z)
        & = 2^{-1/2} \left(
            \e^{\pi \ii \left(n \left( \TranslationParameter - \frac{m}{2}\right) - m(\ModulationParameter+r-n) \right)}
            \Zak f\left(\TranslationParameter,\frac{\ModulationParameter+(r-n)}{2}\right)
        \right)_{r=0,1}\\
        & = 2^{-1/2} \left(
            \e^{\pi \ii \left(n \TranslationParameter + \frac{m n}{2} - m \ModulationParameter -m r \right)}
            \Zak f\left(\TranslationParameter,\frac{\ModulationParameter+(r-n)}{2}\right)
        \right)_{r=0,1}
    \end{align}
    We distinguish two cases. If $n\in2\Z$, then periodicity in the second Zak variable gives
    \[
      \Zak f\left(
        \TranslationParameter,
        \frac{\ModulationParameter+(r-n)}2
      \right)
      =
      \Zak f\left(
        \TranslationParameter,
        \frac{\ModulationParameter+r}2
      \right).
    \]
    Factoring the part of the phase that is independent of $r$, therefore
    yields
    \begin{equation}
      \left(
        \VectorZak\,\WeylShift\left(m,\frac n2\right)\VectorZak^{-1}
      \right)F(\ZakPoint)
      =
      \e^{\pi\ii\left(
        n\ZakPosition-m\ZakFrequency+\frac{mn}{2}
      \right)}
      \begin{pmatrix}
        1&0\\
        0&\e^{-\pi\ii m}
      \end{pmatrix}
      F (\ZakPoint).
    \end{equation}

    If $n\in2\Z+1$, periodicity instead gives
    \[
      \Zak f\left(
        \TranslationParameter,
        \frac{\ModulationParameter+(r-n)}2
      \right)
      =
      \begin{cases}
        \Zak f\left(
          \TranslationParameter,
          \dfrac{\ModulationParameter+1}{2}
        \right),&r=0,\\[2mm]
        \Zak f\left(
          \TranslationParameter,
          \dfrac{\ModulationParameter}{2}
        \right),&r=1.
      \end{cases}
    \]
    Consequently, the two components are interchanged and
    \begin{equation}
      \left(
        \VectorZak\,\WeylShift\left(m,\frac n2\right)\VectorZak^{-1}
      \right)F (\ZakPoint)
      =
      \e^{\pi\ii\left(
        n\ZakPosition-m\ZakFrequency+\frac{mn}{2}
      \right)}
      \begin{pmatrix}
        0&1\\
        \e^{-\pi\ii m}&0
      \end{pmatrix}
      F (\ZakPoint).
    \end{equation}
    In both cases, this is precisely
    $\LatticeWeylMatrix_{m,n}(\ZakPoint)F (\ZakPoint)$ by
    \eqref{eq:Lmn-factorization} and \eqref{eq:Lmn-constant}.
\end{proof}

We next show that the lattice Weyl matrices satisfy the quasi-periodicity relations \eqref{eq:sewing}. 
\begin{lemma} \label{lem:sewing_Weyl_matrix}
    For $(m, n) \in \Z^2$, the lattice Weyl matrix $\LatticeWeylMatrix_{m,n}$
    satisfies the quasi-periodicity relations \eqref{eq:sewing},
    i.e., 
    \begin{align}
        \LatticeWeylMatrix_{m,n}(\ZakPosition+1,\ZakFrequency)
        &=
        \FirstSewingMatrix(\ZakFrequency)
        \LatticeWeylMatrix_{m,n}(\ZakPosition,\ZakFrequency)
        \FirstSewingMatrix(\ZakFrequency)^{-1},
        \label{eq:Lmn-x-sewing}\\
        \LatticeWeylMatrix_{m,n}(\ZakPosition,\ZakFrequency+1)
        &=
        \SecondSewingMatrix
        \LatticeWeylMatrix_{m,n}(\ZakPosition,\ZakFrequency)
        \SecondSewingMatrix^{-1}
        \label{eq:Lmn-omega-sewing}
    \end{align}
    for $(\ZakPosition, \ZakFrequency) \in \R^2$.
\end{lemma}

\begin{proof}
To verify the first identity, note that \eqref{eq:Lmn-factorization} gives
\begin{equation}
 \LatticeWeylMatrix_{m,n}(\ZakPosition+1,\ZakFrequency)
 =
 \e^{\pi\ii n}
 \LatticeWeylMatrix_{m,n}(\ZakPosition,\ZakFrequency)
 =
 (-1)^n\LatticeWeylMatrix_{m,n}(\ZakPosition,\ZakFrequency).
 \label{eq:Lmn-x-sewing-partial}
\end{equation}
On the other hand, put
\[
 D:=\begin{pmatrix}1&0\\0&-1\end{pmatrix},
 \quad \text{so that} \quad
 \FirstSewingMatrix(\ZakFrequency) = \e^{\pi\ii\ZakFrequency}D.
\]
The scalar factor $\e^{\pi\ii\ZakFrequency}$ in $\FirstSewingMatrix(\ZakFrequency)$ cancels under conjugation and we have $D^\ast = D = D^{-1}$. Moreover, the two cases in
\eqref{eq:Lmn-constant} give
\[
 D\LatticeWeylMatrix_{m,n}^{(0)}D
 =
 \begin{cases}
  \LatticeWeylMatrix_{m,n}^{(0)},&n\equiv0\pmod2,\\
  -\LatticeWeylMatrix_{m,n}^{(0)},&n\equiv1\pmod2,
 \end{cases}
 =
 (-1)^n\LatticeWeylMatrix_{m,n}^{(0)}.
\] 
Combining this identity with \eqref{eq:Lmn-factorization} and \eqref{eq:Lmn-x-sewing-partial},
we obtain
\begin{align*}
 \FirstSewingMatrix(\ZakFrequency)
 \LatticeWeylMatrix_{m,n}(\ZakPosition,\ZakFrequency)
 \FirstSewingMatrix(\ZakFrequency)^{-1}
 &=
 \e^{\pi\ii\left(
   n\ZakPosition-m\ZakFrequency+\frac{mn}{2}
 \right)}
 D\LatticeWeylMatrix_{m,n}^{(0)}D\\
 & =
 (-1)^n
 \e^{\pi\ii\left(
   n\ZakPosition-m\ZakFrequency+\frac{mn}{2}
 \right)}
 \LatticeWeylMatrix_{m,n}^{(0)}\\
 &=
 \LatticeWeylMatrix_{m,n}(\ZakPosition+1,\ZakFrequency).
\end{align*}
This proves \eqref{eq:Lmn-x-sewing}.

For the second identity, \eqref{eq:Lmn-factorization} similarly yields
\begin{equation}
    \LatticeWeylMatrix_{m,n}(\ZakPosition,\ZakFrequency+1)
    =
    \e^{-\pi\ii m}
    \LatticeWeylMatrix_{m,n}(\ZakPosition,\ZakFrequency)
    \label{eq:Lmn-omega-sewing-partial}
\end{equation}
Noting that $(\e^{-\pi\ii m})^2 = \e^{-2\pi\ii m} = 1$, $m \in \Z$, direct multiplication gives, when $n$ is even,
\[
    \SecondSewingMatrix \LatticeWeylMatrix_{m,n}^{(0)} \SecondSewingMatrix^{-1}
    = \SecondSewingMatrix\begin{pmatrix}1&0\\0&\e^{-\pi\ii m}\end{pmatrix} \SecondSewingMatrix^{-1}
    = \begin{pmatrix}\e^{-\pi\ii m}&0\\0&1\end{pmatrix}
    = \e^{-\pi\ii m}\begin{pmatrix}1&0\\0&\e^{-\pi\ii m}\end{pmatrix}
    = \e^{-\pi\ii m} \, \LatticeWeylMatrix_{m,n}^{(0)},
\]
whereas, when $n$ is odd,
\[
    \SecondSewingMatrix \LatticeWeylMatrix_{m,n}^{(0)} \SecondSewingMatrix^{-1}
    = \SecondSewingMatrix \begin{pmatrix} 0&1 \\ \e^{-\pi\ii m}&0\end{pmatrix} \SecondSewingMatrix^{-1}
    = \begin{pmatrix}0&\e^{-\pi\ii m}\\ 1&0\end{pmatrix}
    = \e^{-\pi\ii m}\begin{pmatrix}0&1\\ \e^{-\pi\ii m}&0\end{pmatrix}
    = \e^{-\pi\ii m} \, \LatticeWeylMatrix_{m,n}^{(0)}.
\]
Thus, in both cases, 
$
    \SecondSewingMatrix \LatticeWeylMatrix_{m,n}^{(0)} \SecondSewingMatrix^{-1}
    = \e^{-\pi\ii m}\LatticeWeylMatrix_{m,n}^{(0)}.
$
Using \eqref{eq:Lmn-factorization} once more, it follows that
\begin{align*}
    \SecondSewingMatrix
    \LatticeWeylMatrix_{m,n}(\ZakPosition,\ZakFrequency)
    \SecondSewingMatrix^{-1}
    &=
    \e^{\pi\ii\left(
    n\ZakPosition-m\ZakFrequency+\frac{mn}{2}
    \right)}
    \SecondSewingMatrix \LatticeWeylMatrix_{m,n}^{(0)} \SecondSewingMatrix^{-1}\\
    & =
    \e^{\pi\ii\left(
    n\ZakPosition-m\ZakFrequency+\frac{mn}{2}
    \right)}
    \e^{-\pi\ii m}
    \LatticeWeylMatrix_{m,n}^{(0)}\\
    &=
    \LatticeWeylMatrix_{m,n}(\ZakPosition,\ZakFrequency+1),
\end{align*}
where we used \eqref{eq:Lmn-omega-sewing-partial}. This proves \eqref{eq:Lmn-omega-sewing}. 
\end{proof}

Note that the matrix functions
$\LatticeWeylMatrix_{m,n} : \R^2 \to \C^{2 \times 2}$ defined by the lattice Weyl matrices are smooth.
The following proposition provides an explicit bound for the first derivatives
of these matrix functions.
Like some of our prior results, the following result is only used in our proof of the main Theorem~\ref{thm:main}
in Section~\ref{sec:certificate}.

\begin{proposition}
    For every $m,n\in\Z$ and $\ZakPoint\in\R^2$, the matrix
    $\LatticeWeylMatrix_{m,n}(\ZakPoint)\in\C^{2\times2}$ is unitary, and we have
    \begin{equation}\label{eq:L-derivatives}
        \norm{
        \partial_{\ZakPosition}
        \LatticeWeylMatrix_{m,n}(\ZakPosition, \ZakFrequency)
        }_{\op}
        = \pi|n|,
        \qquad
        \norm{\partial_\ZakFrequency \LatticeWeylMatrix_{m,n}(\ZakPosition, \ZakFrequency)}_{\op} = \pi|m|.
    \end{equation}
\end{proposition}

\begin{proof}
    The constant matrix $\LatticeWeylMatrix_{m,n}^{(0)}$ in
    \eqref{eq:Lmn-constant} is unitary, since $|\e^{-\pi\ii m}|=1$. 
    The scalar factor in \eqref{eq:Lmn-factorization} also has modulus one.
    Hence $\LatticeWeylMatrix_{m,n}(\ZakPoint)$ is unitary. 
    Differentiating \eqref{eq:Lmn-factorization} gives
    \begin{equation}
      \partial_{\ZakPosition}
      \LatticeWeylMatrix_{m,n}(\ZakPosition,\ZakFrequency)
      =
      \pi\ii n\,
      \LatticeWeylMatrix_{m,n}(\ZakPosition,\ZakFrequency),
      \qquad
      \partial_{\ZakFrequency}
      \LatticeWeylMatrix_{m,n}(\ZakPosition,\ZakFrequency)
      =
      -\pi\ii m\,
      \LatticeWeylMatrix_{m,n}(\ZakPosition,\ZakFrequency).
    \end{equation}
    Taking operator norms and using unitarity proves
    \eqref{eq:L-derivatives}.
\end{proof}

\subsection{Lattice Weyl expansions}
\label{sub:LatticeWeylDensity}

In this section, we show that any map $A\in C^\infty(\R^2;\C^{2\times2})$ satisfying the matrix quasi-periodicity relations \eqref{eq:sewing} can be written as a series of Weyl shifts. Specifically, we prove the following lemma.

\begin{lemma}\label{lem:fourier-density}
    Let $A \in C^\infty(\R^2, \C^{2 \times 2})$ satisfy the matrix quasi-periodicity relations \eqref{eq:sewing}. Then there exist coefficients $a_{m,n} \in \C$, $(m,n) \in \Z^2$, such that
    \begin{equation}\label{eq:fourier-weyl-expansion}
        A(\ZakPosition,\ZakFrequency)
        = \sum_{(m,n)\in\Z^2}
          a_{m,n} \,
          \LatticeWeylMatrix_{m,n}(\ZakPosition,\ZakFrequency),
        \qquad (\ZakPosition,\ZakFrequency) \in \R^2,
    \end{equation}
    with absolute and uniform convergence;
    in fact, for every $N \in \N$ there exists $K_N < \infty$ with
    \[
      |a_{m,n}|
      \leq K_N \cdot (1 + m^2 + n^2)^{-N}
      \qquad \text{for all } n,m \in \Z
      .
    \]
    
    In particular,
    for every $\eps>0$ there exists a finite
    set $J\subset\Z^2$ such that
    \begin{equation}\label{eq:finite-fourier-approximation}
        \sup_{\ZakPoint\in\R^2}
        \left\|A(\ZakPoint)-
          \sum_{(m,n)\in J}
            a_{m,n} \, \LatticeWeylMatrix_{m,n}(\ZakPoint)
        \right\|_{\op}<\eps.
    \end{equation}
\end{lemma}

\begin{proof} We split the proof into three steps.
\\~\\
\textbf{Step 1.}
As previously done, we set
\[
  \FirstSewingMatrix = \e^{\pi \ii \ZakFrequency} D,
  \quad D = \begin{pmatrix}1&0\\0&-1\end{pmatrix}
\]

For $a,b,c,d,u,v \in \C$, we have
\begin{equation}
  D
  \begin{pmatrix}
      a & b \\ c & d
  \end{pmatrix}
  D
  = \begin{pmatrix}
      a & -b \\
      -c & d
    \end{pmatrix}
  ,
  \;\;\;
  \SecondSewingMatrix
  \begin{pmatrix}
    0 & u \\
    v & 0
  \end{pmatrix}
  \SecondSewingMatrix
  = \begin{pmatrix}
      0 & v \\
      u & 0
    \end{pmatrix}
  ,
  \;\;\;
  \SecondSewingMatrix
  \begin{pmatrix}
    u & 0 \\
    0 & v
  \end{pmatrix}
  \SecondSewingMatrix
  = \begin{pmatrix}
      v & 0 \\
      0 & u
    \end{pmatrix}
  ,
  \label{eq:elementary-matrix-identities}
\end{equation}
where we note that $D^*=D=D^{-1}$ and $\SecondSewingMatrix^* = \SecondSewingMatrix = \SecondSewingMatrix^{-1}$.
\medskip

\noindent
\textbf{Step 2.}
As the exponential factors in $\FirstSewingMatrix$ cancel under matrix inversion, we have
\[
  \FirstSewingMatrix(\ZakFrequency) B \FirstSewingMatrix(\ZakFrequency)^{-1}
  = D B D
\]
for any $B \in \C^{2 \times 2}$.
Therefore, the quasi-periodicity condition \eqref{eq:sewing} for $A$ is equivalent to
\begin{equation}\label{eq:reduced-matrix-sewing}
    A(\ZakPosition+1,\ZakFrequency)
    =D A(\ZakPosition,\ZakFrequency) D,
    \qquad
    A(\ZakPosition,\ZakFrequency+1)
    =\SecondSewingMatrix A(\ZakPosition,\ZakFrequency) \SecondSewingMatrix,
    \qquad (\ZakPosition,\ZakFrequency) \in \R^2
    .
\end{equation}
Applying each of these identities twice shows that $A$ is $2 \Z$-periodic in each variable
and thus $(2 \Z)^2$-periodic.
Therefore, it has the matrix-valued Fourier expansion
\begin{equation}\label{eq:matrix-fourier-series}
    A(\ZakPosition,\ZakFrequency)
    =\sum_{(m,n)\in\Z^2}
       C_{m,n} \,
       \e^{\pi\ii(n\ZakPosition-m\ZakFrequency)},
\end{equation}
where we chose the sign convention in the exponent to be compatible with our Weyl matrix factorization \eqref{eq:Lmn-factorization}, and
\[
  C_{m,n}
  := \frac{1}{4}
     \int_{[0,2]^2}
       A(\ZakPosition,\ZakFrequency)
       \e^{-\pi\ii(n\ZakPosition-m\ZakFrequency)}
     \dd (\ZakPosition, \ZakFrequency)
  \in \C^{2 \times 2}
  .
\]

Since $A$ is $(2 \Z)^2$-periodic and smooth, it follows by standard results in Fourier analysis
(see, e.g., \cite[Corollary~3.3.10]{GrafakosClassicalFourier}) applied to each component of $A(\ZakPoint)$ that for every  $N\in\N$, there exists a constant $K_N < \infty$ such that
\begin{equation}
  \label{eq:rapid-fourier-decay}
  \|C_{m,n}\|_{\op}
  \leq K_N \cdot \bigl(1+m^2+n^2\bigr)^{-N}
  .
\end{equation}
It follows that the series \eqref{eq:matrix-fourier-series} converges absolutely and uniformly.

\medskip{}

\noindent
\textbf{Step 3.}
The first identity of \eqref{eq:reduced-matrix-sewing} gives
because of $(-1)^n = e^{\pi \ii n}$ that
\begin{equation}
\begin{aligned}
  (-1)^n C_{m,n}
  &= \frac{1}{4}
     \int_{[0,2]^2}
       A(\ZakPosition,\ZakFrequency)
       \, \e^{-\pi\ii(n(\ZakPosition-1) - m\ZakFrequency)}
     \dd (\ZakPosition, \ZakFrequency) \\
  &= \frac{1}{4}
     \int_{[0,2]^2}
       A(\ZakPosition+1,\ZakFrequency)
       \, \e^{-\pi\ii(n\ZakPosition - m\ZakFrequency)}
     \dd (\ZakPosition, \ZakFrequency) \\
  &= \frac{1}{4}
     \int_{[0,2]^2}
       D A(\ZakPosition,\ZakFrequency) D
       \, \e^{-\pi\ii(n\ZakPosition - m\ZakFrequency)}
     \dd (\ZakPosition, \ZakFrequency) \\
  &= D C_{m,n} D
  .
\end{aligned}
\label{eq:first-fourier-constraint}
\end{equation}
where the second step used that the integrand is $(2 \Z)^2$-periodic,
so that the integrals over any two fundamental domains agree.

Similarly, the second identity of \eqref{eq:reduced-matrix-sewing} gives
\begin{equation}\label{eq:second-fourier-constraint}
    (-1)^m C_{m,n}
    = \SecondSewingMatrix C_{m,n} \SecondSewingMatrix
    .
\end{equation}

Now, as in the definition of $\LatticeWeylMatrix_{m,n}^{(0)}$ in \eqref{eq:Lmn-constant}, we consider the parity of $n$.

If $n$ is even, then \eqref{eq:elementary-matrix-identities} and \eqref{eq:first-fourier-constraint} show that $C_{m,n}$ is diagonal,
say $C_{m,n} = \operatorname{diag}(u,v)$.
Then, Equations \eqref{eq:second-fourier-constraint} and \eqref{eq:elementary-matrix-identities} yield $v = (-1)^m u$.
Hence,
\begin{equation}\label{eq:even-coefficient-form}
    C_{m,n}
    = (C_{m,n})_{1,1}
    \cdot \begin{pmatrix}1&0\\0&(-1)^m\end{pmatrix}
    = (C_{m,n})_{1,1} \cdot \LatticeWeylMatrix_{m,n}^{(0)}
    \qquad (\text{if } n\ \text{is even}).
    \end{equation}
If $n$ is odd, \eqref{eq:elementary-matrix-identities} and \eqref{eq:first-fourier-constraint} show that $C_{m,n}$ is off-diagonal, say
\(
    C_{m,n} = 
    \left(
    \begin{smallmatrix}
        0 & u \\
        v & 0 
    \end{smallmatrix}
    \right)
\).
Then, equations \eqref{eq:second-fourier-constraint} and \eqref{eq:elementary-matrix-identities} yield $v=(-1)^m u$.
Therefore,
\begin{equation}\label{eq:odd-coefficient-form}
    C_{m,n}
    = (C_{m,n})_{1,2}  \cdot
    \begin{pmatrix}
        0&1\\
        (-1)^m&0
    \end{pmatrix} 
    = (C_{m,n})_{1,2} \cdot \LatticeWeylMatrix^{(0)}_{m,n}.
\end{equation}
Thus, setting 
\begin{equation}
    b_{m,n} := 
    \begin{cases}
        (C_{m,n})_{1,1},& \text{if } n\equiv0\pmod 2,\\
        (C_{m,n})_{1,2},& \text{if } n\equiv1\pmod 2, 
    \end{cases} 
    \qquad \text{and} \qquad
    a_{m,n} := b_{m,n} \, e^{-\pi \ii mn/2}, 
\end{equation}
and recalling from \eqref{eq:Lmn-factorization} that
\[
  \LatticeWeylMatrix_{m,n}(\ZakPosition,\ZakFrequency)
  = \e^{\pi\ii(n\ZakPosition-m\ZakFrequency+mn/2)}
    \cdot \LatticeWeylMatrix^{(0)}_{m,n}
  ,
\]
we obtain
\begin{equation}
\label{eq:Coefficient-Identification}
    C_{m,n} \cdot e^{\pi \ii (n \ZakPosition - m\ZakFrequency)} = a_{m,n} \cdot \LatticeWeylMatrix_{m,n}(\ZakPosition, \ZakFrequency),
    \qquad (\ZakPosition, \ZakFrequency) \in \R^2
    .
\end{equation}

Combining \eqref{eq:matrix-fourier-series} and \eqref{eq:Coefficient-Identification}, we obtain \eqref{eq:fourier-weyl-expansion}. The decay of the coefficients follows from
\[
  |a_{m,n}|
  \leq \| C_{m,n} \|_\op
  \leq K_N \cdot (1 + m^2 + n^2)^{-N}
  ,
\]
which proves absolute and uniform convergence of the series in \eqref{eq:fourier-weyl-expansion}. Finally, taking a sufficiently large finite set $J \subset \Z^2$, we get
\[
  \left\|
    A(\ZakPoint)
    - \sum_{(m,n) \in J}
        a_{m,n} \,
        \LatticeWeylMatrix_{m,n}(\ZakPosition,\ZakFrequency)
  \right\|_{\op}
  \leq \sum_{(m,n) \in \Z^2 \setminus J}
         |a_{m,n}| \, \| \LatticeWeylMatrix_{m,n} \|_{\op}
  \leq \sum_{(m,n) \in \Z^2 \setminus J}
         |a_{m,n}|
  < \eps
\]
for all $\ZakPoint \in \R^2$.
\end{proof}

\subsection{The irrational Weyl shift}

For the reader's convenience, we recall \eqref{eq:arith-data};
\begin{equation}
    \CubicIrrational = \sqrt[3]{2},\qquad
    \TranslationAlpha = \CubicIrrational-1,\qquad
    \TranslationBeta = \CubicIrrational^2-1,\qquad
    \WeylOffset
    =\left(\TranslationAlpha,\frac{\TranslationBeta}{2}\right),
    \qquad
    \TorusShiftVector=(\TranslationAlpha,\TranslationBeta).
\end{equation}
We first derive the vector Zak action of
$\WeylShift(\WeylOffset)$ explicitly. Applying
\eqref{eq:zak-covariance} with translation parameter
$\TranslationAlpha$ and modulation parameter
$\TranslationBeta/2$, and then using the definition of
$\VectorZak$, gives
\begin{align}
  \bigl(\VectorZak(\WeylShift(\WeylOffset) f)\bigr)_{r+1}
  (\ZakPosition,\ZakFrequency)
  &=
  \OffsetPhase(\ZakPosition)
  (\VectorZak f)_{r+1}(
    \ZakPosition-\TranslationAlpha,
    \ZakFrequency-\TranslationBeta),
  \qquad r\in\{0,1\},
  \label{eq:vector Zak-action-R-expanded}
\end{align}
where we define the phase factor $\eta$ to be
\begin{equation}
  \OffsetPhase(\ZakPosition)
  :=
  \e^{\pi\ii\TranslationBeta
      (\ZakPosition-\TranslationAlpha/2)}
  \in \C
  \quad \text{for} \quad
  \ZakPosition \in \R.
  \label{eq:EtaDefinition}
\end{equation}
We note that $\OffsetPhase : \R \to \C$ is smooth.

Consequently, for $F \in \mathscr{H}_\Zak$, we obtain (analogous to Lemma~\ref{lem:vector_Zak-matrix_multiplication})
\begin{equation} \label{eq:shift_action}
  (\VectorZak \, \WeylShift(\WeylOffset) \, \VectorZak^{-1})
  F (\ZakPosition,\ZakFrequency)
  =
  \OffsetPhase(\ZakPosition)
  F (\ZakPosition-\TranslationAlpha,
       \ZakFrequency-\TranslationBeta)
  =
  \OffsetPhase(\ZakPosition) F (\TranslateZak\ZakPoint)
\end{equation}
for a.e. $\ZakPoint \in \R^2$, where $\TranslateZak$ is defined as in \eqref{eq:shift}.
Define
\begin{align}
    \label{eq:def_w}
  \TranslatedReferenceSection(\ZakPoint)
  :=
  \OffsetPhase(\ZakPosition)
  \ReferenceSection(\TranslateZak\ZakPoint)
  \in \C^2 ,
  \qquad \text{for} \quad
  \ZakPoint = (\ZakPosition, \ZakFrequency) \in \R^2
  .
\end{align}
We note that with $\OffsetPhase$ and $\ReferenceSection$ (see Lemma~\ref{lem:unit}),
also $\TranslatedReferenceSection : \R^2 \to \C^2$ is smooth.

We verify that $\TranslatedReferenceSection$ satisfies the same vector Zak quasi-periodicity relations as $\ReferenceSection$. 

\begin{lemma} \label{lem:w_vectorZak}
The function $\TranslatedReferenceSection : \R^2 \to \C^2$ is a vector Zak function. Moreover, we have that $\| \TranslatedReferenceSection  (\ZakPoint) \| =1$ for all $\ZakPoint \in \R^2$.
\end{lemma}
\begin{proof}
First, note that $
  \OffsetPhase(\ZakPosition+1)
  =\e^{\pi\ii\TranslationBeta}\OffsetPhase(\ZakPosition).
$
Using the quasi-periodicity relation \eqref{eq:x-sewing} for $\ReferenceSection$, we obtain
\begin{align*}
    \TranslatedReferenceSection(\ZakPosition+1,\ZakFrequency)
    &=
    \OffsetPhase(\ZakPosition+1)
    \ReferenceSection(
    \ZakPosition+1-\TranslationAlpha,
    \ZakFrequency-\TranslationBeta)\\
    &=
    \e^{\pi\ii\TranslationBeta}\OffsetPhase(\ZakPosition)
    \FirstSewingMatrix(\ZakFrequency-\TranslationBeta)
    \ReferenceSection(
    \ZakPosition-\TranslationAlpha,
    \ZakFrequency-\TranslationBeta).
\end{align*}
Since
\begin{equation}
    \e^{\pi\ii\TranslationBeta}
    \FirstSewingMatrix(\ZakFrequency-\TranslationBeta)
    =
    \e^{\pi\ii\TranslationBeta}
    \e^{\pi\ii(\ZakFrequency-\TranslationBeta)}
    \begin{pmatrix}
    1&0\\
    0&-1
    \end{pmatrix}
    =
    \FirstSewingMatrix(\ZakFrequency),
\end{equation}
it follows that
\begin{equation}
    \TranslatedReferenceSection(\ZakPosition+1,\ZakFrequency)
    =
    \FirstSewingMatrix(\ZakFrequency)
    \TranslatedReferenceSection(\ZakPosition,\ZakFrequency).
    \label{eq:w-x-sewing}
\end{equation}
Similarly, the quasi-periodicity relation \eqref{eq:omega-sewing} gives
\begin{equation}
\begin{aligned}
  \TranslatedReferenceSection(\ZakPosition,\ZakFrequency+1)
  &=
  \OffsetPhase(\ZakPosition)
  \ReferenceSection(
    \ZakPosition-\TranslationAlpha,
    \ZakFrequency+1-\TranslationBeta)\\
  &=
  \OffsetPhase(\ZakPosition)
  \SecondSewingMatrix
  \ReferenceSection(
    \ZakPosition-\TranslationAlpha,
    \ZakFrequency-\TranslationBeta)\\
  &=
  \SecondSewingMatrix
  \TranslatedReferenceSection(\ZakPosition,\ZakFrequency).
\end{aligned}
\label{eq:w-omega-sewing}
\end{equation}
Thus $\TranslatedReferenceSection$ is a vector Zak function. Moreover,
since $\lvert\OffsetPhase(\ZakPosition)\rvert=1$ and
$\|\ReferenceSection(\TranslateZak\ZakPoint)\|=1$, one has that
$
  \|\TranslatedReferenceSection(\ZakPoint)\|=1,
  \ \ZakPoint\in\R^2,
$
as claimed.
\end{proof}

Now, we define the rank-one matrix field by
\begin{equation}\label{eq:A0Definition}
    \IdealMatrixField(\ZakPoint)
    :=
    \ReferenceSection(\ZakPoint)
    \TranslatedReferenceSection(\ZakPoint)^*
    \in \C^{2 \times 2},
    \qquad \ZakPoint \in \R^2
    .
\end{equation}
We note that with $\ReferenceSection$ (see Lemma~\ref{lem:unit}) and $w$ (see \eqref{eq:def_w}) being smooth, also $\IdealMatrixField : \R^2 \to \C^{2 \times 2}$ is smooth. Moreover, because $\ReferenceSection$ and $\TranslatedReferenceSection$ obey \eqref{eq:x-sewing} and \eqref{eq:omega-sewing}, $\IdealMatrixField$ satisfies the matrix quasi-periodicity relations \eqref{eq:sewing}. In addition,
\begin{equation}
    \IdealMatrixField(\ZakPoint)
    \TranslatedReferenceSection(\ZakPoint)
    =
    \ReferenceSection(\ZakPoint)
    \TranslatedReferenceSection(\ZakPoint)^*
    \TranslatedReferenceSection(\ZakPoint)
    =
    \ReferenceSection(\ZakPoint).
\end{equation}
Thus, multiplication by $\IdealMatrixField(\ZakPoint)$ exactly corrects the image under $\WeylShift(\WeylOffset)$ of the distinguished vector $\ReferenceSection(\TranslateZak\ZakPoint)$. Consequently, for every vector Zak function $F \in\mathscr H_\Zak$, applying $\WeylShift(\WeylOffset)$ first and then multiplying by $\IdealMatrixField(\ZakPoint)$ gives
\[
  \ZakPoint\longmapsto
  \IdealMatrixField(\ZakPoint)\OffsetPhase(\ZakPosition)
  F (\TranslateZak\ZakPoint).
\]
Since 
\(
    \TranslatedReferenceSection(\ZakPoint)^*
    =
    \overline{\OffsetPhase(\ZakPosition)}
    \,\ReferenceSection(\TranslateZak\ZakPoint)^*,
\)
the phase $\eta(x)$ cancels:
\begin{align*}
    \IdealMatrixField(\ZakPoint)\OffsetPhase(\ZakPosition)I_2
    =
    \ReferenceSection(\ZakPoint)
    \TranslatedReferenceSection(\ZakPoint)^*
    \OffsetPhase(\ZakPosition) =
    \ReferenceSection(\ZakPoint)
    \overline{\OffsetPhase(\ZakPosition)}
    \ReferenceSection(\TranslateZak\ZakPoint)^*
    \OffsetPhase(\ZakPosition) =
    \ReferenceSection(\ZakPoint)
    \ReferenceSection(\TranslateZak\ZakPoint)^*.
\end{align*}
We therefore set
\begin{equation}
  \IdealCocycle(\ZakPoint)
  :=
  \ReferenceSection(\ZakPoint)
  \ReferenceSection(\TranslateZak\ZakPoint)^*
  \in \C^{2 \times 2},
  \quad \text{for} \quad \ZakPoint \in \R^2.
  \label{eq:B0Definition}
\end{equation}
Again, by smoothness of $\ReferenceSection$ (see Lemma~\ref{lem:unit}),
it follows that $\IdealCocycle : \R^2 \to \C^{2 \times 2}$ is smooth.
Moreover, the unit identity for $\ReferenceSection$ then yields the exact relation
\begin{equation}
  \IdealCocycle(\ZakPoint)
  \ReferenceSection(\TranslateZak\ZakPoint)
  =
  \ReferenceSection(\ZakPoint)
  \ReferenceSection(\TranslateZak\ZakPoint)^*
  \ReferenceSection(\TranslateZak\ZakPoint)
  =
  \ReferenceSection(\ZakPoint)
\end{equation}
for $z \in \R^2$.

\subsection{Admissible coefficient sequences}
\label{sec:AdmissibleSequences}
In the remainder, we will use sequences that allow us to adequately approximate the rank-one matrix function $\IdealCocycle$ by a sum of lattice Weyl matrices. For this, we make the following definition.

\begin{definition}\label{def:AdmissibleSequence}
Given a finitely supported coefficient sequence $\aaa = (\aaa_{m,n})_{(m,n) \in \Z^2} \in \C^{\Z \times \Z}$, define
\[
  A[\aaa], B[\aaa] : \quad \R^2 \to \C^{2 \times 2}
\]
by
\[
  A[\aaa] (\ZakPoint)
  := \sum_{(m,n) \in \Z^2}
       \aaa_{m,n}
       \LatticeWeylMatrix_{m,n}(\ZakPoint)
  \qquad \text{and} \qquad
  B[\aaa] (\ZakPosition, \ZakFrequency) := \OffsetPhase(\ZakPosition) \, A[\aaa](\ZakPoint)
  ,
\]
with $\OffsetPhase(\ZakPosition)$ defined as in \eqref{eq:EtaDefinition}. We say the coefficient sequence $\aaa$ is \emph{admissible} if
\[
  \CertifiedError
  := \CertifiedError (\aaa)
  := \sup_{\ZakPoint \in \R^2}
     \|
       B[\aaa](\ZakPoint) - \IdealCocycle (\ZakPoint)
     \|_{\op}
  < \frac{1}{3}
  ,
\]
with $\IdealCocycle$ as defined in \eqref{eq:B0Definition}.
\end{definition}

\begin{remark}\label{rem:AdmissibleSequencesExist}
    Since $B[\aaa](\ZakPosition, \ZakFrequency) = \OffsetPhase(\ZakPosition) \, A[\aaa](\ZakPoint)$
    and\footnote{
        This follows from \eqref{eq:def_w}, \eqref{eq:A0Definition}, and \eqref{eq:B0Definition}, which show for $\ZakPoint = (\ZakPosition, \ZakFrequency)$ that
        \(
            \IdealMatrixField (\ZakPoint)
            = \ReferenceSection(\ZakPoint) \TranslatedReferenceSection (\ZakPoint)^\ast
            = \ReferenceSection(\ZakPoint) (\OffsetPhase(\ZakPosition) \ReferenceSection (\TranslateZak \ZakPoint))^\ast
            = \ReferenceSection(\ZakPoint) \ReferenceSection(\TranslateZak \ZakPoint)^\ast \overline{\OffsetPhase(\ZakPosition)}
            = \overline{\OffsetPhase(\ZakPosition)}  \IdealCocycle(\ZakPoint)
        \).
    }
    $\IdealCocycle (\ZakPosition, \ZakFrequency) = \OffsetPhase(\ZakPosition)\, \IdealMatrixField (\ZakPoint)$, with $|\OffsetPhase(\ZakPosition)| = 1$, we have
    \begin{equation}
        \|
          B[\aaa](\ZakPoint) - \IdealCocycle (\ZakPoint)
        \|_{\op}
        =   \|
              A[\aaa](\ZakPoint) - \IdealMatrixField (\ZakPoint)
            \|_{\op}
        .
        \label{eq:delta-equality-pointwise}
    \end{equation}
    Therefore,
    \begin{equation}
        \CertifiedError(\aaa)
        = \sup_{\ZakPoint \in \R^2}
           \|
             A[\aaa](\ZakPoint) - \IdealMatrixField (\ZakPoint)
           \|_{\op}
        .
        \label{eq:delta-equality}
    \end{equation}
    Since $\IdealMatrixField : \R^2 \to \C^{2 \times 2}$ is smooth and satisfies the quasi-periodicity relations \eqref{eq:sewing}, Lemma~\ref{lem:fourier-density} implies the existence of a finitely supported admissible sequence $\aaa \in \C^{\Z \times \Z}$. We will furthermore show in Section~\ref{sec:ExplicitAdmissibleSequence} that, in fact, one can choose an explicit admissible sequence with a support containing $11$ elements; the coefficients are given by \eqref{eq:amn-dyadic} and Table~\ref{tab:coefficients}.
\end{remark}

For the rest of the paper, we assume that $\aaa \in \C^{\Z \times \Z}$ is a fixed, finitely supported admissible sequence.
Using this fixed sequence, define $\FixedMatrixField : \R^2 \to \C^{2 \times 2}$ and $\FixedLatticeOperator : L^2(\R) \to L^2(\R)$ as
\begin{equation}\label{eq:Astar}
 \FixedMatrixField(\ZakPoint)
 := A[\aaa](\ZakPoint)
 =
 \sum_{(m,n) \in \Z^2}
   \FixedWeylCoefficient_{m,n}
   \LatticeWeylMatrix_{m,n}(\ZakPoint)
 \in \C^{2 \times 2},
 \qquad
 \FixedLatticeOperator
 := \sum_{(m,n)\in \Z^2}
      \FixedWeylCoefficient_{m,n}
      \WeylShift\!\left(m,\frac{n}{2}\right),
\end{equation}
and define $\FixedCocycle : \R^2 \to \C^{2 \times 2}$ as
\begin{equation}
  \label{eq:Bstar}
  \FixedCocycle(\ZakPoint)
  := B[\aaa](\ZakPoint)
  = \FixedMatrixField(\ZakPoint)\OffsetPhase(\ZakPosition)
  \in \C^{2 \times 2}
  \quad \text{for} \quad
  \ZakPoint = (\ZakPosition, \ZakFrequency) \in \R^2
  .
\end{equation}
Since each matrix function $\LatticeWeylMatrix_{m,n} : \R^2 \to \C^{2 \times 2}$ is smooth and the coefficient sequence $\aaa$ is finitely supported, $\FixedMatrixField : \R^2 \to \C^{2 \times 2}$ is smooth as well. Since $\OffsetPhase(\ZakPoint): \R \to \C$ is also smooth, it follows that $\FixedCocycle$ is smooth.

In terms of the operator $\FixedLatticeOperator$ defined above, we define the fixed finite Weyl operator
\begin{equation}\label{eq:Tstar}
    \FixedWeylOperator :
    L^2(\R) \to L^2(\R),
    \quad
    \FixedWeylOperator := \FixedLatticeOperator\WeylShift(\WeylOffset).
\end{equation}
Its vector Zak action on $F \in \mathscr H_\Zak$ is, for a.e. $\ZakPoint \in \R^2$,
\begin{equation}\label{eq:Bstar-Zak-Action}
 (\VectorZak\FixedWeylOperator\VectorZak^{-1})F(\ZakPoint)
 =\FixedCocycle(\ZakPoint)F(\TranslateZak\ZakPoint),
\end{equation}
cf.\ Lemma~\ref{lem:vector_Zak-matrix_multiplication} and Equation~\eqref{eq:shift_action}.

\section{Existence of a fixed point vector Zak function}
\label{sec:graph}

Let $\CalX$ be the Banach space of continuous vector Zak functions $\CorrectionField : \R^2 \to \C^2$ (i.e., satisfying the quasi-periodicity relations \eqref{eq:x-sewing} and \eqref{eq:omega-sewing}) that satisfy the additional orthogonality condition 
\begin{equation}\label{eq:x-orthogonal}
    \ReferenceSection(\ZakPoint)^*\CorrectionField(\ZakPoint)=0\quad\text{for every }\ZakPoint \in \R^2,
\end{equation}
where $\ReferenceSection = \VectorZak \ReferenceWindow$ is the vector Zak function defined in \eqref{eq:explicit_vector_zak}. We define the norm on $\CalX$ as
\[
  \norm{\CorrectionField}_\infty
  := \sup_{\ZakPoint\in[0,1]^2}\norm{\CorrectionField(\ZakPoint)}
  = \sup_{\ZakPoint \in \R^2}\norm{\CorrectionField(\ZakPoint)}
\]
with equality following from the quasi-periodicity relations
\eqref{eq:x-sewing} and \eqref{eq:omega-sewing}.
Define the ball
\begin{equation}
 \Ball := \left\{\CorrectionField\in\mathcal X:
 \norm{\CorrectionField}_\infty\le\frac{3}{4}\right\}.
\end{equation}

Recall from Section~\ref{sec:AdmissibleSequences} that we fixed a finitely supported, admissible coefficient sequence $\aaa$, and used it to define the matrix-valued functions $\FixedMatrixField,\FixedCocycle : \R^2 \to \C^{2 \times 2}$ and the operators $\FixedWeylOperator, \FixedLatticeOperator : L^2(\R) \to L^2(\R)$. Now, for $\CorrectionField \in \Ball$ and $\ZakPoint \in \R^2$, put 
\begin{equation}
  \CandidateInvariantVector(\ZakPoint)
  := \ReferenceSection(\ZakPoint)+\CorrectionField(\ZakPoint) \in \C^2
  \label{eq:vXDefinition}
\end{equation}
and
\begin{equation}\label{eq:denom}
    \NormalizationScalar{\CorrectionField}(\ZakPoint)
    := \ReferenceSection(\ZakPoint)^*
    \FixedCocycle(\ZakPoint)
    \CandidateInvariantVector(\TranslateZak \ZakPoint)
    \in \C,
    \qquad z \in \R^2,
\end{equation}
where $B_*$ is the matrix field defined in Section~\ref{sec:AdmissibleSequences}.
Note that $\NormalizationScalar{\CorrectionField}$ thus depends on the admissible sequence $\aaa$.

We have the following lemma.

\begin{lemma} \label{lem:q-positive}
   For every $\CorrectionField\in\Ball$, we have that
$\NormalizationScalar{\CorrectionField}(\ZakPoint)\neq 0$ for all $z \in \R^2$.
\end{lemma}

\begin{proof}
    Write
    \begin{equation}
        \CocycleErrorField(\ZakPoint)
        :=\FixedCocycle(\ZakPoint)-\IdealCocycle(\ZakPoint)
        \in \C^{2 \times 2},
        \qquad \ZakPoint \in \R^2
    \end{equation}
    and recall from \eqref{eq:B0Definition} that
    \begin{equation}
        \IdealCocycle(\ZakPoint)
        =
        \ReferenceSection(\ZakPoint)\ReferenceSection(\TranslateZak \ZakPoint)^*.
    \end{equation}
    Since \eqref{eq:x-orthogonal}
    implies $\CorrectionField(\TranslateZak \ZakPoint)\perp\ReferenceSection(\TranslateZak \ZakPoint)$
    while Lemma~\ref{lem:unit} implies
    $\ReferenceSection(\TranslateZak \ZakPoint)^*\ReferenceSection(\TranslateZak \ZakPoint)=1$, we have
    \begin{equation}
    \IdealCocycle(\ZakPoint)\CandidateInvariantVector(\TranslateZak \ZakPoint)
    =
    \ReferenceSection(\ZakPoint)\ReferenceSection(\TranslateZak \ZakPoint)^*
    \bigl(\ReferenceSection(\TranslateZak \ZakPoint)+\CorrectionField(\TranslateZak \ZakPoint)\bigr)
    =
    \ReferenceSection(\ZakPoint).
    \end{equation}
    Consequently,
    \begin{equation}
        \NormalizationScalar{\CorrectionField}(\ZakPoint)
        =
        \ReferenceSection(\ZakPoint)^*\FixedCocycle(\ZakPoint)\CandidateInvariantVector(\TranslateZak \ZakPoint)
        =
        1+\ReferenceSection(\ZakPoint)^*\CocycleErrorField(\ZakPoint)\CandidateInvariantVector(\TranslateZak \ZakPoint).
        \label{eq:qXRewrittenUsingErrorTerm}
    \end{equation}
    As $\CorrectionField \in \Ball$, one has $\|\CorrectionField\|_\infty\leq 3/4$. The orthogonality of $\ReferenceSection(\TranslateZak \ZakPoint)$ and $\CorrectionField(\TranslateZak \ZakPoint)$ (see \eqref{eq:x-orthogonal}) gives
    \begin{equation}
    \|\CandidateInvariantVector(\TranslateZak \ZakPoint)\|^2
    =
    \|\ReferenceSection(\TranslateZak \ZakPoint)\|^2+\|\CorrectionField(\TranslateZak \ZakPoint)\|^2
    \leq 1+\frac9{16}=\frac{25}{16} = \left( \frac{5}{4} \right)^2.
    \end{equation}
    Using \eqref{eq:qXRewrittenUsingErrorTerm}, we get
    \begin{equation}
    |\NormalizationScalar{\CorrectionField}(\ZakPoint)-1|
    \leq
    \|\ReferenceSection(\ZakPoint)\|\,\|\CocycleErrorField(\ZakPoint)\|_{\mathrm{op}}\,
    \|\CandidateInvariantVector(\TranslateZak \ZakPoint)\|
    \leq
    \frac{5}{4}\,\CertifiedError
    <
    \frac{5}{12},
    \end{equation}
    where $\CertifiedError := \sup_{\ZakPoint \in \R^2} \|\FixedCocycle(\ZakPoint) - \IdealCocycle(\ZakPoint)\|_{\op} < 1/3$;
    see Definition~\ref{def:AdmissibleSequence} and \eqref{eq:Bstar}.
    Hence, $|\NormalizationScalar{\CorrectionField}(\ZakPoint)|>\frac{7}{12}$ for all $\ZakPoint \in \R^2$, which settles the claim.
\end{proof}

We now define the following rank-one matrices:
\begin{equation}\label{eq:PQ}
     \ReferenceProjection(\ZakPoint)
     :=\ReferenceSection(\ZakPoint)\ReferenceSection(\ZakPoint)^*,
     \qquad
     \ComplementProjection(\ZakPoint)
     :=I_2 - \ReferenceProjection(\ZakPoint),
     \qquad
     \ZakPoint \in \R^2.
\end{equation}
Note that $\ReferenceProjection (\ZakPoint)$ is the orthogonal projection onto $\mathrm{span} \{\ReferenceSection(\ZakPoint)\}$ and $\ComplementProjection(\ZakPoint)$ is the orthogonal projection onto the orthogonal complement of this space. As $\ReferenceSection \in C^\infty (\R^2; \C^2)$, it follows that $\ReferenceProjection, \ComplementProjection \in C^\infty (\R^2; \C^{2\times 2})$. Also, recall that $\FixedCocycle \in C^\infty (\R^2 ; \C^{2 \times 2})$ (see \eqref{eq:Bstar} and the remark following it).

Using Lemma \ref{lem:q-positive}, given $\CorrectionField \in \Ball$, we can define
\begin{equation}\label{eq:Phi}
    (\CorrectionMap \CorrectionField)(\ZakPoint)
    :=
    \frac{
        \ComplementProjection (\ZakPoint) \FixedCocycle (\ZakPoint) \CandidateInvariantVector (\TranslateZak \ZakPoint)
    }
    {
        \NormalizationScalar{\CorrectionField}(\ZakPoint)
    }
    \in \C^2
    , \quad \ZakPoint \in \R^2 .
\end{equation}
As $\NormalizationScalar{\CorrectionField}(\ZakPoint) \neq 0$ this expression is well-defined. Since $\CorrectionField \in \CalX \subset C(\R^2; \C^2)$, it follows that $\CandidateInvariantVector (\ZakPoint) = \ReferenceSection (\ZakPoint) + \CorrectionField (\ZakPoint) \in C(\R^2; \C^2)$ and $\NormalizationScalar{\CorrectionField} \in C(\R^2; \C)$. Altogether, this implies $\CorrectionMap \CorrectionField \in C(\R^2; \C^2)$.

We next show that $\CorrectionMap \CorrectionField$ is a well-defined vector Zak function.

\begin{lemma} \label{lem:correction_zak}
    The function $\CorrectionMap \CorrectionField$ is a well-defined vector Zak function. In addition, $\CorrectionMap \CorrectionField \in \CalX$.
\end{lemma}
\begin{proof}
    We need to show that $\CorrectionMap \CorrectionField$ satisfies the vector Zak quasi-periodicity relations \eqref{eq:x-sewing} and \eqref{eq:omega-sewing}. For this, let
    $
        e_1 := (1,0),
        \; 
        e_2 := (0,1),
    $
    and, for $j\in\{1,2\}$, let $\SewingMatrix_j(\ZakPoint)$ denote the unitary matrix in the corresponding quasi-periodicity relation
    $
        F (\ZakPoint+e_j)=\SewingMatrix_j(\ZakPoint)F (\ZakPoint)
   $
    for vector Zak functions $F$. Since
    $
        \TranslateZak(\ZakPoint+e_j)=\TranslateZak \ZakPoint+e_j,
    $
    and since both $\ReferenceSection = \VectorZak \ReferenceWindow$ and $\CorrectionField$ satisfy \eqref{eq:x-sewing} and \eqref{eq:omega-sewing}, the relevant quantities satisfy
    \begin{equation}
        \ReferenceSection(\ZakPoint+e_j)=\SewingMatrix_j(\ZakPoint)\ReferenceSection(\ZakPoint),
    \end{equation}
    \begin{equation}
    \CandidateInvariantVector(\TranslateZak (\ZakPoint+e_j))
    =\SewingMatrix_j(\TranslateZak \ZakPoint)\CandidateInvariantVector(\TranslateZak \ZakPoint),
    \end{equation}
    \begin{equation} \label{eq:sewingB_star}
        \FixedCocycle(\ZakPoint+e_j)
        =
        \SewingMatrix_j(\ZakPoint)\FixedCocycle(\ZakPoint)\SewingMatrix_j(\TranslateZak \ZakPoint)^{-1};
    \end{equation}
    see Appendix~\ref{sub:fixed-cocycle-sewing} for a verification of this last identity. Moreover,
    \begin{equation}
        \ComplementProjection(\ZakPoint+e_j)
        =
        \SewingMatrix_j(\ZakPoint)\ComplementProjection(\ZakPoint)\SewingMatrix_j(\ZakPoint)^{-1},
    \end{equation}
    where we recall that $\ComplementProjection$ was introduced in \eqref{eq:PQ}.
    
    Using these identities and the unitarity of $\SewingMatrix_j(\ZakPoint)$, we obtain
    \begin{equation}
    \begin{aligned}
    \NormalizationScalar{\CorrectionField}(\ZakPoint+e_j)
    &=
    \ReferenceSection(\ZakPoint+e_j)^*
    \FixedCocycle(\ZakPoint+e_j)
    \CandidateInvariantVector(\TranslateZak (\ZakPoint+e_j))\\
    &=
    \ReferenceSection(\ZakPoint)^*\SewingMatrix_j(\ZakPoint)^*
    \SewingMatrix_j(\ZakPoint)\FixedCocycle(\ZakPoint)\SewingMatrix_j(\TranslateZak \ZakPoint)^{-1}
    \SewingMatrix_j(\TranslateZak \ZakPoint)\CandidateInvariantVector(\TranslateZak \ZakPoint)\\
    &=
    \NormalizationScalar{\CorrectionField}(\ZakPoint),
    \end{aligned}
    \label{eq:qIsPeriodic}
    \end{equation}
    and
    \begin{equation}
    \begin{aligned}
& \ComplementProjection(\ZakPoint+e_j)\FixedCocycle(\ZakPoint+e_j)\CandidateInvariantVector(\TranslateZak (\ZakPoint+e_j)) \\
    & \qquad =   \SewingMatrix_j(\ZakPoint)\ComplementProjection(\ZakPoint)\SewingMatrix_j(\ZakPoint)^{-1}  \SewingMatrix_j(\ZakPoint)\FixedCocycle(\ZakPoint)\SewingMatrix_j(\TranslateZak \ZakPoint)^{-1}
    \SewingMatrix_j(\TranslateZak \ZakPoint)\CandidateInvariantVector(\TranslateZak \ZakPoint)\\
    &\qquad =
\SewingMatrix_j(\ZakPoint)\ComplementProjection(\ZakPoint)\FixedCocycle(\ZakPoint)\CandidateInvariantVector(\TranslateZak \ZakPoint).
    \end{aligned}
    \end{equation}
    Therefore,
    \begin{equation}
    (\CorrectionMap \CorrectionField)(\ZakPoint+e_j)
    =
    \SewingMatrix_j(\ZakPoint)(\CorrectionMap \CorrectionField)(\ZakPoint),
    \qquad j=1,2.
    \end{equation}
    Thus $\CorrectionMap \CorrectionField$ satisfies the vector Zak quasi-periodicity relations
    \eqref{eq:x-sewing} and \eqref{eq:omega-sewing} and hence is a vector Zak function.
    
    Lastly, note that because the
    numerator in the definition of $\CorrectionMap \CorrectionField$ is multiplied by $\ComplementProjection(\ZakPoint)$, we have that
    \begin{equation}
    \ReferenceSection(\ZakPoint)^*(\CorrectionMap \CorrectionField)(\ZakPoint)=0
    \end{equation}
    for every $\ZakPoint \in \R^2$, so that $(\CorrectionMap \CorrectionField)(\ZakPoint)\perp\ReferenceSection(\ZakPoint)$.
    Therefore, the quotient is a well-defined vector Zak function on all of $\R^2$,
    and thus $\CorrectionMap \CorrectionField \in \CalX$ for every $\CorrectionField \in \Ball$.
\end{proof}

\subsection{Existence of the fixed point}

We next show that the map $\CorrectionMap : \Ball \to \CalX$ defined in \eqref{eq:Phi} forms a contraction admitting a fixed point.

\begin{proposition} \label{prop:contraction}
The map $\CorrectionMap : \Ball \to \CalX$ defined in \eqref{eq:Phi} maps $\Ball$ into itself and is a contraction with
Lipschitz constant smaller than $48/49$.  It has a unique fixed point
$\FixedCorrection\in\Ball$, which satisfies
\begin{equation}\label{eq:xstar-small}
 \norm{\FixedCorrection}_\infty<\frac57.
\end{equation}
With
\begin{equation}\label{eq:vq}
 \FixedInvariantVector(\ZakPoint)
 := \ReferenceSection(\ZakPoint)+\FixedCorrection(\ZakPoint)
 \in \C^2,
 \qquad
 \InvariantMultiplier(\ZakPoint)
 := \NormalizationScalar{\FixedCorrection}(\ZakPoint) \in \C,
\end{equation}
one has
\begin{equation}\label{eq:line-invariance}
 \FixedCocycle(\ZakPoint)\FixedInvariantVector(\TranslateZak \ZakPoint)=\InvariantMultiplier(\ZakPoint)\FixedInvariantVector(\ZakPoint)
 \qquad(\ZakPoint\in\R^2),
\end{equation}
and
\begin{equation}\label{eq:q-disk}
 |\InvariantMultiplier(\ZakPoint)-1|<\frac5{12}
 \qquad(\ZakPoint\in\R^2).
\end{equation}
In particular, $\FixedInvariantVector$ never vanishes, and $\InvariantMultiplier$ never vanishes.
\end{proposition}

\begin{proof}
Write
\[
  \CocycleErrorField(\ZakPoint)
  := \FixedCocycle(\ZakPoint) - \IdealCocycle(\ZakPoint) \in \C^{2 \times 2} ,
  \qquad \ZakPoint \in \R^2
  .
\]
By Definition~\ref{def:AdmissibleSequence} and \eqref{eq:Bstar}, we have
$\norm{\CocycleErrorField (\ZakPoint)}_{\mathrm{op}} \leq \CertifiedError<1/3$ for all $\ZakPoint \in \R^2$.
We split the proof into three steps.
\medskip{}

\noindent
\textbf{Step 1.} In this step, we show that $\CorrectionMap \CorrectionField \in \Ball$ for $\CorrectionField \in \Ball$.
Since we have $\CorrectionField(\TranslateZak \ZakPoint)\perp\ReferenceSection(\TranslateZak \ZakPoint)$
by \eqref{eq:x-orthogonal}, and 
${\IdealCocycle(\ZakPoint) = \ReferenceSection(\ZakPoint) \ReferenceSection(\TranslateZak\ZakPoint)^*}$
by \eqref{eq:B0Definition},
as well as $\| \ReferenceSection(\TranslateZak \ZakPoint) \| = 1$ by Lemma~\ref{lem:unit}, we see
\begin{equation}
 \IdealCocycle(\ZakPoint)(\ReferenceSection(\TranslateZak \ZakPoint)
 +\CorrectionField(\TranslateZak \ZakPoint))
 =\ReferenceSection(\ZakPoint).
\end{equation}
Thus
\begin{equation}\label{eq:d-error}
 \NormalizationScalar{\CorrectionField}(\ZakPoint)=1+\ReferenceSection(\ZakPoint)^*\CocycleErrorField(\ZakPoint)\CandidateInvariantVector(\TranslateZak \ZakPoint),
\end{equation}
and
\begin{equation}\label{eq:numer-error}
 \ComplementProjection(\ZakPoint)\FixedCocycle(\ZakPoint)\CandidateInvariantVector(\TranslateZak \ZakPoint)
 =\ComplementProjection(\ZakPoint)\CocycleErrorField(\ZakPoint)\CandidateInvariantVector(\TranslateZak \ZakPoint).
\end{equation}
Here, we used that $\ComplementProjection (\ZakPoint) \IdealCocycle(\ZakPoint) = 0$
since
\(
  \mathrm{range}(\IdealCocycle(\ZakPoint))
  \subset \mathrm{span} \{ \ReferenceSection (\ZakPoint) \}
  \subset \mathrm{ker} ( \ComplementProjection (\ZakPoint))
  .
\)

Moreover, because $\CorrectionField(\ZakPoint)\perp\ReferenceSection(\ZakPoint)$ in $\C^2$,
\begin{equation}
 \norm{\CandidateInvariantVector(\ZakPoint)}^2
 =\norm{\ReferenceSection(\ZakPoint)}^2+\norm{\CorrectionField(\ZakPoint)}^2
 \le1+\frac9{16}=\frac{25}{16} = \left(\frac{5}{4}\right)^2.
 \label{eq:vXNormEstimate}
\end{equation}
Thus $\norm{\CandidateInvariantVector(\ZakPoint)}\le5/4$, and \eqref{eq:d-error} gives
\begin{equation}\label{eq:denom-lower}
 |\NormalizationScalar{\CorrectionField}(\ZakPoint)|\ge1-\frac54\CertifiedError>1-\frac5{12}=\frac7{12}.
\end{equation}
Equations \eqref{eq:numer-error} -- \eqref{eq:denom-lower} yield
\begin{equation}
 \norm{\CorrectionMap \CorrectionField}_\infty
 \le\frac{(5/4)\CertifiedError}{1-(5/4)\CertifiedError}
 <\frac{5/12}{7/12}=\frac57<\frac34.
\end{equation}
Therefore, $\CorrectionMap(\Ball)\subset\Ball$, which settles the claim.
\\~\\
\indent
Since $\FixedCorrection=\CorrectionMap \FixedCorrection$,
the same estimate will prove \eqref{eq:xstar-small},
once we establish the existence of the unique fixed point $\FixedCorrection \in \Ball$ of $\CorrectionMap$.
\\~\\
\textbf{Step 2.} We show that $\mathrm{Lip}(\CorrectionMap) \leq \frac{48}{49}$.
For $\CorrectionField,Y\in\Ball$, denote the numerator and
denominator in \eqref{eq:Phi} by
\begin{align*}
 N_{\CorrectionField}(\ZakPoint)
 &:=\ComplementProjection(\ZakPoint)\FixedCocycle(\ZakPoint)
   \CandidateInvariantVector(\TranslateZak\ZakPoint) \in \C^2 ,\\
 d_{\CorrectionField}(\ZakPoint)
 &:=\NormalizationScalar{\CorrectionField}(\ZakPoint) \in \C ,
\end{align*}
and define $N_Y,d_Y$ analogously.
Since
$\ComplementProjection(\ZakPoint)\IdealCocycle(\ZakPoint)=0$ and
$(\CorrectionField-Y)(\TranslateZak\ZakPoint)$ is perpendicular to
$\ReferenceSection(\TranslateZak\ZakPoint)$ by \eqref{eq:x-orthogonal},
we see by \eqref{eq:vXDefinition} and \eqref{eq:denom} that
\begin{align*}
 N_{\CorrectionField}(\ZakPoint)-N_Y(\ZakPoint)
 &=\ComplementProjection(\ZakPoint)\CocycleErrorField(\ZakPoint)
   (\CorrectionField-Y)(\TranslateZak\ZakPoint),\\
 d_{\CorrectionField}(\ZakPoint)-d_Y(\ZakPoint)
 &=\ReferenceSection(\ZakPoint)^*\CocycleErrorField(\ZakPoint)
   (\CorrectionField-Y)(\TranslateZak\ZakPoint).
\end{align*}
Thus, the numerator difference has norm at most
$\CertifiedError\norm{\CorrectionField-Y}_\infty$, and the
denominator difference has modulus at most the same quantity. 
We use the elementary identity
\begin{equation}\label{eq:quotient-difference}
 \frac{N_1}{d_1}-\frac{N_2}{d_2}
 =\frac{N_1-N_2}{d_1}
  +N_2\frac{d_2-d_1}{d_1d_2},
 \qquad
 \text{for } k \in \N, \,\,\ N_1, N_2 \in \C^k \text{ and } d_1,d_2 \in \C \setminus \{0\}
 .
\end{equation}
With $r_0 := 1-(5/4)\CertifiedError$, applying
\eqref{eq:quotient-difference} and \eqref{eq:vXNormEstimate} gives
\begin{equation}
 \norm{
   \CorrectionMap\CorrectionField-\CorrectionMap Y
 }_\infty
 \le\left(
   \frac{\CertifiedError}{r_0}
   +\frac{(5/4)\CertifiedError^2}{r_0^2}
 \right)
       \norm{\CorrectionField-Y}_\infty
 =\frac{\CertifiedError}{r_0^2}
  \norm{\CorrectionField-Y}_\infty.
  \label{eq:CorrectionMapLipschitz}
\end{equation}
Here, we used that
\[
  \| N_Y (\ZakPoint) \|
  = \| \ComplementProjection (\ZakPoint) \CocycleErrorField (\ZakPoint) v_Y (\TranslateZak \ZakPoint) \|
  \leq \CertifiedError \, \| v_Y (\TranslateZak \ZakPoint) \|
  \leq \frac{5}{4} \CertifiedError,
\]
thanks to \eqref{eq:numer-error}.
Moreover, the final equality in \eqref{eq:CorrectionMapLipschitz} used that
$r_0+(5/4)\CertifiedError=1$.
Since $r_0>7/12$ and
$\CertifiedError<1/3$,
\begin{equation}
 \frac{\CertifiedError}{r_0^2}
 <\frac{1/3}{(7/12)^2}
 =\frac{48}{49}<1.
\end{equation}
This proves for $\CorrectionMap : \Ball \to \Ball$ that
$\mathrm{Lip} (\CorrectionMap) \leq \frac{\CertifiedError}{r_0^2} < \frac{48}{49}$.
In particular, the Banach fixed-point theorem yields the existence of the unique fixed point 
$\FixedCorrection\in\Ball$ of $\CorrectionMap$.
\\~\\
\textbf{Step 3.}
At a fixed point, \eqref{eq:Phi} says that the $\ComplementProjection(\ZakPoint)$ component
of $\FixedCocycle(\ZakPoint)\FixedInvariantVector(\TranslateZak \ZakPoint)$ is $\InvariantMultiplier(\ZakPoint)\FixedCorrection(\ZakPoint)$,
while the $\ReferenceProjection(\ZakPoint)$ component is $\InvariantMultiplier(\ZakPoint)\ReferenceSection(\ZakPoint)$. 
Adding the two components proves \eqref{eq:line-invariance}.  
Finally, \eqref{eq:d-error} and \eqref{eq:vXNormEstimate} give
\begin{equation}
 |\InvariantMultiplier(\ZakPoint)-1|
 \leq \CertifiedError\norm{\FixedInvariantVector(\TranslateZak \ZakPoint)}
 \leq \frac{5}{4}\CertifiedError
 <    \frac{5}{12}.
\end{equation}
Since $\ReferenceSection^*\FixedInvariantVector=1$, the function $\FixedInvariantVector$ is nowhere zero.
\end{proof}

\subsection{Regularity of the fixed point}\label{sec:regOfZakFunction}
The contraction argument in the previous section above was carried out in the uniform
norm, which guarantees the continuity of the fixed point $\FixedCorrection$.
However, for later purposes, we need the function $\InvariantMultiplier$ to be smooth. 
We establish this, together with smoothness of $\FixedCorrection$ and $\FixedInvariantVector$,
in this subsection.

The results in this subsection are closely related to a fixed point argument for a graph-transform, see, e.g., \cite[Theorem~3.1, p.~25]{HirschPughShub};
the corresponding regularity statement is
\cite[Theorem~3.2, p.~26]{HirschPughShub}. 
However, instead of invoking general theory, we replaced the argument
by a direct, elementary argument. 
Besides making the paper more self-contained and understandable for most researchers
of the time-frequency community,
this keeps the regularity argument in the analytic language used elsewhere in the paper.

Throughout this subsection, we write
\begin{equation}
 \ReferenceSection(\ZakPoint)
 =\begin{pmatrix}\ReferenceSection_1(\ZakPoint)\\ \ReferenceSection_2(\ZakPoint)\end{pmatrix},
 \qquad
 \NormalSection(\ZakPoint)
 :=\begin{pmatrix}-\overline{\ReferenceSection_2(\ZakPoint)}\\
                         \overline{\ReferenceSection_1(\ZakPoint)}\end{pmatrix},
 \qquad
 \UnitaryFrame(\ZakPoint)
 :=\bigl(\ReferenceSection(\ZakPoint) \,|\, \NormalSection(\ZakPoint)\bigr) \in \C^{2 \times 2},
\end{equation}
where we recall that $\ReferenceSection$ is as defined in \eqref{eq:explicit_vector_zak}.
Since $\ReferenceSection(\ZakPoint)$ is a unit vector, so is $\NormalSection(\ZakPoint)$,
$\NormalSection(\ZakPoint)\perp\ReferenceSection(\ZakPoint)$, and the $2\times2$ matrix $\UnitaryFrame(\ZakPoint)$ is
unitary.
All three functions are smooth on $\R^2$.

Given $\CorrectionField \in \CalX$, we know that $\CorrectionField(\ZakPoint)$
is  perpendicular to $\ReferenceSection(\ZakPoint)$ for every $\ZakPoint \in \R^2$
and can therefore be written uniquely as
\begin{equation}
 \CorrectionField(\ZakPoint)
 =\ProjectiveCoordinate(\ZakPoint)\NormalSection(\ZakPoint),
 \qquad \text{for} \qquad
 \ProjectiveCoordinate(\ZakPoint)
 := \NormalSection(\ZakPoint)^*\CorrectionField(\ZakPoint)\in\C.
 \label{eq:ProjectiveCoordinateDefinition}
\end{equation}
We now compute how the function $\ProjectiveCoordinate$ transforms under shifts.
With the matrices $\SewingMatrix_j(\ZakPoint)$, $j \in \{1,2\}$,
from the quasi-periodicity relations \eqref{eq:x-sewing} and \eqref{eq:omega-sewing}, we have
\begin{equation}
 \NormalSection(\ZakPoint+e_j)
 =
 \overline{\det\SewingMatrix_j(\ZakPoint)}\,
 \SewingMatrix_j(\ZakPoint)\NormalSection(\ZakPoint),
 \label{eq:OrthogonalComplementSewingRelations}
\end{equation}
as follows from Lemma~\ref{lem:OrthogonalVectorTransformation}.

Thus, for a given vector Zak function $\CorrectionField=\ProjectiveCoordinate\NormalSection$,
the scalar function $\ProjectiveCoordinate : \R^2 \to \C$ satisfies
\begin{equation}\label{eq:xi-boundary}
 \ProjectiveCoordinate(\ZakPoint+e_1)=-\e^{2\pi\ii\ZakFrequency}\ProjectiveCoordinate(\ZakPoint),
 \qquad
 \ProjectiveCoordinate(\ZakPoint+e_2)=-\ProjectiveCoordinate(\ZakPoint),
 \qquad \ZakPoint=(\ZakPosition,\ZakFrequency).
\end{equation}
This follows from the quasi-periodicity relations \eqref{eq:x-sewing} and \eqref{eq:omega-sewing}.
Indeed, since we have $\det \FirstSewingMatrix (\ZakPoint) = - e^{2 \pi i \ZakFrequency}$
for $\ZakPoint = (\ZakPosition, \ZakFrequency)$, we see
\begin{align*}
  \ProjectiveCoordinate(\ZakPoint + e_1)
  &= \NormalSection(\ZakPoint + e_1)^\ast \CorrectionField(\ZakPoint + e_1) \\
  &=
  \Bigl( \overline{\det \FirstSewingMatrix (\ZakPoint)} \, \FirstSewingMatrix (\ZakPoint) \, \NormalSection(\ZakPoint) \Bigr)^\ast
  \FirstSewingMatrix (\ZakPoint) \CorrectionField (\ZakPoint) \\
  &= - e^{2 \pi i \ZakFrequency}
     \, \NormalSection (\ZakPoint)^\ast
     \FirstSewingMatrix (\ZakPoint)^\ast
     \FirstSewingMatrix (\ZakPoint)
     \CorrectionField (\ZakPoint)
  = - e^{2 \pi i \ZakFrequency} \, \ProjectiveCoordinate (\ZakPoint)
  ,
\end{align*}
and, since $\det \SecondSewingMatrix (\ZakPoint) = -1$, we have that
\begin{align*}
  \ProjectiveCoordinate(\ZakPoint + e_2)
  &= \NormalSection(\ZakPoint + e_2)^\ast \CorrectionField(\ZakPoint + e_2) \\
  &=
  \Bigl( \overline{\det \SecondSewingMatrix (\ZakPoint)} \, \SecondSewingMatrix (\ZakPoint) \, \NormalSection(\ZakPoint) \Bigr)^\ast
  \SecondSewingMatrix (\ZakPoint) \CorrectionField (\ZakPoint) \\
  &= - \NormalSection (\ZakPoint)^\ast \, \SecondSewingMatrix (\ZakPoint)^\ast \SecondSewingMatrix (\ZakPoint) \CorrectionField (\ZakPoint)
  = - \ProjectiveCoordinate (\ZakPoint)
  .
\end{align*}

We next express $\FixedCocycle(\ZakPoint)$ using the orthonormal bases
$\bigl(\ReferenceSection(\TranslateZak \ZakPoint),\NormalSection(\TranslateZak \ZakPoint)\bigr)$ at the input and
$\bigl(\ReferenceSection(\ZakPoint),\NormalSection(\ZakPoint)\bigr)$ at the output to obtain the matrix representation
\begin{equation}\label{eq:local-cocycle-matrix}
 \LocalMatrix(\ZakPoint)
 :=
 \UnitaryFrame(\ZakPoint)^*
 \FixedCocycle(\ZakPoint)
 \UnitaryFrame(\TranslateZak\ZakPoint)
 =:
 \begin{pmatrix}
  \LocalMatrixEntry_{1,1}(\ZakPoint)
    &\LocalMatrixEntry_{1,2}(\ZakPoint)\\
  \LocalMatrixEntry_{2,1}(\ZakPoint)
    &\LocalMatrixEntry_{2,2}(\ZakPoint)
 \end{pmatrix}.
\end{equation}
Similarly, we express the matrix $\IdealCocycle(z)$ in the same coordinates as
\begin{align} \label{eq:matrixB0}
 \UnitaryFrame(\ZakPoint)^*
 \IdealCocycle(\ZakPoint)
 \UnitaryFrame(\TranslateZak\ZakPoint)
 = \begin{pmatrix}1&0\\0&0\end{pmatrix},
\end{align}
cf. the calculation in Appendix~\ref{sub:ideal-local-matrix}.
Since unitary changes of coordinates preserve the operator norm,
it follows from Definition~\ref{def:AdmissibleSequence} and \eqref{eq:Bstar} that
\begin{equation}\label{eq:local-matrix-error}
 \sup_{\ZakPoint \in \R^2}
   \left\|\LocalMatrix(\ZakPoint)-\begin{pmatrix}1&0\\0&0\end{pmatrix}\right\|_{\op}
 \leq \CertifiedError
 <\frac{1}{3}.
\end{equation}

With this preparation, we can now prove the following technical lemma
which will be the central ingredient for proving that the fixed point
$\FixedCorrection$ from Proposition~\ref{prop:contraction} is smooth.

\begin{lemma}
\label{lem:projective-map}
Let the functions $\LocalMatrixEntry_{i,j} : \R^2 \to \C$ be as defined in \eqref{eq:local-cocycle-matrix}.
There exists an open set $V \subset \C$ satisfying
\(
  V_0 := \{ \ProjectiveCoordinate \in \C \,\,:\,\, |\ProjectiveCoordinate| \leq 3/4 \} 
  \subset V
\)
such that the function
\begin{equation}\label{eq:local-projective-map}
  \ProjectiveMap :  \R^2 \times V \to \C, \qquad
 \ProjectiveMap(\ZakPoint,\ProjectiveCoordinate)
 :=
 \frac{
  \LocalMatrixEntry_{2,1}(\ZakPoint)
  +\LocalMatrixEntry_{2,2}(\ZakPoint)\ProjectiveCoordinate
 }{
  \LocalMatrixEntry_{1,1}(\ZakPoint)
  +\LocalMatrixEntry_{1,2}(\ZakPoint)\ProjectiveCoordinate
 }
\end{equation}
is a well-defined $C^\infty$ function that is holomorphic in its second variable.

Moreover, for $\ProjectiveCoordinate \in V_0$,
the denominator in \eqref{eq:local-projective-map} satisfies
\begin{equation}\label{eq:local-denom-regularity}
 \left|
  \LocalMatrixEntry_{1,1}(\ZakPoint)
  +\LocalMatrixEntry_{1,2}(\ZakPoint)\ProjectiveCoordinate
 \right|>\frac7{12}.
\end{equation}
Additionally,
\begin{equation}\label{eq:vertical-derivative-bound}
 \sup_{\ZakPoint \in \R^2,\, \ProjectiveCoordinate \in V_0}
 \left|
  \partial_{\ProjectiveCoordinate}
  \ProjectiveMap(\ZakPoint,\ProjectiveCoordinate)
 \right|
 \le \frac{\CertifiedError}{(1-(5/4)\CertifiedError)^2}
 <\frac{48}{49}=: \ProjectiveContractionConstant<1.
\end{equation}

Finally, given any $\CorrectionField \in \Ball \subset \CalX$,
using \eqref{eq:ProjectiveCoordinateDefinition} to write
$\CorrectionField = \ProjectiveCoordinate \cdot \NormalSection$
and  $\CorrectionMap \CorrectionField = \ProjectiveCoordinate^\sharp \cdot \NormalSection$,
then
\begin{equation}
 \ProjectiveCoordinate^\sharp (\ZakPoint)
 =
 \ProjectiveMap\bigl(
  \ZakPoint,
  \ProjectiveCoordinate(\TranslateZak\ZakPoint)
 \bigr),
 \qquad \ZakPoint \in \R^2
 .
    \label{eq:CorrectionMapInNormalCorrdinate}
\end{equation}
In particular, writing the fixed point $\FixedCorrection$
from Proposition~\ref{prop:contraction} as
$\FixedCorrection(\ZakPoint) =\ProjectiveCoordinate_*(\ZakPoint)\NormalSection(\ZakPoint)$
we have
\begin{equation}\label{eq:local-fixed}
 \ProjectiveCoordinate_*(\ZakPoint)
 =
 \ProjectiveMap\bigl(
  \ZakPoint,
  \ProjectiveCoordinate_*(\TranslateZak\ZakPoint)
 \bigr),
 \qquad
 |\ProjectiveCoordinate_*(\ZakPoint)|<\frac57.
\end{equation}
\end{lemma}

\begin{proof}
Define
$
  V := \{ \ProjectiveCoordinate \in \C \,\,:\,\, |\ProjectiveCoordinate| < 2 \}
  .
$
Then, for $\ProjectiveCoordinate \in V$, a direct calculation shows that
\begin{equation}
    \begin{aligned}
        \LocalMatrixEntry_{1,1} (\ZakPoint) + \LocalMatrixEntry_{1,2}(\ZakPoint) \ProjectiveCoordinate
        &=  \left( \LocalMatrix (\ZakPoint) \begin{pmatrix} 1 \\ \ProjectiveCoordinate \end{pmatrix} \right)_1 \\
        &=  \left[
        \left(
          \LocalMatrix (\ZakPoint) - \begin{pmatrix} 1 & 0 \\ 0 & 0 \end{pmatrix}
          \right)
          \begin{pmatrix} 1 \\ \ProjectiveCoordinate \end{pmatrix}
        \right]_1
        + \left[ \begin{pmatrix} 1 & 0 \\ 0 & 0 \end{pmatrix} \begin{pmatrix} 1 \\ \ProjectiveCoordinate \end{pmatrix} \right]_1 \\
        &=: \psi(\ZakPoint, \ProjectiveCoordinate) + 1 .
    \end{aligned}
    \label{eq:DenominatorRewriting}
\end{equation}
Next, the matrix estimate \eqref{eq:local-matrix-error} yields
\[
  |\psi (\ZakPoint, \ProjectiveCoordinate)|
  \leq \CertifiedError \cdot \| (1,\ProjectiveCoordinate) \|
  =    \CertifiedError \cdot \sqrt{1 + |\xi|^2}
  <    \CertifiedError \sqrt{5}
  \leq \frac{1}{3} \sqrt{5}
  <    1
  .
\]
This implies that $\ProjectiveMap : \R^2 \times V \to \C$ is well-defined and smooth, and holomorphic in the second variable. Moreover, if  $\ProjectiveCoordinate \in V_0$, then the above calculation shows $|\psi (\ZakPoint, \ProjectiveCoordinate)| \leq \CertifiedError \sqrt{1 + |\ProjectiveCoordinate|^2} < \frac{1}{3} \cdot \frac{5}{4}$, and hence
\[
 \left|
  \LocalMatrixEntry_{1,1}(\ZakPoint)
  +\LocalMatrixEntry_{1,2}(\ZakPoint)\ProjectiveCoordinate
 \right|
 \ge
 1-\CertifiedError\sqrt{1+|\ProjectiveCoordinate|^2}
 >1-\frac{1}{3}\cdot\frac{5}{4}=\frac{7}{12}
\]
which proves \eqref{eq:local-denom-regularity}.

To prove \eqref{eq:vertical-derivative-bound}, write the numerator and
denominator in \eqref{eq:local-projective-map} as
$N_{\ZakPoint}(\ProjectiveCoordinate)$ and
$D_{\ZakPoint}(\ProjectiveCoordinate)$, respectively, and put
$r_0 :=1-(5/4)\CertifiedError$. 
For
$\ProjectiveCoordinate, \widetilde{\ProjectiveCoordinate} \in V_0$,
\eqref{eq:local-matrix-error} gives
the following estimates:
\begin{equation} \label{eq:estimates_differences}
\begin{aligned}
 |N_{\ZakPoint}(\ProjectiveCoordinate)
   -N_{\ZakPoint}(\widetilde{\ProjectiveCoordinate})|
 &\le\CertifiedError
   |\ProjectiveCoordinate-\widetilde{\ProjectiveCoordinate}|,\\
 |D_{\ZakPoint}(\ProjectiveCoordinate)
   -D_{\ZakPoint}(\widetilde{\ProjectiveCoordinate})|
 &\le\CertifiedError
   |\ProjectiveCoordinate-\widetilde{\ProjectiveCoordinate}|,\\
 |N_{\ZakPoint}(\widetilde{\ProjectiveCoordinate})|
 &\le\frac54\CertifiedError,\\
 |D_{\ZakPoint}(\ProjectiveCoordinate)|,
 |D_{\ZakPoint}(\widetilde{\ProjectiveCoordinate})|
 &\ge r_0,
\end{aligned}
\end{equation}
cf.
Appendix~\ref{sub:local-projective-estimates} for a detailed verification.
The quotient identity \eqref{eq:quotient-difference} therefore yields
\begin{align*}
 |\ProjectiveMap(\ZakPoint,\ProjectiveCoordinate)
   -\ProjectiveMap(\ZakPoint,\widetilde{\ProjectiveCoordinate})|
 &\quad\le
 \left(
  \frac{\CertifiedError}{r_0}
  +\frac{(5/4)\CertifiedError^2}{r_0^2}
 \right)
 |\ProjectiveCoordinate-\widetilde{\ProjectiveCoordinate}|
 =\frac{\CertifiedError}{r_0^2}
 |\ProjectiveCoordinate-\widetilde{\ProjectiveCoordinate}| \leq \frac{48}{49}  |\ProjectiveCoordinate-\widetilde{\ProjectiveCoordinate}|,
\end{align*}
where the last inequality holds since by $\CertifiedError<1/3$ we have
\[
 r_0=1-\frac54\CertifiedError
 >1-\frac5{12}=\frac7{12},
\]
and therefore
\[
 \frac{\CertifiedError}{r_0^2}
 <\frac{1/3}{(7/12)^2}
 =\frac{48}{49}.
\]
Taking the limit as
$\widetilde{\ProjectiveCoordinate}\to\ProjectiveCoordinate$ proves
\eqref{eq:vertical-derivative-bound}.

For the proof of \eqref{eq:CorrectionMapInNormalCorrdinate}, let $\CorrectionField\in\Ball$, and use
\eqref{eq:ProjectiveCoordinateDefinition} to write
\[
 \CorrectionField(\ZakPoint)
 =\ProjectiveCoordinate(\ZakPoint)\NormalSection(\ZakPoint),
 \qquad
 (\CorrectionMap\CorrectionField)(\ZakPoint)
 =\ProjectiveCoordinate^\sharp(\ZakPoint)
  \NormalSection(\ZakPoint).
\]
Since $\NormalSection(\ZakPoint)$ is a unit vector, we have
$
 \ProjectiveCoordinate^\sharp(\ZakPoint)
 =\NormalSection(\ZakPoint)^*
  (\CorrectionMap\CorrectionField)(\ZakPoint).
$
The definition \eqref{eq:Phi} of $\CorrectionMap$ therefore gives
\begin{align*}
 \ProjectiveCoordinate^\sharp(\ZakPoint)
 &=
 \frac{
  \NormalSection(\ZakPoint)^*
  \ComplementProjection(\ZakPoint)
  \FixedCocycle(\ZakPoint)
  \CandidateInvariantVector(\TranslateZak\ZakPoint)
 }{
  \NormalizationScalar{\CorrectionField}(\ZakPoint)
 }
 =
 \frac{
  \NormalSection(\ZakPoint)^*
  \FixedCocycle(\ZakPoint)
  \CandidateInvariantVector(\TranslateZak\ZakPoint)
 }{
  \ReferenceSection(\ZakPoint)^*
  \FixedCocycle(\ZakPoint)
  \CandidateInvariantVector(\TranslateZak\ZakPoint)
 },
\end{align*}
where we used that $\ComplementProjection(\ZakPoint)$ is the
orthogonal projection onto the span of $\NormalSection(\ZakPoint)$,
and hence
$\NormalSection(\ZakPoint)^*\ComplementProjection(\ZakPoint)
=\NormalSection(\ZakPoint)^*$.
By \eqref{eq:vXDefinition}, we have
\begin{align*}
 \CandidateInvariantVector(\TranslateZak\ZakPoint)
 &=\ReferenceSection(\TranslateZak\ZakPoint)
   +\ProjectiveCoordinate(\TranslateZak\ZakPoint)
    \NormalSection(\TranslateZak\ZakPoint)\\
 &=\UnitaryFrame(\TranslateZak\ZakPoint)
   \begin{pmatrix}
    1\\
    \ProjectiveCoordinate(\TranslateZak\ZakPoint)
   \end{pmatrix}.
\end{align*}
Using the definition \eqref{eq:local-cocycle-matrix} of
$\LocalMatrix(\ZakPoint)$, the numerator and denominator above
satisfy
\begin{align*}
    \NormalSection(\ZakPoint)^*
    \FixedCocycle(\ZakPoint)
    \CandidateInvariantVector(\TranslateZak\ZakPoint)
    & = \LocalMatrixEntry_{2,1}(\ZakPoint)
    +\LocalMatrixEntry_{2,2}(\ZakPoint)
    \ProjectiveCoordinate(\TranslateZak\ZakPoint),\\
    \ReferenceSection(\ZakPoint)^*
    \FixedCocycle(\ZakPoint)
    \CandidateInvariantVector(\TranslateZak\ZakPoint)
    & = \LocalMatrixEntry_{1,1}(\ZakPoint)
    +\LocalMatrixEntry_{1,2}(\ZakPoint)
    \ProjectiveCoordinate(\TranslateZak\ZakPoint).
\end{align*}
Consequently,
\begin{equation}
    \ProjectiveCoordinate^\sharp(\ZakPoint)
    =
    \frac{
    \LocalMatrixEntry_{2,1}(\ZakPoint)
    +\LocalMatrixEntry_{2,2}(\ZakPoint)
    \ProjectiveCoordinate(\TranslateZak\ZakPoint)
    }{
    \LocalMatrixEntry_{1,1}(\ZakPoint)
    +\LocalMatrixEntry_{1,2}(\ZakPoint)
    \ProjectiveCoordinate(\TranslateZak\ZakPoint)
    }
    =\ProjectiveMap\bigl(
    \ZakPoint,
    \ProjectiveCoordinate(\TranslateZak\ZakPoint)
    \bigr),
\end{equation}
which proves \eqref{eq:CorrectionMapInNormalCorrdinate}.

Lastly, based on \eqref{eq:CorrectionMapInNormalCorrdinate}, the fixed-point equation $\CorrectionMap \FixedCorrection=\FixedCorrection$ gives the first part of \eqref{eq:local-fixed}, while \eqref{eq:xstar-small} gives its last assertion.
\end{proof}

Using Lemma~\ref{lem:projective-map}, we prove the following theorem.

\begin{theorem}\label{lem:regularity}
    The functions $\FixedCorrection$, $\FixedInvariantVector$ belong to $ C^\infty(\R^2;\C^2)$, and $\InvariantMultiplier$ is a smooth $\Z^2$-periodic function.
\end{theorem}

\begin{proof}
    Consider the iteration
    \begin{equation}
        \CorrectionField_0 := 0 \in \Ball,
        \quad \text{and} \quad
        \CorrectionField_{n+1} := \CorrectionMap \CorrectionField_n \in \Ball,
        \qquad n \in \N_0.
    \end{equation}
    We refer to Proposition~\ref{prop:contraction} for the fact that $\CorrectionField_n \in \Ball$.
    Using \eqref{eq:ProjectiveCoordinateDefinition}, we write $\CorrectionField_n(\ZakPoint) = \ProjectiveCoordinate_n(\ZakPoint)\NormalSection(\ZakPoint)$ and $\FixedCorrection(\ZakPoint) =\ProjectiveCoordinate_*(\ZakPoint)\NormalSection(\ZakPoint)$. Equation~\eqref{eq:CorrectionMapInNormalCorrdinate} of Lemma~\ref{lem:projective-map} shows that
    \begin{equation}\label{eq:scalar-iteration}
        \ProjectiveCoordinate_0(\ZakPoint) = 0,
        \qquad
        \ProjectiveCoordinate_{n+1}(\ZakPoint)
        =
        \ProjectiveMap\bigl(
        \ZakPoint,
        \ProjectiveCoordinate_n(\TranslateZak\ZakPoint)
        \bigr),
        \qquad \ZakPoint \in \R^2
        ,
    \end{equation}
    where $\ProjectiveMap$ was defined in \eqref{eq:local-projective-map}. By induction, the smoothness assertion in Lemma~\ref{lem:projective-map} shows that every $\ProjectiveCoordinate_n\in C^\infty(\R^2)$, and hence every $\CorrectionField_n$ is a $C^\infty$ vector Zak function on $\R^2$. Therefore, $\ProjectiveCoordinate_n$ satisfies \eqref{eq:xi-boundary}. Moreover, $\CorrectionField_n \to \FixedCorrection$ in $\CalX$ by Proposition~\ref{prop:contraction} and Banach's fixed point theorem, meaning that $\CorrectionField_n \to \FixedCorrection$ uniformly on $\R^2$. As $\ProjectiveCoordinate_n (\ZakPoint) = \NormalSection(\ZakPoint)^\ast \CorrectionField_n (\ZakPoint)$ and $\|\NormalSection(\ZakPoint)\| = 1$, that implies $\ProjectiveCoordinate_n \to \ProjectiveCoordinate_*$ uniformly on $\R^2$.
    
    The main idea of the proof is as follows. For  each $T \geq 1$ set $\Omega_T := [-T,T]^2$ and show for each multi-index $\gamma \in \N_0^2$ that
    \begin{equation}
      C_0 (\gamma, T)
      := \sup_{n \in \N_0}
           \| \partial^\gamma \ProjectiveCoordinate_n \|_{L^\infty (\Omega_T)}
      < \infty
      .
      \label{eq:LocallyUniformDerivativeBound}
    \end{equation}
    
    In Step~1 below, we will show this inductively for derivatives of order $|\gamma|$. In Step~2, we will then verify that this implies smoothness of $\ProjectiveCoordinate_*$, and hence of $\FixedCorrection$ and $\FixedInvariantVector$. To achieve this, we combine an Arzelà–Ascoli argument with the fact that $\CorrectionField_n \to \FixedCorrection$ uniformly on $\R^2$. Lastly, in Step~3, we will prove the claimed properties of $\InvariantMultiplier$.
    \\~\\
    \textbf{Step~1.} In this step, we prove that \eqref{eq:LocallyUniformDerivativeBound} holds for all multi-indices $\gamma \in \N_0^2$ with $|\gamma| \leq r$ by induction on $r \in \N_0$. For $r = 0$, we have $\gamma = 0$ and
    $
      \| \ProjectiveCoordinate_n \|_{L^\infty(\R^2)}
      \leq \frac{3}{4}
      ,
    $
    so that \eqref{eq:LocallyUniformDerivativeBound} holds.
    
    For the inductive step, let $r \in \N$ and suppose that \eqref{eq:LocallyUniformDerivativeBound} holds for all multi-indices $\gamma \in \N_0^2$ with $|\gamma| < r$. Let $\gamma = (a,b) \in \N_0^2$ with $|\gamma| = r$; we want to verify \eqref{eq:LocallyUniformDerivativeBound} for $\gamma$. To this end, we first note that
    \begin{equation}
      \TranslateZak (\Omega_T)
      = \Omega_T - \TorusShiftVector
      \subset [-(T+1), T]^2
      \subset \bigcup_{(\eps_1, \eps_2) \in \{0,1\}^2} \bigl( [-T, T]^2 - \eps_1 e_1 - \eps_2 e_2 \bigr)
      ,
      \label{eq:TranslationInclusion}
    \end{equation}
    where we recall that
    $\TorusShiftVector=(\TranslationAlpha,\TranslationBeta)\in(0,1)^2$.
    We now use this inclusion to derive an estimate controlling derivatives on
    $\TranslateZak(\Omega_T)$ by derivatives on $\Omega_T$. 
    Let $f\in C^r(\R^2)$ satisfy \eqref{eq:xi-boundary}, i.e.,
    \begin{equation}
     f(\ZakPosition+1,\ZakFrequency)
     =-\e^{2\pi\ii\ZakFrequency}f(\ZakPosition,\ZakFrequency),
     \qquad
     f(\ZakPosition,\ZakFrequency+1)
     =-f(\ZakPosition,\ZakFrequency).
     \label{eq:DerivativeProofQuasiPeriodicity}
    \end{equation}
    For each fixed pair $(\eps_1,\eps_2) \in \{0,1\}^2$,
    \eqref{eq:DerivativeProofQuasiPeriodicity} implies 
    \[
      f(\ZakPoint)
      =(-1)^{\eps_1+\eps_2}
       \e^{-2\pi\ii\eps_1\ZakFrequency}
      f(\ZakPoint+\eps_1e_1+\eps_2e_2),
      \qquad \ZakPoint=(\ZakPosition,\ZakFrequency)\in\R^2.
    \] 
    Keeping this pair fixed, differentiate the identity. For every $\ZakPoint\in\R^2$, this gives by the product rule for higher derivatives that
    \begin{equation}\label{eq:translated-square-derivative}
    \begin{split}
     \partial^\gamma f (\ZakPoint)
     &=\partial_{\ZakPosition}^a\partial_{\ZakFrequency}^b f(\ZakPoint)\\
     &=
     (-1)^{\eps_1+\eps_2}
     \e^{-2\pi\ii\eps_1\ZakFrequency}
     \sum_{j=0}^{b}\binom bj
     (-2\pi\ii\eps_1)^{b-j}
     \bigl(\partial_{\ZakPosition}^a
           \partial_{\ZakFrequency}^j f\bigr)
     (\ZakPoint+\eps_1e_1+\eps_2e_2) \\
     &= (-1)^{\eps_1+\eps_2}
       \e^{-2\pi\ii\eps_1\ZakFrequency}
       \cdot \bigl( \partial^\gamma f\bigr)
       (\ZakPoint+\eps_1e_1+\eps_2e_2) \\
       &\quad +
       (-1)^{\eps_1+\eps_2}
       \e^{-2\pi\ii\eps_1\ZakFrequency}
       \sum_{j=0}^{b-1}
       \binom{b}{j}
       (-2\pi\ii\eps_1)^{b-j}
       \bigl(\partial_{\ZakPosition}^a
             \partial_{\ZakFrequency}^j f\bigr)
       (\ZakPoint+\eps_1e_1+\eps_2e_2)
    \end{split}
    \end{equation}
    Fix $\ZakPoint'=(\ZakPosition',\ZakFrequency') \in \TranslateZak (\Omega_T)$ for the moment.
    By \eqref{eq:TranslationInclusion}, we can choose $(\eps_1, \eps_2) \in \{0,1\}^2$ with
    \[
      \ZakPoint'+\eps_1e_1+\eps_2e_2 \in \Omega_T.
    \]
    Applying \eqref{eq:translated-square-derivative} to $\ZakPoint'$,
    we thus obtain existence of a constant $C_1 (\gamma,T)>0$ such that
    \begin{equation}\label{eq:translated-square-estimate}
     \norm{\partial^\gamma f}_{L^\infty(\TranslateZak (\Omega_T))}
     \leq
     \norm{\partial^\gamma f}_{L^\infty(\Omega_T)}
     +  C_1 (\gamma,T) \sum_{|\kappa|<r} \norm{\partial^\kappa f}_{L^\infty(\Omega_T)}.
    \end{equation}
    Here, it is \emph{crucial} for what follows that the term $\norm{\partial^\gamma f}_{L^\infty(\Omega_T)}$ appears without a factor in front of it. We remark that $C_1(\gamma,T)$ can in fact be chosen independently of $T \geq 1$, but we will not need this.
    
    Next, we use \eqref{eq:scalar-iteration} and a Fa\'a di Bruno-type formula (cf.\ Lemma~\ref{lem:FaaDiBrunoBound}) to write
    \begin{equation}\label{eq:highest-derivative-recursion}
        \partial^\gamma\ProjectiveCoordinate_{n+1}(\ZakPoint)
        = \partial^\gamma \Bigl[ \ProjectiveMap\bigl( \ZakPoint, \ProjectiveCoordinate_n(\TranslateZak\ZakPoint) \bigr) \Bigr]
        = A_n(\ZakPoint)
        (\partial^\gamma\ProjectiveCoordinate_n) (\TranslateZak\ZakPoint)
        + H_{\gamma,n}(\ZakPoint),
    \end{equation}
    where
    \begin{equation}
        A_n(\ZakPoint)
        := (\partial_{\ProjectiveCoordinate}\ProjectiveMap)
        \bigl( \ZakPoint, \ProjectiveCoordinate_n(\TranslateZak\ZakPoint) \bigr),
        \qquad
        \ZakPoint \in \R^2
        \label{eq:IteratorDerivative}
    \end{equation}
    and where the term
    \[
      H_{\gamma,n}(\ZakPoint)
      := R_\gamma [\ProjectiveCoordinate_n] (\ZakPoint)
    \]
    contains all the partial derivatives of $\ProjectiveCoordinate_n$ obtained from successive applications of the chain rule that are of lower order than $|\gamma| = r$. The precise form of $R_\gamma [\ProjectiveCoordinate_n] (\ZakPoint)$ is not important for us; we only need the following estimate, for $n \in \N_0$:
    \begin{align*}
        \| H_{\gamma,n} \|_{L^\infty(\Omega_T)}
        &= \| R_\gamma [\ProjectiveCoordinate_n] \|_{L^\infty(\Omega_T)} \\
        &\leq
        C_2 (\gamma, T) \cdot
        (1 + \| \ProjectiveCoordinate_n \|_{C^{|\gamma| - 1}(\TranslateZak(\Omega_T))})^{|\gamma|} \\
        &\leq
        C_2 (\gamma, T) \cdot
        \Bigl(1+\max_{\ell\in\N_0^2,\,|\ell|\leq r-1}
        C_0(\ell,T+1)\Bigr)^{|\gamma|} \\
        &=: C_3 (\gamma,T) < \infty.
    \end{align*}
    Here, the second estimate follows from Lemma~\ref{lem:FaaDiBrunoBound}, which is applicable since $\Omega_T \subset \R^2$ and $V_0 \subset V$ are compact and $|\ProjectiveCoordinate_n(\ZakPoint)| = \| \CorrectionField_n(\ZakPoint) \|_\infty \leq 3/4$, as $\CorrectionField_n \in \Ball$. The penultimate step used that $\TranslateZak(\Omega_T) \subset \Omega_{T+1}$ and inductively applied \eqref{eq:LocallyUniformDerivativeBound} (for $T+1$ instead of $T$).
    
    Next, note that the quantity $A_n(\ZakPoint)$ from \eqref{eq:IteratorDerivative} satisfies the bound
    \[
      |A_n (\ZakPoint)|
      \leq \ProjectiveContractionConstant
      <    1
      , \qquad \ZakPoint \in \R^2
      ,
    \]
    due to \eqref{eq:vertical-derivative-bound}. Combining this estimate with \eqref{eq:highest-derivative-recursion} and \eqref{eq:translated-square-estimate}, we finally see for
    \[
      a_n
      :=a_{n} (\gamma, T)
      := \| \partial^\gamma \ProjectiveCoordinate_{n} \|_{L^\infty(\Omega_T)}
    \]
    that
    \begin{align*}
      a_{n+1} (\gamma, T)
      &= \| \partial^\gamma \ProjectiveCoordinate_{n+1} \|_{L^\infty(\Omega_T)} \\
      &\leq \|A_n\|_{L^\infty(\R^2)}
             \cdot \| \partial^\gamma \ProjectiveCoordinate_n \|_{L^\infty(\TranslateZak (\Omega_T))}
             + \| H_{\gamma,n} \|_{L^\infty(\Omega_T)} \\
      &\leq
        \ProjectiveContractionConstant
        \cdot  \Bigl(
                   \| \partial^\gamma \ProjectiveCoordinate_n \|_{L^\infty (\Omega_T)}
                   + C_1 (\gamma,T) \sum_{|\kappa| < r} \| \partial^\kappa \ProjectiveCoordinate_n \|_{L^\infty(\Omega_T)}
               \Bigr)
        + C_3 (\gamma, T) \\
      &\leq
        \ProjectiveContractionConstant
        \cdot \| \partial^\gamma \ProjectiveCoordinate_n \|_{L^\infty (\Omega_T)}
        + C_1 (\gamma,T) \sum_{|\kappa| < r} C_0 (\kappa, T)
        + C_3 (\gamma, T) \\
      &=: \ProjectiveContractionConstant\,a_n (\gamma, T)
          + C_4 (\gamma,T)
      ,
    \end{align*}
    where the penultimate step used that $\ProjectiveContractionConstant<1$ and the induction assumption.
    
    The above recursive estimate implies that the sequence $\bigl(a_n(\gamma,T) \bigr)_{n \in \N_0}$ is bounded;
    in fact,  Lemma~\ref{lem:RecursiveSequenceEstimate} shows that
    \[
      \| \partial^\gamma \ProjectiveCoordinate_{n} \|_{L^\infty(\Omega_T)}
      = a_n (\gamma,T)
      \leq a_0 (\gamma,T)
      + \frac{C_4 (\gamma,T)}{1 - \ProjectiveContractionConstant}
      \qquad \forall \, n \in \N
      .
    \]
    This completes the induction and thus Step~1 of the proof.
    \\~\\
    \textbf{Step 2.} We next prove the smoothness of $\ProjectiveCoordinate_*$, $\FixedCorrection$, and $\FixedInvariantVector$. Note that $\Omega_T = [-T,T]^2$ is convex. Hence, the estimate \eqref{eq:LocallyUniformDerivativeBound} (applied to $\gamma + e_i$ for $i \in \{1,2\}$ instead of $\gamma$) implies for each $\gamma \in \N_0^2$ and each $T \geq 1$ that the sequence $(\partial^\gamma \ProjectiveCoordinate_n)_{n \in \N_0}$ is uniformly Lipschitz, and thus equicontinuous, on $\Omega_T$. Moreover, \eqref{eq:LocallyUniformDerivativeBound} implies that the sequence is also uniformly bounded on this set.
    
    By an application of the Arzelà–Ascoli theorem, this implies for arbitrary $r \in \N$ and $T \geq 1$ that for a subsequence $(n_\ell)_\ell$ we have $\partial^\gamma \ProjectiveCoordinate_{n_\ell} \to f_\gamma$ uniformly on $\Omega_T$ as $\ell \to \infty$, for all $|\gamma| \leq r$. Since $\ProjectiveCoordinate_{n_\ell} \to \ProjectiveCoordinate_*$ uniformly on $\R^2$, this implies that $\ProjectiveCoordinate_*$ is $C^r$ on $\Omega_T$. Since $r \in \N$ and $T \geq 1$ were arbitrary, this implies that $\ProjectiveCoordinate_*$ is $C^\infty$ on $\R^2$.
    
    Since $\FixedCorrection=\ProjectiveCoordinate_*\NormalSection$ and $\NormalSection\in C^\infty(\R^2;\C^2)$, it follows that $\FixedCorrection\in C^\infty(\R^2;\C^2)$. Finally, \eqref{eq:vq} gives $\FixedInvariantVector=\ReferenceSection+\FixedCorrection$; hence $\FixedInvariantVector\in C^\infty(\R^2;\C^2)$ as well.
    \\~\\
    \textbf{Step 3.} Finally, we verify the claimed properties of $\InvariantMultiplier$.
    Applying \eqref{eq:qIsPeriodic} to $\NormalizationScalar{\FixedCorrection}$ shows directly that
    $\InvariantMultiplier$ is $\Z^2$-periodic. Moreover, since $\ReferenceSection$ is smooth by Lemma~\ref{lem:unit}, and since it follows directly from the definitions of  $\LatticeWeylMatrix_{m,n}$ (see \eqref{eq:Lmn-factorization}) and $\FixedMatrixField$ (see \eqref{eq:Astar}) that $\FixedMatrixField$ is smooth, while $\OffsetPhase$ is smooth by definition (see \eqref{eq:EtaDefinition}), we see that also the function $\FixedCocycle : \R^2 \to \C^{2 \times 2}$ defined in \eqref{eq:Bstar} is smooth.
    
    Hence, since we showed above that $\FixedCorrection$ and $\FixedInvariantVector$ are smooth, it follows from the definition of $\NormalizationScalar{\FixedCorrection}$ in \eqref{eq:vq} that $\InvariantMultiplier$ is smooth on $\R^2$.
\end{proof}

\section{Reduction of the scalar multiplier to a constant}
\label{sec:cohomology}

In this section, we will show that the scalar multiplier $\InvariantMultiplier$ can be reduced to a constant. 
This will be shown in Proposition~\ref{prop:cohomology}. For this, we will first prove an auxiliary result in Subsection \ref{sec:estimate}.

\subsection{Estimate for the translation vector} \label{sec:estimate}

For  $s \in \R$, we define its distance to the nearest integer by
\begin{equation}
 \norm{s}_{\R/\Z}=\inf_{\ell\in\Z}|s-\ell|.
\end{equation}

We recall the arithmetic data $
    \CubicIrrational := \sqrt[3]{2},$
    $\TranslationAlpha := \CubicIrrational-1$
    and $\TranslationBeta := \CubicIrrational^2-1$ from Equation \eqref{eq:arith-data}. 
    The following lemma is well-known. For example, it is a special case of \cite[Chapter~VI, Lemma~1A]{Schmidt1980Diophantine}. Nevertheless, we include its proof for the sake of completeness.

\begin{lemma} \label{lem:dio}
For every nonzero $(m,n)\in\Z^2$,
\begin{equation}\label{eq:dio}
 \norm{m\TranslationAlpha+n\TranslationBeta}_{\R/\Z}
 \ge \frac{1}{12}(1+|m|+|n|)^{-2}.
\end{equation}
\end{lemma}

\begin{proof}
Choose $\ell\in\Z$ nearest to $-m\TranslationAlpha-n\TranslationBeta$ and set
\begin{equation}
 d:=\ell+m\TranslationAlpha+n\TranslationBeta
 = k+m\CubicIrrational+n\CubicIrrational^2
 \qquad \text{for} \qquad k:=\ell-m-n.
\end{equation}
Then
\begin{equation}
 \norm{m\TranslationAlpha+n\TranslationBeta}_{\R/\Z} = |d| 
 \le\frac12.
\end{equation}
Since the polynomial $x^3-2$ has no rational root, it is
irreducible over $\Q$, which implies that $1,\CubicIrrational,\CubicIrrational^2$ are linearly
independent over $\Q$. Therefore, $d\ne0$.

Next, set
$
 q := k^3+2m^3+4n^3-6kmn\in\Z.
$
Applying the elementary identity 
\begin{equation}
 a^3+b^3+c^3-3abc
 =(a+b+c)(a^2+b^2+c^2-ab-ac-bc)
\end{equation}
for $a, b, c \in \R$,  with $a=k$, $b=m\CubicIrrational$, and
$c=n\CubicIrrational^2$, gives
\begin{equation}\label{eq:elementary-cubic-identity}
 q
 = d \left[
  \left(k-\frac{m\CubicIrrational+n\CubicIrrational^2}{2}\right)^2
  +\frac34\left(m\CubicIrrational-n\CubicIrrational^2\right)^2
 \right].
\end{equation}
The expression in brackets on the right-hand side above is positive. Indeed, if both squares were zero, then $m\CubicIrrational=n\CubicIrrational^2$, and hence the irrationality of $\CubicIrrational$ would imply that $m=n=0$, which would in turn imply that $k=0$. This is impossible because $(m,n) \in \Z^2$ is assumed to be nonzero.  Since, as argued above, also $d \ne0$, it follows from \eqref{eq:elementary-cubic-identity}  that the integer $q$ is nonzero.  Therefore $|q|\ge1$.

Let
$
 p:= 1+|m|+|n|.
$
Using $1<\CubicIrrational<3/2$, $\CubicIrrational^2<2$, and $|d|\le1/2$, we
obtain
\begin{equation}
 |k|
 \le|d|+\CubicIrrational|m|+\CubicIrrational^2|n|
 \le\frac12+\frac32|m|+2|n|
 \le2p.
\end{equation}
It follows that
\begin{equation}
 |m\CubicIrrational-n\CubicIrrational^2|
 \le2p \text{ and } \left|k-\frac{m\CubicIrrational+n\CubicIrrational^2}{2}\right|
 \le3p.
\end{equation}
Consequently, the bracket in
\eqref{eq:elementary-cubic-identity} can be estimated from above by
\begin{equation}
 (3p)^2+\frac34(2p)^2=12p^2.
\end{equation}
In combination, this shows that
\begin{equation}
 1\le |q|\le12 p^2|d|,
\end{equation}
and therefore
\begin{equation}
 \norm{m\TranslationAlpha+n\TranslationBeta}_{\R/\Z}
 =|d|
 \ge\frac1{12(1+|m|+|n|)^2},
\end{equation}
which settles the claim.
\end{proof}

As a consequence of Lemma \ref{lem:dio}, for nonzero $k = (m,n) \in \Z^2$,
\begin{equation}\label{eq:denominator-bound}
 \abs{1-\e^{-2\pi\ii k\cdot\TorusShiftVector}}
 =2\abs{\sin(\pi k\cdot\TorusShiftVector)}
 \ge\frac13(1+|m|+|n|)^{-2}.
\end{equation}
Indeed, if
$d := \norm{m\TranslationAlpha+n\TranslationBeta}_{\R/\Z} \le1/2$, then concavity of
$\sin$ on $[0,\pi/2]$ gives $\sin(\pi d)\ge2d$, and Lemma \ref{lem:dio} gives the stated bound. 

\subsection{Reduction of the scalar multiplier}
For a $\Z^2$-periodic, locally integrable function $f : \R^2 \to \C$, we use the Fourier convention
\begin{equation}
    \widehat f(k)
    :=\int_{[0,1]^2}f(\ZakPoint)\e^{-2\pi\ii k\cdot \ZakPoint}\,\dd \ZakPoint \in \C,
    \qquad k\in\Z^2.
\end{equation}

Let $\InvariantMultiplier : \R^2 \to \C$ be defined as in \eqref{eq:vq}, where we recall that this depends on the fixed admissible coefficient sequence $\aaa$. By \eqref{eq:q-disk} from Proposition \ref{prop:contraction}, the image $\InvariantMultiplier(\R^2)$ is a subset of the disk $D(1,5/12)$ of radius $\frac{5}{12}$ around $1$, where we observe that $D(1,5/12) \subset \C_- := \C \setminus (-\infty,0]$. Let $\Log : \C_- \to \C$ be the principal branch of the logarithm on $\C_-$ and define $\LogMultiplier : \R^2 \to \C$ by
\begin{equation}\label{eq:phi}
    \LogMultiplier(\ZakPoint)
    :=\Log \InvariantMultiplier(\ZakPoint),
    \qquad \ZakPoint \in \R^2
    .
\end{equation}
By Theorem~\ref{lem:regularity}, $\InvariantMultiplier$ is smooth on $\R^2$ and $\Z^2$-periodic. Hence, $\LogMultiplier \in C^\infty(\R^2)$ and it is $\Z^2$-periodic. Finally, we define a $\Z^2$-periodic function $u : \R^2 \to \C$ via its Fourier coefficients, given by
\begin{equation}\label{eq:u-hat}
    \widehat{\CohomologySolution}(0) := 0,
    \quad \text{and} \quad
    \widehat{\CohomologySolution}(k)
    :=
    \frac{
    \widehat{\LogMultiplier}(k)
    }{
    1-\e^{-2\pi\ii k\cdot\TorusShiftVector}
    }
    \quad \text{ for } k \in \Z^2 \setminus \{0\},
\end{equation}
where we recall that $\TorusShiftVector = (\TranslationAlpha,\TranslationBeta)$ was defined in \eqref{eq:arith-data}.

The following Proposition~\ref{prop:cohomology} summarizes the main consequences of this definition. In the literature, equations of the type \eqref{eq:additive-cohomology} below are commonly referred to as \emph{cohomological equations}. A standard approach is to solve the equation on the Fourier side, where one divides by factors of the form 
$
    1- \e^{-2\pi\ii k\cdot\TorusShiftVector}
$
and controls the size of these factors. Our argument below follows this strategy in the two-dimensional setting. An analogous argument on the circle can be found in, e.g., \cite[Proposition~2.9.5]{Katok_Hasselblatt_1995}.

\begin{proposition}\label{prop:cohomology}
    The Fourier series defined by the coefficients in \eqref{eq:u-hat} defines a $\Z^2$-periodic function $\CohomologySolution\in C^\infty(\R^2)$ satisfying
    \begin{equation}\label{eq:additive-cohomology}
        \CohomologySolution(\ZakPoint)
        -\CohomologySolution(\TranslateZak\ZakPoint)
        =
        \LogMultiplier(\ZakPoint)-\widehat{\LogMultiplier}(0),
        \qquad \ZakPoint \in \R^2.
    \end{equation}
    Consequently, with
    \begin{equation}\label{eq:hE}
        h :  \R^2 \to \C, \quad
        \ScalarCorrection(\ZakPoint)
        :=\e^{\CohomologySolution(\ZakPoint)},
        \quad \text{and} \quad
        \FixedEigenvalue
        := \e^{\widehat{\LogMultiplier}(0)} \in \C,
    \end{equation}
    one has $\FixedEigenvalue\ne0$,
    $\ScalarCorrection(\ZakPoint) \neq 0$ for all $\ZakPoint \in \R^2$, and
    \begin{equation}\label{eq:multiplicative-cohomology}
     \InvariantMultiplier(\ZakPoint)
     \ScalarCorrection(\TranslateZak\ZakPoint)
     =
     \FixedEigenvalue\ScalarCorrection(\ZakPoint),
     \quad \ZakPoint \in \R^2
     .
    \end{equation}
\end{proposition}

\begin{proof} We split the proof into three steps.
\\~\\
\textbf{Step 1.} We start by proving $\CohomologySolution$ is well-defined, $\Z^2$-periodic, and smooth.
Because $\LogMultiplier$ is $\Z^2$-periodic and smooth,
it follows from standard results in Fourier analysis
(see e.g. \cite[Corollary~3.3.10]{GrafakosClassicalFourier})
that for
each $N \in \N$, there exists a constant $C_N > 0$ such that
\begin{equation}
 |\widehat{\LogMultiplier}(k)|
 \leq C_N \cdot (1+|k_1|+|k_2|)^{-N},
 \quad k = (k_1, k_2) \in \Z^2.
\end{equation}
Combining this with \eqref{eq:denominator-bound} yields
\begin{equation}\label{eq:arbitraryFastDecayOfFourierSeries}
 |\widehat{\CohomologySolution}(k)|
 \leq C_N \cdot (1+|k_1|+|k_2|)^{-N+2}
 \quad k = (k_1, k_2) \in \Z^2
 .
\end{equation} 
Now, for any prescribed derivative order $j \in \N_0$, choose $N > j+4$. Since we are working in dimension two, multiplying \eqref{eq:arbitraryFastDecayOfFourierSeries} by $|k|^j$ then gives an absolutely summable sequence. Thus, the Fourier series for $\CohomologySolution$ and all of its derivatives of any fixed order converge absolutely and uniformly, so that $\CohomologySolution$ is a well-defined, $\Z^2$-periodic function with  $\CohomologySolution\in C^\infty(\R^2)$,  see, e.g., \cite[Proposition~3.3.12]{GrafakosClassicalFourier} for more details on this argument.
\\~\\
\textbf{Step 2.} We next show \eqref{eq:additive-cohomology}.
For this, note that
$
 \e^{2\pi\ii k\cdot\TranslateZak\ZakPoint}
 =
 \e^{-2\pi\ii k\cdot\TorusShiftVector}
 \e^{2\pi\ii k\cdot\ZakPoint},
$
and therefore
\begin{align*}
 &\CohomologySolution(\ZakPoint)
 -\CohomologySolution(\TranslateZak\ZakPoint)\\
 &\quad=
 \sum_{k\in\Z^2}
 \bigl(1-\e^{-2\pi\ii k\cdot\TorusShiftVector}\bigr)
 \widehat{\CohomologySolution}(k)
 \e^{2\pi\ii k\cdot\ZakPoint}\\
 &\quad=
 \sum_{k\ne0}\widehat{\LogMultiplier}(k)
 \e^{2\pi\ii k\cdot\ZakPoint}
 =\LogMultiplier(\ZakPoint) - \widehat{\LogMultiplier} (0).
\end{align*}
This proves
\eqref{eq:additive-cohomology}.
\\~\\
\textbf{Step 3.} Lastly, we prove \eqref{eq:multiplicative-cohomology}. Rearranging \eqref{eq:additive-cohomology} gives
\begin{align}\label{eq:rearrangedIdentitywithPhiandu}
 \LogMultiplier(\ZakPoint)
 +\CohomologySolution(\TranslateZak\ZakPoint)
 =\widehat{\LogMultiplier} (0)
 +\CohomologySolution(\ZakPoint),
 \quad \ZakPoint \in \R^2
 .
\end{align}
Exponentiating \eqref{eq:rearrangedIdentitywithPhiandu} and using
$\e^{\LogMultiplier}=\InvariantMultiplier$ (see \eqref{eq:phi}),
together with the definitions in \eqref{eq:hE}, yields
\[
 \InvariantMultiplier(\ZakPoint)
 \ScalarCorrection(\TranslateZak\ZakPoint)
 =\FixedEigenvalue\ScalarCorrection(\ZakPoint),
 \quad \ZakPoint \in \R^2
 ,
\]
which is \eqref{eq:multiplicative-cohomology}.
\end{proof}

\section{Proof of Theorem~\ref{thm:main_Weyl}}
\label{sec:lifting}

We will prove Theorem~\ref{thm:main} and Theorem~\ref{thm:qualitative} using the following general theorem.

\begin{theorem}\label{prop:GeneralMainResult}
  Given a finitely supported, admissible sequence $\aaa \in \C^{\Z \times \Z}$,
  there exists a nonzero Schwartz function $\FinalWindow \in \mathcal{S} (\R)$
  and a nonzero $\FixedEigenvalue \in \C$ satisfying
  \begin{equation}
  \label{eq:GeneralMainResultEquation}
  \begin{aligned}
    \FixedEigenvalue \, \FinalWindow
    =  \sum_{(m,n)\in \Z^2}
          \FixedWeylCoefficient_{m,n} \,
          \e^{\frac{\pi\ii}{2}(n\TranslationAlpha-m\TranslationBeta)} 
          \WeylShift\!\left(m+\TranslationAlpha,\frac{n+\TranslationBeta}{2}\right) \FinalWindow
    .
  \end{aligned}
  \end{equation}
  In particular, the finite family
  \[
    \TimeFrequencyShift(0,0) \FinalWindow,
    \quad
    \TimeFrequencyShift \bigl(m + \TranslationAlpha, \tfrac{n + \TranslationBeta}{2} \bigr) \FinalWindow,
    \qquad
    (m,n) \in \supp \aaa
  \]
  of time-frequency shifts of $\FinalWindow$ is linearly dependent.
\end{theorem}

\begin{proof}
The proof uses the notation and results from Sections \ref{sec:EigenvalueReformulation}--\ref{sec:cohomology}. In particular, we use the matrix-valued functions $\FixedMatrixField,\FixedCocycle$, the map $\CorrectionMap : \Ball \to \CalX$, and the operators $\FixedLatticeOperator,\FixedWeylOperator : L^2(\R) \to L^2(\R)$ that all depend on the sequence $\aaa$.
Using the Weyl product formula \eqref{eq:Weyl-shift_composition} and \eqref{eq:Astar}, \eqref{eq:Tstar} as well as \eqref{eq:Weyl_vs_TF-shifts}, \eqref{eq:arith-data}, we see that $\FixedWeylOperator = \FixedLatticeOperator\WeylShift(\WeylOffset)$ is explicitly given by
\begin{equation}
  \label{eq:WeylOperatorExplicit}
  \begin{aligned}
    \FixedWeylOperator
    &=  \sum_{(m,n)\in \Z^2}
          \FixedWeylCoefficient_{m,n} \,
          \e^{\frac{\pi\ii}{2}(n\TranslationAlpha-m\TranslationBeta)} 
          \WeylShift\!\left(m+\TranslationAlpha,\frac{n+\TranslationBeta}{2}\right).
  \end{aligned}
\end{equation}

For proving Equation~\eqref{eq:GeneralMainResultEquation}, let $\FixedCorrection \in \Ball \subset \CalX \subset C(\R^2; \C^2)$ be the fixed point constructed in Proposition~\ref{prop:contraction}, and let $\FixedInvariantVector$ and $\InvariantMultiplier$ be as in \eqref{eq:vq}. Equation~\eqref{eq:line-invariance} shows
\[
 \FixedCocycle(\ZakPoint)
 \FixedInvariantVector(\TranslateZak\ZakPoint)
 =\InvariantMultiplier(\ZakPoint)
  \FixedInvariantVector(\ZakPoint),
  \quad \ZakPoint \in \R^2,
\]
and Theorem~\ref{lem:regularity} shows that $\FixedInvariantVector : \R^2 \to \C^2$ and $\InvariantMultiplier : \R^2 \to \C$ are smooth. Moreover, Proposition~\ref{prop:cohomology} provides a $\Z^2$-periodic function $\ScalarCorrection \in C^{\infty} (\R^2; \C)$ satisfying $h(z) \neq 0$ for all $z \in \R^2$ and a nonzero constant $\FixedEigenvalue \in \C$ such that
\[
 \InvariantMultiplier(\ZakPoint)
 \ScalarCorrection(\TranslateZak\ZakPoint)
 =\FixedEigenvalue\ScalarCorrection(\ZakPoint)
 , \quad \ZakPoint \in \R^2
 .
\]
We combine these two constructions to define the function 
\begin{equation}\label{eq:Psi}
\FinalZakEigenfunction : \quad
\R^2 \to \C^2, \qquad
 \FinalZakEigenfunction(\ZakPoint)
 :=
\ScalarCorrection(\ZakPoint)\FixedInvariantVector(\ZakPoint).
\end{equation}
Note that $\FixedInvariantVector$ is a vector Zak function,
since $\FixedInvariantVector = \ReferenceSection + \FixedCorrection$,
where $\ReferenceSection$ satisfies 
\eqref{eq:x-sewing} and \eqref{eq:omega-sewing} by Lemma~\ref{lem:unit}
and $\FixedCorrection \in \Ball \subset \CalX$ satisfies \eqref{eq:x-sewing} and \eqref{eq:omega-sewing}
by definition of the space $\CalX$.
Moreover, $\FixedInvariantVector \in C^\infty (\R^2 ; \C^2)$
and $\ScalarCorrection \in C^{\infty} (\R^2)$
is $\Z^2$-periodic.
Together, this shows that $\FinalZakEigenfunction$ is a $C^\infty$ vector Zak function,
i.e., $\FinalZakEigenfunction$ is smooth and $\FinalZakEigenfunction \in \mathscr{H}_\Zak$.
Here, we use that the function $\FinalZakEigenfunction$ is trivially square-integrable
on a fundamental domain by continuity.

Next, note that $\ScalarCorrection$ never vanishes and
$\ReferenceSection(\ZakPoint)^*\FixedInvariantVector(\ZakPoint)=1$,
as follows from Lemma~\ref{lem:unit} and \eqref{eq:x-orthogonal}.
Thus, $\FixedInvariantVector \neq 0$ and by \eqref{eq:Psi} we conclude that $\FinalZakEigenfunction$
is everywhere nonzero.

Finally, using \eqref{eq:Bstar-Zak-Action}, \eqref{eq:line-invariance}, and \eqref{eq:multiplicative-cohomology},
we obtain that
\begin{align*}
  (\VectorZak\FixedWeylOperator\VectorZak^{-1} \FinalZakEigenfunction)(\ZakPoint)
    &=
    \FixedCocycle(\ZakPoint)
    \FinalZakEigenfunction(\TranslateZak\ZakPoint)
    =
    \FixedCocycle(\ZakPoint)
    \ScalarCorrection(\TranslateZak\ZakPoint)
    \FixedInvariantVector(\TranslateZak\ZakPoint)
    =
    \InvariantMultiplier(\ZakPoint)
    \ScalarCorrection(\TranslateZak\ZakPoint)
    \FixedInvariantVector(\ZakPoint)\\
    &=
    \FixedEigenvalue
    \ScalarCorrection(\ZakPoint)
    \FixedInvariantVector(\ZakPoint)
    =
    \FixedEigenvalue\FinalZakEigenfunction(\ZakPoint)
\end{align*}
for $\ZakPoint \in \R^2$.
Since $\VectorZak : L^2(\R) \to \mathscr{H}_\Zak$ is unitary by Lemma~\ref{lem:vector_zak_unitary}, the inverse image
\begin{equation}\label{eq:gstar}
    \FinalWindow := \VectorZak^{-1}\FinalZakEigenfunction
\end{equation}
is a nonzero element of $L^2(\R)$ and satisfies
\begin{equation}\label{eq:eigen}
\FixedWeylOperator\FinalWindow=\FixedEigenvalue\FinalWindow.
\end{equation}
In view of \eqref{eq:WeylOperatorExplicit}, this proves
\eqref{eq:GeneralMainResultEquation}.

For showing that $\FinalWindow \in \mathcal{S} (\R)$, note that by the definition of the vector Zak transform $\VectorZak$ in \eqref{eq:vector Zak}, we can recover the ordinary Zak transform of $\FinalWindow$ via
\begin{align}\label{eq:recoveredZakFunction}
  (\Zak \FinalWindow) (\ZakPosition, \ZakFrequency)
  = \sqrt{2} \cdot \bigl[ (\VectorZak \FinalWindow) (\ZakPosition, 2 \ZakFrequency) \bigr]_1
  = \sqrt{2} \cdot \bigl[ \FinalZakEigenfunction (\ZakPosition, 2 \ZakFrequency) \bigr]_1
  ,
  \quad (\ZakPosition, \ZakFrequency) \in \R^2
  ,
\end{align}
which shows that $\Zak \FinalWindow : \R^2 \to \C$ is smooth.
Hence, it follows from \cite[Theorem~8.2.5]{Gro01} that ${\FinalWindow \in \CalS(\R)}$.

Lastly, we observe that a direct calculation gives
\[
\FixedWeylOperator =   \sum_{(m,n)\in \Z^2}
          \FixedWeylCoefficient_{m,n} \,
          e^{- \frac{\pi \ii}{2} (m n + 2 m \TranslationBeta + \TranslationAlpha \TranslationBeta)}
          \TimeFrequencyShift\!\left(m+\TranslationAlpha,\frac{n+\TranslationBeta}{2}\right),
\]
yielding the final claim on time-frequency shifts.
\end{proof}

Using Theorem~\ref{prop:GeneralMainResult}, we can now prove Theorem~\ref{thm:main_Weyl} and, hence, the qualitative Theorem~\ref{thm:qualitative}.

\begin{proof}[Proof of Theorem~\ref{thm:main_Weyl}]
By Remark~\ref{rem:AdmissibleSequencesExist}, there exists a finitely supported,
admissible sequence $\aaa = (\aaa_{m,n})_{m,n \in \Z} \in \C^{\Z \times \Z}$.
Thus, Theorem~\ref{prop:GeneralMainResult} yields a nonzero Schwartz function
$\FinalWindow$ and a nonzero $\FixedEigenvalue \in \C \setminus \{ 0\}$
satisfying
\[
  \FixedEigenvalue \, \FinalWindow
   =  \sum_{(m,n)\in \Z^2}
        \FixedWeylCoefficient_{m,n} \,
        \e^{\frac{\pi\ii}{2}(n\TranslationAlpha-m\TranslationBeta)} 
        \WeylShift\!\left(m+\TranslationAlpha,\frac{n+\TranslationBeta}{2}\right) \FinalWindow
  .
\]
Note that since $\FixedEigenvalue \neq 0$ and $\FinalWindow \neq 0$, this implies
that the sequence $\aaa$ is not identically zero.
Altogether, this proves the claim of Theorem~\ref{thm:main_Weyl}.
\end{proof}

\section{\texorpdfstring{The certified twelve-point construction}{The certified twelve-point construction}}
\label{sec:ExplicitAdmissibleSequence}
\label{sec:certificate}

In this section, we prove the main Theorem~\ref{thm:main}. We do this by explicitly constructing an admissible coefficient sequence by
establishing a uniform numerical certificate for the
quantity in \eqref{eq:delta-equality}.
For this, we first reduce the
problem analytically to an estimate on a finite grid of points
and then perform a certified numerical
computation using the \texttt{python-flint} interface to FLINT/Arb;
see \cite{johansson2017arb}.
Before presenting the construction, let us recall the enclosure
principle underlying the certified computation.

\subsection{Background on certified numerics}

For $d \in \N$ and $D \subset \R^d$, let \(f\colon D\to\R\) be a function and let \(X\subset D\).  We write
\[
 f(X):=\{f(x):x\in X\}.
\]
An interval \(Z\) is called an \emph{enclosure} of \(f\) on \(X\) if
\(f(X)\subseteq Z\).
A standard form of the \emph{fundamental theorem of interval analysis} states
that an interval-valued map \(F\) provides such enclosures if it agrees
with \(f\) on point intervals (i.e., $F(\{x\}) = \{f(x)\}$)
and is monotone with respect to inclusion;
see \cite[Theorem~5.1]{moore2009}. 
The
same conclusion clearly holds under the weaker enclosure condition
\begin{align}\label{eq:enclosureConditions}
 [f(x),f(x)]\subseteq F([x,x])
 \quad\text{and}\quad
 X\subseteq Y \Longrightarrow F(X)\subseteq F(Y).
\end{align}

Floating-point interval arithmetic constructs such enclosures by
propagating intervals through the operations occurring in an expression.
For exact real intervals \(X = [\underline{X}, \overline{X}]\) and \(Y = [\underline{Y}, \overline{Y}]\),
and any arithmetic operation \(\circ\in\{+,-,\cdot,/\}\), one first defines
\[
 X\circ Y:=\{x\circ y:x\in X,\ y\in Y\},
 \qquad\text{where \(0\notin Y\) is required for division.}
\]
For example,
\[
 X+Y
 =
 [\underline X+\underline Y,\overline X+\overline Y],
\]
whereas the endpoints of \(XY\) are the minimum and maximum of the
four products of endpoints.  On a computer, the endpoints are
floating-point numbers.
The lower endpoint of each result is rounded down to the closest smaller floating-point number,
and the upper endpoint is rounded up to the next larger floating point number,
so that the computed interval contains the exact interval result; see
\cite[Sections~5.3--5.5]{rump2010}.
As a result, for basic arithmetic operations, interval arithmetic satisfies \eqref{eq:enclosureConditions}.

More generally, suppose that every arithmetic operation and every
elementary function appearing in a formula for \(f\) is replaced by a
floating-point-based interval routine that encloses its exact range. 
If \(\widetilde F(X)\) denotes the resulting floating-point interval
evaluation, then repeated application of the enclosure property gives
\[
 f(X)\subseteq\widetilde F(X).
\]
Consequently, a computed inclusion such as
\(\widetilde F(X)\subset(-\infty,c)\) is a rigorous certificate that
\(f(x)<c\) for every \(x\in X\).  Standard references for interval
arithmetic and validated numerics include
\cite{moore1966,moore2009,rump2010,tucker2011}.

For the computation below, we use the ball arithmetic implemented by
Arb.  Instead of storing lower and upper endpoints, Arb represents a
real quantity by a midpoint--radius ball
\[
 [m\mathbin{\pm}r]
 =
 \{x\in\R:|x-m|\le r\}.
\]
The midpoint is an arbitrary-precision binary floating-point number,
whereas the radius is stored at a fixed lower precision and rounded
upward.  For a basic operation at a user-specified target precision,
Arb computes the midpoint at that precision, bounds the errors propagated
from the input radii, and adds a rigorous bound for the rounding error of
the midpoint.  In our computational supplement, the target precision is
set explicitly to \(256\) bits through \texttt{ctx.prec}.  Individual Arb
routines may use additional internal precision or choose algorithmic
parameters depending on the input and the target precision, but the
returned ball remains an enclosure.

Arb also provides rigorous enclosures for elementary and special
functions. In particular, the exponential and
trigonometric functions used in our certificate are evaluated with
rigorous error bounds.  Complex quantities are represented by pairs of midpoint–radius balls for their real and imaginary parts, respectively.  For details on the ball
representation, precision management, error propagation, and function
evaluation, see \cite[Sections~2, 3, and~5]{johansson2017arb}.

Arb is integrated into FLINT, which also provides efficient exact
arithmetic for integers, rational numbers, and polynomials.  Its
functionality is available in Python through \texttt{python-flint}.
Thus, all balls produced by the executed computation are rigorous
enclosures of the corresponding exact quantities.

\subsection{An explicit admissible coefficient sequence} 
The fixed set of integer indices is 

\begin{equation}\label{eq:index-set}
    \begin{split}
     \CoefficientIndexSet := \{
        &(0,-4),(0,-2),(0,-1),(0,0),(0,1),\\
        &(-1,-3),(-1,-2),(-1,-1),(-1,0),(-1,1),(-1,2)
     \}.
    \end{split}
\end{equation}
Next, define 
\begin{equation}\label{eq:amn-dyadic}
  \FixedWeylCoefficient_{m,n}
  := \begin{cases}
       \displaystyle
       \frac{\RealCoefficientNumerator_{m,n}+\ii \, \ImagCoefficientNumerator_{m,n}}{2^{60}} \in \C\setminus\{0\},
       & \text{if } (m,n) \in \CoefficientIndexSet, \\[0.2cm]
       0 & \text{if } (m,n) \notin \CoefficientIndexSet,
     \end{cases}
\end{equation}
where the integers $\RealCoefficientNumerator_{m,n},\ImagCoefficientNumerator_{m,n}$
are listed in Table~\ref{tab:coefficients} and the approximate values are given in
Table~\ref{tab:approximate-coefficients}.
Every coefficient for $(m,n) \in \CoefficientIndexSet$ is nonzero,
meaning $\supp \FixedWeylCoefficient = \CoefficientIndexSet$.
The values in Table~\ref{tab:approximate-coefficients} are included only for orientation;
all arguments use the exact dyadic coefficients from Table~\ref{tab:coefficients}.

Our main goal in this section is to prove that the coefficient sequence $\FixedWeylCoefficient$
defined above is admissible (see Definition~\ref{def:AdmissibleSequence}).

\begin{table}[ht]
    \centering
    \caption{Exact dyadic coefficients in \eqref{eq:amn-dyadic}.}
    \label{tab:coefficients}
        \begin{tabular}{@{}rrrr@{}}
        \toprule
        $m$ & $n$ & $\RealCoefficientNumerator_{m,n}$ & $\ImagCoefficientNumerator_{m,n}$\\
        \midrule
        $0$ & $-4$ & $-70 \, 931 \, 713 \, 049 \, 615 \, 624$ & $15 \, 284 \, 944 \, 543 \, 743 \, 724$\\
        $0$ & $-2$ & $121 \, 726 \, 843 \, 886 \, 294 \, 432$ & $-202 \, 158 \, 869 \, 833 \, 490 \, 592$\\
        $0$ & $-1$ & $-60 \, 331 \, 328 \, 971 \, 522 \, 640$ & $448 \, 536 \, 698 \, 320 \, 036 \, 992$\\
        $0$ & $0$ & $-116 \, 016 \, 377 \, 251 \, 609 \, 520$ & $-411 \, 853 \, 055 \, 664 \, 041 \, 984$\\
        $0$ & $1$ & $119 \, 046 \, 687 \, 232 \, 322 \, 112$ & $146 \, 355 \, 954 \, 620 \, 108 \, 544$\\
        $-1$ & $-3$ & $-3 \, 398 \, 089 \, 318 \, 489 \, 594$ & $123 \, 163 \, 038 \, 349 \, 035 \, 856$\\
        $-1$ & $-2$ & $-63 \, 437 \, 999 \, 652 \, 999 \, 096$ & $-158 \, 506 \, 055 \, 955 \, 978 \, 304$\\
        $-1$ & $-1$ & $140 \, 771 \, 717 \, 126 \, 566 \, 400$ & $139 \, 766 \, 062 \, 381 \, 477 \, 056$\\
        $-1$ & $0$ & $-182 \, 082 \, 933 \, 907 \, 714 \, 272$ & $-71 \, 363 \, 812 \, 363 \, 555 \, 152$\\
        $-1$ & $1$ & $163 \, 324 \, 206 \, 988 \, 562 \, 784$ & $-5 \, 678 \, 230 \, 433 \, 708 \, 979$\\
        $-1$ & $2$ & $-103 \, 058 \, 338 \, 818 \, 812 \, 480$ & $48 \, 900 \, 594 \, 575 \, 167 \, 184$\\
        \bottomrule
    \end{tabular}
\end{table}

\begin{table}[ht]
    \centering
    \caption{Approximate values of the real and imaginary parts of the coefficients $\FixedWeylCoefficient_{m,n}$.}
    \label{tab:approximate-coefficients}
    \begin{tabular}{r r r r}
        \hline
        $m$ & $n$
            & $\operatorname{Re} \FixedWeylCoefficient_{m,n}$
            & $\operatorname{Im} \FixedWeylCoefficient_{m,n}$ \\
        \hline
            $0$ & $-4$ & $-0.0615$ & $\phantom{-}0.0133$ \\
            $0$ & $-2$ & $\phantom{-}0.1056$ & $-0.1753$ \\
            $0$ & $-1$ & $-0.0523$ & $\phantom{-}0.3890$ \\
            $0$ &  $0$ & $-0.1006$ & $-0.3572$ \\
            $0$ &  $1$ & $\phantom{-}0.1033$ & $\phantom{-}0.1269$ \\
            $-1$ & $-3$ & $-0.0029$ & $\phantom{-}0.1068$ \\
            $-1$ & $-2$ & $-0.0550$ & $-0.1375$ \\
            $-1$ & $-1$ & $\phantom{-}0.1221$ & $\phantom{-}0.1212$ \\
            $-1$ &  $0$ & $-0.1579$ & $-0.0619$ \\
            $-1$ &  $1$ & $\phantom{-}0.1417$ & $-0.0049$ \\
            $-1$ &  $2$ & $-0.0894$ & $\phantom{-}0.0424$ \\
        \hline
    \end{tabular}
\end{table}

For the coefficient sequence just defined, let $\FixedMatrixField,\FixedCocycle : \R^2 \to \C^{2 \times 2}$ be as defined in Section~\ref{sec:AdmissibleSequences}. Moreover, recall the definition of the matrix-valued functions $\IdealMatrixField,\IdealCocycle : \R^2 \to \C^{2 \times 2}$ introduced in \eqref{eq:A0Definition} and \eqref{eq:B0Definition}. As will be explained below in the proof, the special structure of these matrix fields allows us to eliminate the $\ZakFrequency$-dependence analytically. The remaining problem is to bound a univariate function in $\ZakPosition$. This is achieved by using analytic upper bounds for the derivative of this function, thereby reducing the problem to establishing estimates on a finite grid. This is then verified using certified numerics via the Arb library.

\begin{theorem} \label{thm:certificate}
Let 
\[
  \CertifiedError_0
  :=\sup_{\ZakPoint\in[0,1]^2}\norm{\FixedMatrixField(\ZakPoint)-\IdealMatrixField(\ZakPoint)}_{\op}.
\]
It holds that
\begin{equation}\label{eq:delta-third}
 \CertifiedError_0
 \leq \frac{41629}{125000}
 < \frac{1}{3}.
\end{equation}
\end{theorem}

\begin{proof}
The proof is based on the following steps:

\begin{enumerate}
    \item writing $\FixedMatrixField(\ZakPosition,\ZakFrequency)$
          and $\IdealMatrixField(\ZakPosition,\ZakFrequency)$ for each fixed $\ZakPosition \in [0,1]$
          as Laurent series (with finitely many summands) in terms of $\LaurentVariable(\ZakFrequency) := \e^{-\pi\ii\ZakFrequency}$;

    \item using this Laurent series to derive a  bound 
          of the form
          $\| \FixedMatrixField(\ZakPosition,\ZakFrequency) - \IdealMatrixField(\ZakPosition,\ZakFrequency) \|_{\mathrm{op}}
          \leq \LaurentMajorant(\ZakPosition)$,
          where $\LaurentMajorant(\ZakPosition)$ is an explicitly given quantity only
          depending on $\ZakPosition$, independent of $\ZakFrequency$;

    \item using bounds on the derivatives $\| \partial_{\ZakPosition} \FixedMatrixField(\ZakPosition,\ZakFrequency) \|_{\mathrm{op}}$
          and $\| \partial_{\ZakPosition} \IdealMatrixField(\ZakPosition,\ZakFrequency) \|_{\mathrm{op}}$
          to show that it suffices to bound
          $\| \FixedMatrixField(\ZakPosition,\ZakFrequency) - \IdealMatrixField(\ZakPosition,\ZakFrequency) \|_{\mathrm{op}}$
          sufficiently well on finitely many grid points $\ZakPosition_j$;

    \item verifying the required bound for the grid points using certified numerics.
\end{enumerate}
The remainder of the proof provides the details on these steps.
\\~\\
\textbf{Step 1.} In this step, we write $\FixedMatrixField - \IdealMatrixField$ as a Laurent sum in terms of $\LaurentVariable(\ZakFrequency) := \e^{-\pi\ii\ZakFrequency}$. This step consists of three substeps.
\\~\\
\emph{Step 1a.} We first write  $\FixedMatrixField$ as a Laurent sum. For $r,s\in\{0,1\}$, denote $\epsilon_r := (-1)^r$, set $\sigma := \e^{\pi\ii\TranslationBeta}$, and define the coefficients of $\FixedMatrixField$ by
\begin{equation}\label{eq:fixed-laurent-coefficients}
 \FixedLaurentCoefficient_{r,s,\LaurentExponent}(\ZakPosition)
 :=
 \sum_{\substack{(m,n)\in\CoefficientIndexSet\\
                  m=\LaurentExponent\\
                  s\equiv r-n\;(\mathrm{mod}\,2)}}
 \FixedWeylCoefficient_{m,n}
 \exp\!\left[
  \pi\ii\left(nx-mr+\frac{mn}{2}\right)
 \right],
 \quad r,s \in \{0,1\}, \LaurentExponent \in \Z, \ZakPosition \in \R,
\end{equation}
noting that $\FixedLaurentCoefficient_{r,s,\LaurentExponent}(\ZakPosition) = 0$ for $\LaurentExponent \notin \{-1,0\}$.
Equations~\eqref{eq:Lmn-factorization}, \eqref{eq:Lmn-constant}, and \eqref{eq:Astar} then show
\begin{equation}
 (\FixedMatrixField(\ZakPosition,\ZakFrequency))_{r+1,s+1}
 =
 \sum_{\LaurentExponent=-1}^{0}
 \FixedLaurentCoefficient_{r,s,\LaurentExponent}(\ZakPosition)
 \LaurentVariable^\LaurentExponent(\ZakFrequency)
 , \qquad r,s \in \{0,1\}, \,\, (\ZakPosition,\ZakFrequency) \in \R^2
 ,
 \label{eq:Aast-laurent-series}
\end{equation}
which is the desired Laurent sum expression of $\FixedMatrixField$.

~\\~\\
\emph{Step 1b.}  We next write $\IdealMatrixField$ as a Laurent sum. Explicitly, our goal in this step is to show that
\begin{equation}
  \label{eq:IdealMatrixFieldLaurentSum}
  (\IdealMatrixField (\ZakPosition, \ZakFrequency))_{r+1,s+1}
  = \sum_{\LaurentExponent \in \LaurentSupport_{\ZakPosition}}
      \IdealLaurentCoefficient_{r,s,\LaurentExponent}(\ZakPosition)
      \LaurentVariable^\LaurentExponent (\ZakFrequency)
  \qquad \text{for } \ZakPosition \in [0,1], \ZakFrequency \in \R
  ,
\end{equation}
where we define
\begin{equation}
\label{eq:LaurentSupport}
\LaurentSupport_{\ZakPosition}
:= \begin{cases}
     \{-1,0,1\},& \text{if }\ZakPosition\geq\TranslationAlpha,\\
     \{-2,-1,0\},& \text{if } \ZakPosition<\TranslationAlpha,
   \end{cases}
\end{equation}
and moreover
\begin{equation}
 A_-
 :=
 \begin{cases}
     \ReferenceSine(\ZakPosition-\TranslationAlpha), & \text{if } \ZakPosition \geq \TranslationAlpha, \\
     \ReferenceSine(\ZakPosition-\TranslationAlpha+1), & \text{if  } \ZakPosition < \alpha,
 \end{cases}
 \qquad
 B_-
 :=
 \begin{cases}
     \ReferenceCosine(\ZakPosition-\TranslationAlpha), & \text{if } \ZakPosition \geq \TranslationAlpha , \\
     \ReferenceCosine(\ZakPosition-\TranslationAlpha+1) , & \text{if } \ZakPosition < \TranslationAlpha
 \end{cases}
 \label{eq:A-B-definition}
\end{equation}
as well as
\begin{equation}
 \label{eq:laurent-coefficient-d0-definition}
 \IdealLaurentCoefficient_{r,s,0}(\ZakPosition)
 := \begin{cases}
        \frac{\overline{\OffsetPhase(\ZakPosition)}}{2}
        \left(
             \ReferenceSine(\ZakPosition) A_- + \epsilon_r \epsilon_s  \ReferenceCosine(\ZakPosition) B_- \, \overline{\sigma}
        \right)
       , & \text{if } \ZakPosition \geq \TranslationAlpha, \\[0.2cm]
       \frac{\overline{\OffsetPhase(\ZakPosition)}}{2}\,
       \epsilon_r \epsilon_s \ReferenceCosine(\ZakPosition) A_- \, \overline{\sigma}
       , & \text{if } \ZakPosition < \TranslationAlpha ,
 \end{cases}
\end{equation}
and finally
\begin{equation}
 \label{eq:laurent-coefficient-d-1-definition}
 \IdealLaurentCoefficient_{r,s,-1}(\ZakPosition)
 := \begin{cases}
      \frac{\overline{\OffsetPhase(\ZakPosition)}}{2}\,
      \epsilon_s \ReferenceSine(\ZakPosition) B_- \, \overline{\sigma}
       , & \text{if } \ZakPosition \geq \TranslationAlpha, \\[0.2cm]
      \frac{\overline{\OffsetPhase(\ZakPosition)}}{2}
      \left(
       \epsilon_s \ReferenceSine(\ZakPosition) A_- \, \overline{\sigma}
       + \epsilon_r \ReferenceCosine(\ZakPosition) B_- \, \overline{\sigma}^{\,2}
      \right)
       , & \text{if } \ZakPosition < \TranslationAlpha ,
    \end{cases}
\end{equation}
and
\begin{equation}
 \label{eq:laurent-coefficient-d1-and-d-2-definition}
 \IdealLaurentCoefficient_{r,s,-2}(\ZakPosition)
 :=
 \frac{\overline{\OffsetPhase(\ZakPosition)}}{2}\,
 \ReferenceSine(\ZakPosition) B_- \, \overline{\sigma}^{\,2},
 \qquad
 \IdealLaurentCoefficient_{r,s,1}(\ZakPosition)
 := \frac{\overline{\OffsetPhase(\ZakPosition)}}{2}\,
    \epsilon_r \ReferenceCosine(\ZakPosition) \, A_-
    .
\end{equation}
Here, $\epsilon_r = (-1)^r$ and $\sigma = \e^{\pi\ii\TranslationBeta}$ are as in Step~1a.

\medskip{}

To prove \eqref{eq:IdealMatrixFieldLaurentSum}, we distinguish the two cases $\ZakPosition \geq \TranslationAlpha$
and $\ZakPosition < \TranslationAlpha$.
For both cases, we will use that \eqref{eq:chi} shows for $r\in\{0,1\}$  that the component of
$\ReferenceSection(\ZakPosition,\ZakFrequency)$ with index $r+1$ is
\begin{equation}\label{eq:chi-component-laurent}
 (\ReferenceSection(\ZakPosition,\ZakFrequency))_{r+1}
 =\frac{1}{\sqrt2}\Bigl(\ReferenceSine(\ZakPosition) + \epsilon_r \ReferenceCosine(\ZakPosition)\LaurentVariable (\ZakFrequency) \Bigr),
\end{equation}
for $r\in\{0,1\}, \,\, \ZakPosition \in [0,1]$.
\\~\\
\emph{Case 1 ($\ZakPosition \geq \TranslationAlpha$):}
In this case, $\ZakPosition - \TranslationAlpha \in [0,1]$.
Hence, \eqref{eq:chi-component-laurent} shows that
the component of $\ReferenceSection(\TranslateZak \ZakPoint)$ with index $s+1$ is
\begin{equation}\label{eq:translated-chi-component-positive}
\begin{aligned}
 (\ReferenceSection(\TranslateZak\ZakPoint) )_{s+1}
 &= (\ReferenceSection(\ZakPosition-\TranslationAlpha,\ZakFrequency-\TranslationBeta) )_{s+1} \\
 &= \frac{1}{\sqrt{2}}
    \bigl(
      \ReferenceSine(\ZakPosition - \TranslationAlpha)
      + \epsilon_s \ReferenceCosine(\ZakPosition - \TranslationAlpha)
        \LaurentVariable(\ZakFrequency - \TranslationBeta)
    \bigr) \\
 &= \frac{1}{\sqrt2}
    \bigl(A_- + \epsilon_s B_ - \sigma \LaurentVariable (\ZakFrequency) \bigr)
   .
\end{aligned}
\end{equation}
Expanding the outer product
$\IdealMatrixField =\ReferenceSection\TranslatedReferenceSection^*$
entry by entry gives, by \eqref{eq:chi-component-laurent} and \eqref{eq:def_w},
\begin{equation}
\begin{aligned}
 (\IdealMatrixField(\ZakPosition,\ZakFrequency))_{r+1,s+1}
 &= (\ReferenceSection(\ZakPosition,\ZakFrequency) )_{r+1}
 \,\overline{
   (\TranslatedReferenceSection(\ZakPosition,\ZakFrequency) )_{s+1}
 }\\
 &=\frac{\overline{\OffsetPhase(\ZakPosition)}}2
   \Bigl(\ReferenceSine(\ZakPosition) + \epsilon_r \ReferenceCosine(\ZakPosition) \LaurentVariable(\ZakFrequency)\Bigr)
   \Bigl(A_- + \epsilon_s B_-\,\overline{\sigma}\,\LaurentVariable^{-1}(\ZakFrequency)\Bigr).
\end{aligned}
\end{equation}
Expanding this and recalling the definitions of the coefficients
$\IdealLaurentCoefficient_{r,s,\LaurentExponent}$ thus proves
\eqref{eq:IdealMatrixFieldLaurentSum} for the case $\ZakPosition \geq \TranslationAlpha$.
\\~\\
\emph{Case 2 ($\ZakPosition < \TranslationAlpha$):}
In this case, $\ZakPosition - \TranslationAlpha + 1 \in [0,1]$ since $\TranslationAlpha < 1$.
Now, \eqref{eq:x-sewing} shows
\begin{align*}
  \ReferenceSection (\TranslateZak \ZakPoint)
  &= \ReferenceSection( \ZakPosition - \TranslationAlpha, \ZakFrequency - \TranslationBeta) \\
  &= \FirstSewingMatrix(\ZakFrequency - \TranslationBeta)^\ast
     \ReferenceSection( \ZakPosition - \TranslationAlpha + 1, \ZakFrequency - \TranslationBeta) \\
  &= e^{- \pi i (\ZakFrequency - \TranslationBeta)}
     \begin{pmatrix}
         1 & 0 \\ 0 & -1
     \end{pmatrix}
     \ReferenceSection( \ZakPosition - \TranslationAlpha + 1, \ZakFrequency - \TranslationBeta)
\end{align*}
and hence \eqref{eq:chi-component-laurent} shows
\begin{equation}
  \label{eq:TranslatedReferenceSectionLaurentSum}
  \begin{aligned}
    (\ReferenceSection (\TranslateZak \ZakPoint))_{s+1}
    &= e^{\pi i \TranslationBeta} \,
       \LaurentVariable(\ZakFrequency) \,
       \epsilon_s \,
       \bigl(\ReferenceSection( \ZakPosition - \TranslationAlpha + 1, \ZakFrequency - \TranslationBeta)\bigr)_{s+1} \\
    &= \sigma
       \LaurentVariable(\ZakFrequency)
       \frac{1}{\sqrt{2}}
       \Bigl(
         \ReferenceSine (\ZakPosition - \TranslationAlpha + 1)
         + \epsilon_s \,  \ReferenceCosine (\ZakPosition - \TranslationAlpha + 1)
           \LaurentVariable(\ZakFrequency - \TranslationBeta)
       \Bigr) \\
    &= \frac{1}{\sqrt{2}}
       \Bigl(
         \epsilon_s \, A_- \, \sigma \, \LaurentVariable(\ZakFrequency)
         + \sigma^2 \ B_- \, \LaurentVariable^2 (\ZakFrequency)
       \Bigr)
    .
  \end{aligned}
\end{equation}
Consequently, recalling the definitions of the coefficients
$\IdealLaurentCoefficient_{r,s,\LaurentExponent}$, we see using \eqref{eq:A0Definition} that
\begin{equation}
\begin{aligned}
 (\IdealMatrixField(\ZakPosition,\ZakFrequency) )_{r+1,s+1}
 &=
   (\ReferenceSection(\ZakPosition,\ZakFrequency))_{r+1}
   \overline{(\TranslatedReferenceSection (\ZakPosition,\ZakFrequency))_{s+1}} \\
 &=
   \frac{\overline{\OffsetPhase(\ZakPosition)}}{2}
   \Bigl(
     \ReferenceSine(\ZakPosition) + \epsilon_r \, \ReferenceCosine(\ZakPosition) \LaurentVariable(\ZakFrequency)
   \Bigr)
   \overline{
       \Bigl(
         \epsilon_s \, A_- \, \sigma \, \LaurentVariable (\ZakFrequency)
         + B_- \, \sigma^2 \, \LaurentVariable^2 (\ZakFrequency)
       \Bigr)
   } \\
 &= \sum_{\LaurentExponent=-2}^{0}
      \IdealLaurentCoefficient_{r,s,\LaurentExponent}(\ZakPosition)\LaurentVariable^\LaurentExponent (\ZakFrequency),
\end{aligned}
\end{equation}
for $\ZakPosition \in [0,1]$ and $\ZakFrequency \in \R$,
which proves \eqref{eq:IdealMatrixFieldLaurentSum} for the case $\ZakPosition < \TranslationAlpha$. Here, the second step used \eqref{eq:def_w}, \eqref{eq:TranslatedReferenceSectionLaurentSum} and \eqref{eq:chi-component-laurent}.
\\~\\
\emph{Step 1c.} Lastly, we write $\FixedMatrixField - \IdealMatrixField$ as a Laurent sum.
Overall, setting
\begin{equation}
 \MatrixErrorField(\ZakPosition,\ZakFrequency)
 :=\FixedMatrixField(\ZakPosition,\ZakFrequency)-\IdealMatrixField(\ZakPosition,\ZakFrequency)
\end{equation}
and defining
\begin{equation}\label{eq:error-laurent-coefficients}
 \ErrorLaurentCoefficient_{r,s,\LaurentExponent}(\ZakPosition)
 :=
 \FixedLaurentCoefficient_{r,s,\LaurentExponent}(\ZakPosition)
 -\IdealLaurentCoefficient_{r,s,\LaurentExponent}(\ZakPosition),
\end{equation}
for $\ZakPosition \in [0,1]$,
it follows from \eqref{eq:Aast-laurent-series} and \eqref{eq:IdealMatrixFieldLaurentSum} that
\begin{equation}\label{eq:laurent-error}
 (\MatrixErrorField(\ZakPosition,\ZakFrequency) )_{r+1,s+1}
 = \sum_{\LaurentExponent\in\LaurentSupport_{\ZakPosition}}
     \ErrorLaurentCoefficient_{r,s,\LaurentExponent}(\ZakPosition)
     \LaurentVariable^\LaurentExponent (\ZakFrequency)
\end{equation}
for all $r,s \in \{0,1\}$, all $\ZakPosition \in [0,1]$, and all $\ZakFrequency \in \R$.
\\~\\
\textbf{Step 2.} In this step,  we prove the bound 
\(
  \|\FixedMatrixField (\ZakPosition, \ZakFrequency) - \IdealMatrixField(\ZakPosition, \ZakFrequency) \|_{\op}
  \leq \LaurentMajorant(\ZakPosition)
\).
Bounding the operator norm by the Frobenius norm,
using \eqref{eq:laurent-error}, and applying the triangle inequality,
we obtain, for every $\ZakPosition \in [0,1]$ and $\ZakFrequency \in \R$, that
\begin{equation}\label{eq:omega-triangle}
 \norm{\MatrixErrorField(\ZakPosition,\ZakFrequency)}_{\op}
 \leq \norm{\MatrixErrorField(\ZakPosition,\ZakFrequency)}_{\mathrm F}
 \leq \LaurentMajorant(\ZakPosition):=
 \left(
  \sum_{r,s=0}^1
  \left(
    \sum_{\LaurentExponent\in\LaurentSupport_{\ZakPosition}}
    |\ErrorLaurentCoefficient_{r,s,\LaurentExponent}(\ZakPosition)|
  \right)^2
 \right)^{1/2}.
\end{equation}
Note that \eqref{eq:omega-triangle} is independent of
$\ZakFrequency$, and the proof is complete once the one-variable
function $\LaurentMajorant$ is bounded by the fraction in
\eqref{eq:delta-third}.
\\~\\
\textbf{Step 3.} We next obtain a bound for $\|\partial_\ZakPosition (\FixedMatrixField - \IdealMatrixField)(\ZakPosition, \ZakFrequency) \|_{\op}$.
We derive this bound by considering $\FixedMatrixField$ and $\IdealMatrixField$ individually
and then applying the triangle inequality. We split this step into three substeps.
\\~\\
\emph{Step 3a.} We start by bounding $\|\partial_\ZakPosition \FixedMatrixField (\ZakPosition, \ZakFrequency) \|_{\op}$.
From \eqref{eq:Astar} and \eqref{eq:L-derivatives} and the elementary estimate
$|\FixedWeylCoefficient_{m,n}| \leq |\Re \FixedWeylCoefficient_{m,n}|+|\Im \FixedWeylCoefficient_{m,n}|$, we get
\begin{equation}\label{eq:cand-dx}
\begin{aligned}
 \sup_{\ZakPoint}
 \norm{
   \partial_{\ZakPosition}\FixedMatrixField(\ZakPoint)
 }_{\op}
 &\leq \pi
       \sum_{(m,n)\in\CoefficientIndexSet}
         |n|
         (|\Re \FixedWeylCoefficient_{m,n}|+|\Im \FixedWeylCoefficient_{m,n}|) \\
 &=   \frac{\pi}{2^{60}}
      \sum_{(m,n)\in\CoefficientIndexSet}
        |n|
        (|\RealCoefficientNumerator_{m,n}|+|\ImagCoefficientNumerator_{m,n}|)
        ,
\end{aligned}
\end{equation}
which is the required bound.
\\~\\
\emph{Step 3b.} We next estimate $\|\partial_\ZakPosition \IdealMatrixField (\ZakPosition, \ZakFrequency) \|_{\op}$.
First, recall from \eqref{eq:def_w} that
$\TranslatedReferenceSection(\ZakPoint) = \OffsetPhase(\ZakPosition) \ReferenceSection(\TranslateZak\ZakPoint) \in \C^2$
and from \eqref{eq:EtaDefinition} that $\OffsetPhase(\ZakPosition) = \e^{\pi\ii\TranslationBeta (\ZakPosition-\TranslationAlpha/2)}$.
Using the product rule, we thus get
\begin{equation}
\begin{aligned}
  \| \partial_\ZakPosition \TranslatedReferenceSection (\ZakPosition, \ZakFrequency) \|
  &= \|
       \partial_\ZakPosition
       (\OffsetPhase(\ZakPosition) \ReferenceSection (\ZakPosition - \TranslationAlpha, \ZakFrequency - \TranslationBeta))
     \| \\
  &= \|
       \OffsetPhase'(\ZakPosition) \ReferenceSection(\TranslateZak \ZakPoint)
       + \OffsetPhase(\ZakPosition) (\partial_\ZakPosition \ReferenceSection) (\ZakPosition - \TranslationAlpha, \ZakFrequency - \TranslationBeta)
     \| \\
  &\leq |\OffsetPhase'(\ZakPosition)|
        \cdot \| \ReferenceSection (\TranslateZak \ZakPoint) \|
        + |\OffsetPhase(\ZakPosition)|
          \cdot \| (\partial_\ZakPosition \ReferenceSection) (\ZakPosition - \TranslationAlpha, \ZakFrequency - \TranslationBeta) \| \\
  &\leq \pi \TranslationBeta + \pi
   =    \pi \cdot (1 + \TranslationBeta)
   ,
\end{aligned}
\label{eq:w-derivative-estimate}
\end{equation}
where the last step used that $|\OffsetPhase(\ZakPosition)| = 1$,
$|\OffsetPhase'(\ZakPosition)| = \pi \TranslationBeta$
and moreover that $\|\ReferenceSection (\TranslateZak \ZakPoint)\| = 1$ by Lemma~\ref{lem:unit}
and 
\(
  \|
    (\partial_\ZakPosition \ReferenceSection) (\ZakPosition - \TranslationAlpha, \ZakFrequency - \TranslationBeta)
  \|
  \leq \pi
\)
by Lemma~\ref{lem:ChiDerivativeEstimate}.

Next, recall from \eqref{eq:A0Definition} that
$\IdealMatrixField(\ZakPoint) = \ReferenceSection(\ZakPoint) \TranslatedReferenceSection(\ZakPoint)^* \in \C^{2 \times 2}$.
Thus, using the product rule and the elementary estimate
$\| u v^\ast \|_{\op} \leq \| u \| \cdot \| v \|$ for $u,v \in \C^2$, we see
\begin{equation}
  \label{eq:A0-derivative-estimate}
  \begin{aligned}
      \| \partial_\ZakPosition \IdealMatrixField (\ZakPosition, \ZakFrequency) \|_{\op}
      &= \|
           \partial_\ZakPosition (\ReferenceSection(\ZakPosition, \ZakFrequency) \TranslatedReferenceSection(\ZakPosition, \ZakFrequency)^\ast)
         \|_{\op} \\
      &= \|
           (\partial_\ZakPosition \ReferenceSection)(\ZakPosition, \ZakFrequency)
           \cdot \TranslatedReferenceSection(\ZakPosition, \ZakFrequency)^\ast
           + \ReferenceSection(\ZakPosition, \ZakFrequency)
             \cdot ((\partial_\ZakPosition \TranslatedReferenceSection)(\ZakPosition, \ZakFrequency))^\ast
         \|_{\op} \\
      &\leq \|
              (\partial_\ZakPosition \ReferenceSection)(\ZakPosition, \ZakFrequency)
              \cdot \TranslatedReferenceSection(\ZakPosition, \ZakFrequency)^\ast
            \|_{\op}
            + \| \ReferenceSection(\ZakPosition, \ZakFrequency)
                  \cdot ((\partial_\ZakPosition \TranslatedReferenceSection)(\ZakPosition, \ZakFrequency))^\ast
              \|_{\op} \\
      &\leq \| (\partial_\ZakPosition \ReferenceSection)(\ZakPosition, \ZakFrequency) \|
              \cdot \| \TranslatedReferenceSection(\ZakPosition, \ZakFrequency) \|
            + \| \ReferenceSection(\ZakPosition, \ZakFrequency) \|
                  \cdot \| (\partial_\ZakPosition \TranslatedReferenceSection)(\ZakPosition, \ZakFrequency) \| \\
      &\leq \pi \cdot 1 + 1 \cdot \pi \cdot (1 + \TranslationBeta) \\
      &=    \pi \cdot (2 + \TranslationBeta)
      ,
  \end{aligned}
\end{equation}
where the penultimate step combines \eqref{eq:w-derivative-estimate} and the observation $\| \TranslatedReferenceSection(\ZakPoint) \| = \| \ReferenceSection (\TranslateZak \ZakPoint) \|$, together with Lemmas \ref{lem:unit} and \ref{lem:ChiDerivativeEstimate}.
\\~\\
\emph{Step 3c.} Lastly, we estimate $\|\partial_\ZakPosition (\FixedMatrixField - \IdealMatrixField)(\ZakPosition, \ZakFrequency) \|_{\op}$.
By combining the bounds from the previous two sub-steps using the triangle inequality
and using the  rational inequalities $\pi<355/113$ and $\TranslationBeta<3/5$,
we obtain
\begin{align*}
  \|\partial_\ZakPosition (\FixedMatrixField - \IdealMatrixField)(\ZakPosition, \ZakFrequency) \|_{\op}
  &\leq  \| \partial_\ZakPosition \FixedMatrixField (\ZakPosition, \ZakFrequency) \|_{\op}
        + \| \partial_\ZakPosition \IdealMatrixField(\ZakPosition, \ZakFrequency) \|_{\op} \\
  &\leq  \frac{\pi}{2^{60}}
         \sum_{(m,n)\in\CoefficientIndexSet}
           |n|
           (|\RealCoefficientNumerator_{m,n}|+|\ImagCoefficientNumerator_{m,n}|)
        + \pi \cdot (2 + \beta) \\
  &\leq \frac{355}{113}
        \cdot
        \Bigl(
          2
          + \frac{3}{5}
         + \frac{1}{2^{60}} \sum_{(m,n)\in\CoefficientIndexSet}
             |n|
             (|\RealCoefficientNumerator_{m,n}|+|\ImagCoefficientNumerator_{m,n}|)
        \Bigr)
   =: \circledast
   ,
\end{align*}
for all $\ZakPosition \in [0,1]$ and $\ZakFrequency \in \R$.
The right-hand side of this inequality only involves rational numbers.

The supplementary code
(specifically the function \lstinline[language=Python]|exact_x_derivative_bound()|)
uses exact rational number arithmetic to compute
\begin{equation}
\label{eq:derivative-numbers}
\begin{aligned}
  \DerivativeBound
  &:= \sup_{\ZakPosition \in [0,1], \ZakFrequency \in \R}
       \|\partial_\ZakPosition \MatrixErrorField(\ZakPosition, \ZakFrequency) \|_{\op} \\
  & = \sup_{\ZakPosition \in [0,1], \ZakFrequency \in \R}
       \|\partial_\ZakPosition (\FixedMatrixField - \IdealMatrixField)(\ZakPosition, \ZakFrequency) \|_{\op} \\
  &\leq \circledast
   =  \frac{2 \, 251 \, 244 \, 646 \, 990 \, 129 \, 974 \, 723}{130 \, 280 \, 130 \, 020 \, 573 \, 708 \, 288}
   <  17.280031.
\end{aligned}
\end{equation}
\\~\\
\textbf{Step 4.} We complete the proof with a certified numerical computation. 
For this, let $N := 2048$ and set
\begin{equation}
 \ZakPosition_j := \frac{2j+1}{2N},
 \qquad 0 \leq j<N.
 \label{eq:GridPoints}
\end{equation}
At these $2048$ centers, the supplementary Python script evaluates all
coefficients $\ErrorLaurentCoefficient_{r,s,\LaurentExponent}(\ZakPosition_j)$
in \eqref{eq:laurent-error} using $256$-bit Arb real and
complex balls and uses that to produce a certified bound for
$\LaurentMajorant(\ZakPosition_j)$.

Specifically, the supplementary script certifies that
\begin{equation}\label{eq:center-bound}
  \LaurentMajorant(\ZakPosition_j)
  <     0.328813
  \quad \text{for all } 0 \leq j < N
  .
\end{equation}

Now, given an arbitrary $(\ZakPosition,\ZakFrequency) \in [0,1] \times \R$,
we can choose $0 \leq j < N$ with $|\ZakPosition-\ZakPosition_j|\leq \frac{1}{2N}$.
Equations \eqref{eq:omega-triangle} and
\eqref{eq:derivative-numbers} yield with the mean value theorem that
\begin{equation}
 \norm{\MatrixErrorField(\ZakPosition,\ZakFrequency)}_{\op}
 \leq  \norm{\MatrixErrorField(\ZakPosition_j,\ZakFrequency)}_{\op}
       + \frac{D_\ast}{2N}
 \leq
 \LaurentMajorant(\ZakPosition_j)
 +\frac{\DerivativeBound}{2N}.
\end{equation}
Note that the mean-value theorem is applicable, since $\FixedMatrixField, \IdealMatrixField : \R^2 \to \C^{2 \times 2}$
are smooth.

The supplementary Python script also certifies that
\begin{equation}\label{eq:interpolation}
 \frac{\DerivativeBound}{2N}
 <    0.004219.
\end{equation}
Adding the estimates \eqref{eq:center-bound} and \eqref{eq:interpolation} gives
\begin{equation}
\label{eq:finalNumericalBound}
 \CertifiedError_0
 < 0.328813 + 0.004219
 = 0.333032
 = \frac{41629}{125000}
 < \frac{1}{3},
\end{equation}
which proves \eqref{eq:delta-third}.
\end{proof}

\begin{remark}[Reproducibility]
The computational supplement consists of the single standalone script
\nolinkurl{hrt_fixed_12_point_arb_1d_audit.py}.  It was run with
python-flint $0.8.0$, FLINT $3.3.1$, and $256$-bit Arb precision.
The script is deliberately repetitive: it writes the eleven dyadic
coefficients directly in its source, implements
$\FixedLaurentCoefficient_{r,s,\LaurentExponent}$ and the two
piecewise formulas for
$\IdealLaurentCoefficient_{r,s,\LaurentExponent}$
as separate literal functions, then forms
$\ErrorLaurentCoefficient_{r,s,\LaurentExponent}
 =\FixedLaurentCoefficient_{r,s,\LaurentExponent}
  -\IdealLaurentCoefficient_{r,s,\LaurentExponent}$
before applying
\eqref{eq:omega-triangle}. The script verifies 
\eqref{eq:center-bound}, \eqref{eq:derivative-numbers}, and \eqref{eq:interpolation} and the final bound \eqref{eq:finalNumericalBound}.
The certificate relies on the documented rigorous
ball-arithmetic guarantees of FLINT/Arb.
\end{remark}

Lastly, we prove the following corollary.

\begin{corollary}\label{cor:GlobalErrorBound}
The coefficient sequence $\aaa \in \C^{\Z \times \Z}$ defined in \eqref{eq:amn-dyadic}
is admissible in the sense of Definition~\ref{def:AdmissibleSequence}.
\end{corollary}

\begin{proof}
Recall from Definition~\ref{def:AdmissibleSequence} that 
\(
  \CertifiedError(\aaa)
  = \sup_{\ZakPoint \in \R^2}
      \norm{\FixedCocycle(\ZakPoint)-\IdealCocycle(\ZakPoint)}_{\op}
\)
and that admissibility of $\aaa$ means that $\CertifiedError(\aaa) < \frac{1}{3}$.
We will show that
\begin{equation}
  \CertifiedError (\aaa) \leq \CertifiedError_0
  ,
  \label{eq:ExplicitSequenceDeltaGlobalization}
\end{equation}
with $\CertifiedError_0$ as in Theorem~\ref{thm:certificate}.
The admissibility of $\aaa$ then follows from that theorem,
since it shows $\CertifiedError_0 < \frac{1}{3}$.

As the first step for proving \eqref{eq:ExplicitSequenceDeltaGlobalization},
note that the pointwise identity \eqref{eq:delta-equality-pointwise} shows that
$\CertifiedError(\aaa) = \sup_{\ZakPoint \in \R^2} \norm{\FixedMatrixField(\ZakPoint)-\IdealMatrixField(\ZakPoint)}_{\op}$.

For brevity, set
\(
 \MatrixErrorField(\ZakPoint)
 :=
 \FixedMatrixField(\ZakPoint)-\IdealMatrixField(\ZakPoint) \in \C^{2 \times 2}
\)
for $\ZakPoint \in \R^2$.
Both $\FixedMatrixField$ and $\IdealMatrixField$
satisfy the quasi-periodicity relations \eqref{eq:sewing}.  
For $\IdealMatrixField$, this follows from its definition \eqref{eq:A0Definition}
and the vector Zak sewing formulas \eqref{eq:x-sewing} and \eqref{eq:omega-sewing}
for $\ReferenceSection$ and $\TranslatedReferenceSection$. 
Here, $\ReferenceSection$ satisfies the quasi-periodicity relations since
$\ReferenceSection = \VectorZak \ReferenceWindow$ (see \eqref{eq:explicit_vector_zak}),
and $\TranslatedReferenceSection$ satisfies the quasi-periodicity relations by Lemma~\ref{lem:w_vectorZak}.
For $\FixedMatrixField$,
first recall from Lemma~\ref{lem:sewing_Weyl_matrix}
that each matrix function $\LatticeWeylMatrix_{m,n}$ satisfies \eqref{eq:sewing}.
In view of \eqref{eq:Astar}, this shows that
$\FixedMatrixField$ and hence $\MatrixErrorField$ also satisfy \eqref{eq:sewing}.

Now let $\ZakPoint\in\R^2$. 
Choose $\ZakPoint_0\in[0,1)^2$ and
$k\in\Z^2$ such that $\ZakPoint=\ZakPoint_0+k$. 
Repeated application
of the two identities in \eqref{eq:sewing} yields a unitary matrix
$U_{\ZakPoint}$ such that
\[
 \MatrixErrorField(\ZakPoint)
 =
 U_{\ZakPoint}\MatrixErrorField(\ZakPoint_0)U_{\ZakPoint}^{-1}.
\]
The operator norm is invariant under unitary conjugation.
Hence, we see with $\CertifiedError_0$ as in Theorem~\ref{thm:certificate} that
\[
 \norm{\MatrixErrorField(\ZakPoint)}_{\op}
 =
 \norm{\MatrixErrorField(\ZakPoint_0)}_{\op}
 \leq \CertifiedError_0
 .
\] 
Taking the supremum over $\ZakPoint\in\R^2$ gives
\[
  \CertifiedError (\aaa)
  = \sup_{\ZakPoint\in\R^2}
      \norm{\FixedMatrixField(\ZakPoint)-\IdealMatrixField(\ZakPoint)}_{\op}
  \leq \sup_{\ZakPoint\in[0,1]^2}
  \norm{\FixedMatrixField(\ZakPoint)-\IdealMatrixField(\ZakPoint)}_{\op}
  \leq  \CertifiedError_0.
\]
As noted at the beginning of the proof, this proves admissibility of $\aaa$.
\end{proof}

We can now finally prove the main Theorem~\ref{thm:main}.

\begin{proof}[Proof of Theorem~\ref{thm:main}]
Consider the sequence $\aaa$ defined in \eqref{eq:amn-dyadic},
with the set $\CoefficientIndexSet = \supp \aaa$ defined in \eqref{eq:index-set}.
By Corollary~\ref{cor:GlobalErrorBound}, the sequence $\aaa$ is admissible.
Thus, Theorem~\ref{prop:GeneralMainResult} yields a nonzero Schwartz function
$\FinalWindow$ and a nonzero $\FixedEigenvalue \in \C \setminus \{ 0\}$
satisfying
\begin{align*}
  \FixedEigenvalue \, \TimeFrequencyShift(0,0) \FinalWindow
  = \FixedEigenvalue \, \FinalWindow
  &=  \sum_{(m,n)\in \Z^2}
        \aaa_{m,n} \,
        e^{- \frac{\pi \ii}{2} (m n + 2 m \TranslationBeta + \TranslationAlpha \TranslationBeta)}
        \TimeFrequencyShift \bigl( m + \TranslationAlpha, \tfrac{n + \TranslationBeta}{2} \bigr) \FinalWindow \\
  &=  \sum_{(m,n)\in \CoefficientIndexSet}
        \aaa_{m,n} \,
        e^{- \frac{\pi \ii}{2} (m n + 2 m \TranslationBeta + \TranslationAlpha \TranslationBeta)}
        \TimeFrequencyShift \bigl( m + \TranslationAlpha, \tfrac{n + \TranslationBeta}{2} \bigr) \FinalWindow
  .
\end{align*}
The set $\CoefficientIndexSet$ has exactly $11$ elements and the points $( m + \TranslationAlpha, (n + \TranslationBeta)/2 )$, $(m,n) \in \CoefficientIndexSet$ are pairwise distinct and none equals $(0,0)$; moreover, $\aaa_{m,n} \neq 0$ for every $(m,n) \in \CoefficientIndexSet$. We have therefore identified a linear dependence among $12$ time-frequency shifts $\FinalWindow$, in which all coefficients are nonzero. This easily implies the claim of Theorem~\ref{thm:main}.
\end{proof}

\section*{Acknowledgements}
For M.F. and J.v.V., this research was funded in whole or in part by the Austrian Science Fund (FWF): [10.55776/PAT5102224] and [10.55776/PAT2545623].
F.V. acknowledges support by the German Science Foundation (DFG) in the context of the Emmy Noether junior research group VO 2594/1-1.
F.V. acknowledges support by the Hightech Agenda Bavaria. 

\appendix

\section{A Fa\`a di Bruno-type formula}
\label{sec:FaaDiBrunoByHand}

In this appendix, we provide a proof for the formula \eqref{eq:highest-derivative-recursion} in the proof of Theorem~\ref{lem:regularity}. Essentially, this formula is a (less precise) special case of the \emph{multivariate Faà di Bruno formula}; see e.g.\  \cite{FaaDiBrunoMultivariate}. However, since we consider a mix of partial derivatives in the sense of real variables and of partial derivatives in the sense of holomorphic functions (see below), we could not locate a version of \emph{Faà di Bruno's formula} that strictly applies in our setting. For this reason, and to make the paper more self-contained, we provide a proof.

\subsection{Assumptions and notation}
\label{sub:FaaDiBrunoNotation}

Let $V \subset \C$ be open.
Let $\ProjectiveMap \in C^\infty(\R^2 \times V;\C)$,
and suppose that 
$\ProjectiveCoordinate\mapsto \ProjectiveMap(\ZakPoint,\ProjectiveCoordinate)$
is holomorphic on $V$ for every $\ZakPoint \in \R^2$.
For $\alpha\in\N_0^2$ and $k\in\N_0$, write
\[
 \ProjectiveMap_{\alpha,k}
 :=
 \partial_{\ZakPoint}^{\alpha}
 \partial_{\ProjectiveCoordinate}^{k} \ProjectiveMap
 ,
\]
where the partial derivatives with respect to $\ProjectiveCoordinate \in \C$ are taken
in the sense of complex variables and holomorphic functions,
while the partial derivatives with respect to $\z \in \R^2$ are
taken in the sense of real variables.

For $\alpha = (\alpha_1, \alpha_2) , \gamma = (\gamma_1, \gamma_2) \in \N_0^2$,
we use the following well-known multi-index notation:
\[
  |\alpha| := \alpha_1 + \alpha_2
  \quad \text{and} \quad
  \alpha \leq \gamma \,\, :\Leftrightarrow  \,\, \alpha_i \leq \gamma_i \text{ for } i \in \{1,2\}
  .
\]
Moreover, we write $\alpha < \gamma$ if and only if $\alpha \leq \gamma$ and $\alpha \neq \gamma$.

For a tuple $\BoldEta = (\eta^1,\dots,\eta^k)$ with $\eta^1,\dots,\eta^k \in \N_0^2$, define
\[
  \overline{\BoldEta} := \eta^1 + \cdots + \eta^k \in \N_0^2
  \quad \text{and} \quad
  k(\BoldEta) := k \in \N
\]
  In what follows, the empty tuple is denoted by $()$.
If $\BoldEta = ()$, we set $\overline{\BoldEta} := 0 \in \N_0^2$
and $k(\BoldEta) := 0$.

Using this notation, for $\gamma\in\N_0^2 \setminus \{ 0 \}$, define
\[
  \FrakI_\gamma
  := \{ () \}
     \cup \Bigl\{
            \BoldEta = (\eta^1,\dots,\eta^k)
            \,\,:\,\,
            k(\BoldEta) \in \{1, \dots, |\gamma|\} ,
            \eta^i \in \N_0^2 \setminus \{0\},
            \text{ and } \overline{\BoldEta} \leq \gamma
          \Bigr\}
  .
\]

Now, given $g \in C^\infty (\R^2 ; V)$, define
\[
  \widetilde{g} : \quad
  \R^2 \to V, \quad
  \widetilde {g} (\z) := g(\TranslateZak \z)
\]
and
\[
  G : \quad
  \R^2 \to \C, \quad
  G(\z) := \ProjectiveMap (\z , \widetilde{g} (\z))
  .
\]
Additionally, given $\BoldEta \in \FrakI_\gamma$, define
\[
  G_{\gamma, \BoldEta} [g] : \quad
  \R^2 \to \C, \quad
  G_{\gamma, \BoldEta} [g] (\z)
  := \ProjectiveMap_{\gamma - \overline{\BoldEta}, k(\BoldEta)} (\z, \widetilde{g}(\z))
     \cdot \prod_{j=1}^{k(\BoldEta)} (\partial^{\eta^j} g) (\TranslateZak \z)
  ,
\]
where the empty product $\prod_{j=1}^0 a_j$ is interpreted as $1$.

Finally, for $\Omega \subset \R^2$, define
\[
  \Omega_{\TorusShiftVector}
  := \Omega - \TorusShiftVector
  = \{ \z - \TorusShiftVector \,\,:\,\, \z \in \Omega \}
  \subset \R^2
  .
\]

\subsection{A formula for higher derivatives of \texorpdfstring{$G(\z)$}{G(z)}}
\label{sub:FaaDiBrunoResult}

\begin{lemma}\label{lem:FaaDiBrunoResult}
    With notation and assumptions as in Section~\ref{sub:FaaDiBrunoNotation}, for every
    $\gamma \in \N_0^2 \setminus \{ 0 \}$,  there exist coefficients
    $c_{\gamma, \BoldEta} \in \N_0$, $\BoldEta \in \FrakI_\gamma$, independent of $g$, such that
    \begin{equation}
      \partial^\gamma G (\z)
      = \sum_{\BoldEta \in \FrakI_\gamma} c_{\gamma, \BoldEta} G_{\gamma,\BoldEta}[g] (\z)
      , \qquad  \z \in \R^2
      .
      \label{eq:FaaDiBruno}
    \end{equation}
    Moreover,
    \begin{enumerate}[label=(\roman*)]
        \item $c_{\gamma, (\gamma)} = 1$, and
        \item for $\BoldEta \in \FrakI_\gamma \setminus \{ (\gamma) \}$, each partial derivative
              $(\partial^{\eta^j} g) (\TranslateZak \z)$ occurring in the definition of
              $G_{\gamma, \BoldEta}$ has order $|\eta^j| \leq |\gamma| - 1$.
    \end{enumerate}
\end{lemma}

\begin{proof}
We split the proofs into three steps.
\\~\\
\textbf{Step 1.}
Let
\[
  h = h(\z, \vv) \in C^\infty(\R^2 \times V;\C),
\]
and assume that $V \ni \vv \mapsto h(\z, \vv)$ is holomorphic for each $\z \in \R^2$.
Let us identify $\C \cong \R^2$.
Then, it is well-known (see e.g.\ \cite[Pages 11--12]{SteinComplexAnalysis})
that as a consequence of $h(\z, \cdot)$ being holomorphic,
the partial derivatives $\partial_{\Re \vv} h, \partial_{\Im \vv} h$ in the sense of real variables
and the partial derivative $\partial_{\vv} h$ in the sense of complex variables are related via
\[
  (\partial_\vv h) (\z, \vv)
  = (\partial_{\Re \vv} h) (\z, \vv)
  = -\ii \cdot (\partial_{\Im \vv} h)(\z, \vv)
  .
\]
Thus, writing $g = g_1 + \ii \, g_2$ with real-valued functions $g_1, g_2 \in C^\infty (\R^2)$,
we can apply the usual (real-variable) multivariate chain rule to get
\begin{equation}
\begin{aligned}
  &\partial_{\z_i} [h(\z, \widetilde{g}(\z))] \\
  &= (\partial_{\z_i} h) (\z, \widetilde{g}(\z))
    + (\partial_{\Re \vv} h) (\z, \widetilde{g}(\z)) \cdot (\partial_{\z_i} g_1)(\TranslateZak \z)
    + (\partial_{\Im \vv} h) (\z, \widetilde{g}(\z)) \cdot (\partial_{\z_i} g_2)(\TranslateZak \z) \\
  &= (\partial_{\z_i} h) (\z, \widetilde{g}(\z))
    + (\partial_{\vv} h) (\z, \widetilde{g}(\z)) \cdot (\partial_{\z_i} g_1)(\TranslateZak \z)
    + \ii \cdot (\partial_{\vv} h) (\z, \widetilde{g}(\z)) \cdot (\partial_{\z_i} g_2)(\TranslateZak \z) \\
  &= (\partial_{\z_i} h) (\z, \widetilde{g}(\z))
    + (\partial_{\vv} h) (\z, \widetilde{g}(\z)) \cdot (\partial_{\z_i} g)(\TranslateZak \z)
    .
\end{aligned}
\label{eq:EasyChainRule}
\end{equation}
This shows that the expected formula obtained from the chain rule also holds in this setting
of mixed ``real variables'' and ``complex variables'' derivatives.
Above, we also used that the simple translation
$\TranslateZak \z = \z - \TorusShiftVector$
has constant partial derivatives $1$.
\\~\\
\textbf{Step 2.} 
Let $\BoldEta \in \FrakI_\gamma \setminus \{ (\gamma) \}$.
In case of $\BoldEta = ()$, the claim is vacuously true;
thus, let $\BoldEta = (\eta^1,\dots,\eta^k)$.
In case of $k = 1$, we then have $\eta^1 = \overline{\BoldEta} \leq \gamma$.
Since furthermore $\BoldEta \neq (\gamma)$, this implies $\eta^1 < \gamma$
and hence $|\eta^1| \leq |\gamma| - 1$.
In case of $k \geq 2$, note that $\eta^i \in \N_0^2 \setminus \{ 0 \}$ for every
$i$, by definition of $\FrakI_\gamma$.
Hence, we get for every $1 \leq j \leq k$ that
\[
  \eta^j
  < \eta^1 + \cdots + \eta^k
  = \overline{\BoldEta}
  \leq \gamma
  .
\]
This implies as before that  $|\eta^j| \leq |\gamma| - 1$.
\\~\\
\textbf{Step 3.}
We prove \eqref{eq:FaaDiBruno} by induction on $|\gamma| \in \N$.
\\~\\
\emph{Base case:}
In this case, $\gamma = e_i$ with $i \in \{1, 2\}$.
We then have $\FrakI_\gamma = \{()\} \cup \{(e_i)\}$.
Moreover, by Step~1 (applied to $h = \ProjectiveMap$), we see
\[
  \begin{aligned}
    \partial^\gamma G(\z)
    &= \partial_{\z_i} [\ProjectiveMap(\z, g(\TranslateZak \z))] \\
    &= (\partial_{\z_i} \ProjectiveMap)(\z, g(\TranslateZak \z)) + (\partial_\vv \ProjectiveMap)(\z, g(\TranslateZak \z)) \cdot (\partial^{e_i} g)(\TranslateZak \z) \\
    &= \underbrace{\ProjectiveMap_{e_i, 0}(\z, g(\TranslateZak \z))}_{= G_{\gamma, ()}[g](\z)}
       + \underbrace{\ProjectiveMap_{0, 1}(\z, g(\TranslateZak \z)) \cdot (\partial^{e_i} g)(\TranslateZak \z)}_{= G_{\gamma, (e_i)}[g](\z)}
    ,
  \end{aligned}
\]
which proves the claim with
$
  c_{\gamma, ()} = c_{\gamma, (\gamma)} = 1
  .
$
\\~\\
\emph{Induction step:}
For the induction step, assume the claim is proven for $\gamma \in \N_0^2 \setminus \{0\}$, and let $i \in \{1, 2\}$.
By the product rule and by Step~1 of the proof (applied to $h = \ProjectiveMap_{\gamma-\overline{\BoldEta},k(\BoldEta)}$)),
we see for $\bm{\eta} \in \FrakI_\gamma$
\[
  \begin{aligned}
    \partial_i G_{\gamma, \bm{\eta}}(\z)
    &= \partial_i
       \left[
          \ProjectiveMap_{\gamma - \overline{\BoldEta}, k(\BoldEta)}(\z, g(\TranslateZak \z))
          \cdot \prod_{j=1}^{k(\bm{\eta})} (\partial^{\eta^j} g)(\TranslateZak \z)
       \right] \\
    &= \partial_i \left[ \ProjectiveMap_{\gamma - \overline{\BoldEta}, k(\BoldEta)}(\z, g(\TranslateZak \z)) \right]
       \cdot \prod_{j=1}^{k(\bm{\eta})} (\partial^{\eta^j} g)(\TranslateZak \z) \\
    &\quad +  \ProjectiveMap_{\gamma - \overline{\BoldEta}, k(\BoldEta)}(\z, g(\TranslateZak \z))
             \sum_{\ell=1}^{k(\BoldEta)}
             \left[
              (\partial^{\eta^\ell + e_i} g) (\TranslateZak \z)
              \prod_{\substack{1 \leq j \leq k(\BoldEta) \\ j \neq \ell}}
                (\partial^{\eta^j} g)(\TranslateZak \z) 
             \right]
    \\
    &= (\partial_i \ProjectiveMap_{\gamma - \overline{\BoldEta}, k(\BoldEta)})(\z, g(\TranslateZak \z))
       \cdot \prod_{j=1}^{k(\bm{\eta})} (\partial^{\eta^j} g)(\TranslateZak \z) \\
    &\quad + \ProjectiveMap_{\gamma - \overline{\BoldEta}, k(\BoldEta)+1}(\z, g(\TranslateZak \z))
             \cdot (\partial^{e_i} g)(\TranslateZak \z)
             \cdot \prod_{j=1}^{k(\bm{\eta})} (\partial^{\eta^j} g)(\TranslateZak \z) \\
    &\quad + \ProjectiveMap_{\gamma - \overline{\BoldEta}, k(\BoldEta)}(\z, g(\TranslateZak \z))
              \sum_{\ell=1}^{k(\bm{\eta})}
              \left[
                   (\partial^{\eta^\ell + e_i} g)(\TranslateZak \z)
                   \prod_{\substack{1 \leq j \leq k(\BoldEta) \\ j \neq \ell}}
                     (\partial^{\eta^j} g)(\TranslateZak \z)
              \right] \\
    &= G_{\gamma + e_i, \BoldEta}[g](\z)
       + G_{\gamma + e_i, (\BoldEta, e_i)} [g](\z)
       + \sum_{\ell=1}^{k(\BoldEta)}
           G_{\gamma + e_i, \BoldEta^\ast (i,\ell)} [g] (\z)
    ,
  \end{aligned}
\]
where $(\BoldEta, e_i) = (\eta^1,\dots,\eta^k,e_i)$
denotes the concatenation of $\BoldEta$ with $e_i$ and where we defined
\[
  \BoldEta^{*} (i, \ell)
  := (\eta^1, \dots, \eta^{\ell-1}, \eta^\ell + e_i, \eta^{\ell+1}, \dots, \eta^k)
  \in \FrakI_{\gamma + e_i}
  .
\]
Since $\BoldEta, (\BoldEta, e_i), \BoldEta^{\ast}(i,\ell) \in \FrakI_{\gamma + e_i}$,
this proves \eqref{eq:FaaDiBruno} by induction.
It only remains to verify that the resulting coefficient $c_{\gamma+e_i, (\gamma+e_i)}$
satisfies $c_{\gamma+e_i, (\gamma+e_i)} = 1$.
To this end, it is enough to show that the condition
\begin{align}
  (\gamma + e_i)
  \in \bigl\{ 
        \BoldEta, \,\,
        (\BoldEta, e_i), \,\,
        \BoldEta^{\ast}(i,1)
        ,\dots,
        \BoldEta^{\ast}(i,k(\BoldEta))
      \bigr\}
  \label{eq:FaaDiBrunoLeadingCoefficientCondition}
\end{align}
holds for precisely one $\BoldEta \in \FrakI_\gamma$,
and then precisely one of the elements of the set listed above
coincides with $(\gamma + e_i)$.

First, note that $\BoldEta = (\gamma) \in \FrakI_\gamma$ and for this choice of $\BoldEta$,
it is easy to see
\[
  \BoldEta^\ast (i, 1)
  = (\gamma + e_i)
  .
\]
We thus need only show that there is no other possibility for generating $(\gamma + e_i)$.
To do so, let $\BoldEta \in \FrakI_\gamma$, and note the following:
\begin{itemize}
    \item We have $\overline{\BoldEta} \leq \gamma < \gamma + e_i$,
          and thus $\BoldEta \neq (\gamma + e_i)$.
    \item If $(\BoldEta, e_i) = (\gamma + e_i)$, then this would imply $\BoldEta = ()$
          simply by comparing the length of the tuples.
          Hence, $(\gamma + e_i) = (\BoldEta, e_i) = (e_i)$
          and hence $\gamma = 0$, in contradiction to $\gamma \in \N_0^2 \setminus \{ 0\}$.
    \item Finally, if $\BoldEta^\ast (i,\ell) = (\gamma + e_i)$ for some $1 \leq \ell \leq k(\BoldEta)$,
          then this implies
          \[
            1 = k( (\gamma + e_i) )
              = k(\BoldEta^\ast (i, \ell))
              = k(\BoldEta)
            ,
          \]
          so that $\BoldEta = (\eta^1)$ and $\ell = 1$, which finally implies
          $(\gamma + e_i) = \BoldEta^\ast (i,1) = (\eta^1 + e_i)$
          and hence $\BoldEta = (\eta^1) = (\gamma)$.
\end{itemize}
This case distinction shows that indeed $\BoldEta = (\gamma)$ is the only element of $\FrakI_\gamma$
satisfying the condition \eqref{eq:FaaDiBrunoLeadingCoefficientCondition}
and that in this case, exactly one of the elements in the set appearing in that condition
coincides with $(\gamma + e_i)$, even when considering possibly repeated elements.
As noted above, this completes the proof of the lemma.
\end{proof}

\subsection{A bound for derivatives of \texorpdfstring{$G(\z)$}{G(z)}}
\label{sub:FaaDiBrunoDerivativeBound}

In the following, for $k \in \N_0$ and $\Omega \subset \R^2$, we write
\[
  \| g \|_{C^k (\Omega)}
  := \max_{\alpha \in \N_0^2, |\alpha| \leq k} \,\,
       \sup_{\z \in \Omega} \,\,
         |\partial^\alpha g (\z)|
  .
\]
Using this notation, we have the following result.

\begin{lemma}\label{lem:FaaDiBrunoBound}
With notation and assumptions as in Section~\ref{sub:FaaDiBrunoNotation},
for each $\gamma \in \N_0^2 \setminus \{ 0 \}$, there exists a
 map
\[
  R_\gamma : \quad C^\infty(\R^2, V) \to C^\infty(\R^2, \C)
\]
satisfying
\begin{equation}
  \partial^\gamma [\ProjectiveMap(\z, g(\TranslateZak \z))]
  = \ProjectiveMap_{0,1}(\z, g(\TranslateZak \z)) \cdot (\partial^\gamma g)(\TranslateZak \z)
    + R_\gamma [g](\z)
  \label{eq:FaaDiBrunoRemainder}
\end{equation}
and such that for arbitrary
compact sets $\Omega \!\subset\! \R^2$ and $V_0 \!\subset\! V$,
there exists a constant ${C_{\gamma} (\Omega,V_0) \!>\! 0}$ such that
\[
  \| R_\gamma [g] \|_{L^\infty(\Omega)}
  \leq C_{\gamma}(\Omega, V_0) \cdot
  (1 + \| g \|_{C^{|\gamma| - 1}(\Omega_{\TorusShiftVector})})^{|\gamma|}
\]
for all $g \in C^\infty(\R^2; V)$ with $g(\R^2) \subset V_0$.
\end{lemma}

\begin{proof}
Based on Lemma~\ref{lem:FaaDiBrunoResult}, set
\[
  R_\gamma [g] (\z)
  := \sum_{\BoldEta \in \FrakI_{\gamma} \setminus \{ (\gamma) \}}
       c_{\gamma, \BoldEta} \,\, G_{\gamma, \BoldEta}[g](\z).
\]
Given the formula for $G_{\gamma, \BoldEta}$, we see that $R_\gamma [g]$ only depends on
$g, \gamma$ (and $\ProjectiveMap$).
Moreover, setting
\[
  C_\gamma^{(0)} (\Omega, V_0)
  := \max_{\alpha \in \N_0^2, k \in \N_0, |\alpha| + k \leq |\gamma|} \,\,
       \max_{\z \in \Omega, \vv \in V_0} \,\,
         \left[
           | \partial_{\z}^\alpha \partial_{\vv}^k \ProjectiveMap (\z, \vv) |
         \right]
\]
and
\[
  C_\gamma (\Omega, V_0)
  := C_{\gamma}^{(0)}(\Omega, V_0) \cdot \sum_{\BoldEta \in \FrakI_{\gamma}} | c_{\gamma, \BoldEta} |
  ,
\]
we see for $\z \in \Omega$ because of $\widetilde{g} (\z) = g(\TranslateZak \z) \in V_0$ that
\begin{align*}
  |R_\gamma [g] (\z)|
  &\leq \sum_{\BoldEta \in \FrakI_{\gamma} \setminus \{(\gamma)\}}
        \left[
         | c_{\gamma, \BoldEta} |
         \cdot | \ProjectiveMap_{\gamma - \overline{\BoldEta}, k(\BoldEta)} (\z, \widetilde{g}(\z))|
         \cdot \prod_{j=1}^{k(\BoldEta)}
                 |(\partial^{\eta^j} g) (\TranslateZak \z)|
        \right]
        \\
  &\leq \sum_{\BoldEta \in \FrakI_{\gamma} \setminus \{(\gamma)\}}
         | c_{\gamma, \BoldEta} |
         \cdot C_\gamma^{(0)} (\Omega, V_0)
         \cdot \prod_{j=1}^{k(\BoldEta)}
                 \| g \|_{C^{|\gamma| - 1} (\Omega_{\TorusShiftVector})} \\
  &\leq \sum_{\BoldEta \in \FrakI_{\gamma} \setminus \{(\gamma)\}}
         | c_{\gamma, \BoldEta} |
         \cdot C_\gamma^{(0)} (\Omega, V_0)
         \cdot (1 + \| g \|_{C^{|\gamma| - 1} (\Omega_{\TorusShiftVector})})^{k(\BoldEta)} \\
  &\overset{(\ast)}{\leq} C_\gamma (\Omega, V_0)
        \cdot (1 + \| g \|_{C^{|\gamma| - 1} (\Omega_{\TorusShiftVector})})^{|\gamma|}
        .
\end{align*}
Here, we used that if $\BoldEta \in \FrakI_\gamma \setminus \{ (\gamma) \}$, then
$|\eta^j| \leq |\gamma| - 1$ (see Lemma~\ref{lem:FaaDiBrunoResult})
and that $|\eta^j| \geq 1$ for each $j$ (by definition of $\FrakI_\gamma$),
so that
\[
  k(\BoldEta) + |\gamma - \overline{\BoldEta}|
  = k(\BoldEta) + |\gamma| - \sum_{j=1}^{k(\BoldEta)} |\eta^j|
  \leq k(\BoldEta) + |\gamma| - \sum_{j=1}^{k(\BoldEta)} 1
  = |\gamma|
  .
\]
Moreover, the step marked with $(\ast)$ used that $k(\BoldEta) \leq |\gamma|$ for $\BoldEta \in \FrakI_\gamma$.
\end{proof}

\section{Additional calculations}
This appendix presents several calculations omitted from the main text to improve the article's overall readability.

\subsection{Proof of Equation \texorpdfstring{~\eqref{eq:s0-derivative-sech-formula}}{(3.13)}}
\label{sub:smooth-step-derivative}

We provide the details of the computation used in the derivative estimate for $\SmoothStep$ in the proof of Lemma~\ref{lem:s0-derivative-estimate}, Equation~\eqref{eq:s0-derivative-sech-formula}. Let $\ZakPosition\in(0,1)$ and set, as in \eqref{eq:s0-derivative-auxiliary-quantities},
\[
    t := 2\ZakPosition-1,
    \qquad
    r(t) := \frac{2t}{1-t^2}.
\]
Then
\[
    \ZakPosition=\frac{1+t}{2},
    \qquad
    1-\ZakPosition=\frac{1-t}{2},
\]
and hence
\begin{align*}
    \frac{1}{\ZakPosition}-\frac{1}{1-\ZakPosition}
    &=\frac{2}{1+t}-\frac{2}{1-t} 
    =-\frac{4t}{1-t^2}
    =-2r(t).
\end{align*}
Dividing the numerator and denominator in \eqref{eq:smooth-step} of $\SmoothStep$ by $\exp(-1/\ZakPosition)$ therefore yields
\begin{align*}
    \SmoothStep(\ZakPosition)
    &=\frac{1}{
    1+\exp\left(
    \frac{1}{\ZakPosition}-\frac{1}{1-\ZakPosition}
    \right)}
    =\frac{1}{1+\exp(-2r(t))}
    = \frac{1+\tanh(r(t))}{2},
\end{align*}
where we recall the definition $\tanh (r) = \frac{e^r - e^{-r}}{e^r + e^{-r}}$, which implies the last identity. Indeed,
\[
  \frac{1 + \tanh(r)}{2}
  = \frac{1}{2} \cdot \frac{e^r + e^{-r} + (e^r - e^{-r})}{e^r + e^{-r}}
  = \frac{e^r}{e^r + e^{-r}}
  = \frac{1}{1 + e^{-2r}}
  .
\]

To complete the proof, note that
\[
    \frac{\dd t}{\dd\ZakPosition}=2
    \qquad\text{and}\qquad
    r'(t)
    =\frac{2(1-t^2)+4t^2}{(1-t^2)^2}
    =2\frac{1+t^2}{(1-t^2)^2}.
\]
The chain rule and the well-known identity $\frac{\dd}{\dd y}\tanh(y)=\operatorname{sech}^2(y)$ now give
\begin{align*}
    \SmoothStep'(\ZakPosition)
    =\frac12\operatorname{sech}^2(r(t))
    r'(t)\frac{\dd t}{\dd\ZakPosition} 
    =2\frac{1+t^2}{(1-t^2)^2}
    \operatorname{sech}^2(r(t)),
\end{align*}
which proves \eqref{eq:s0-derivative-sech-formula}.
\hfill$\square$

\subsection{Proof of Equation\texorpdfstring{~\eqref{eq:sewingB_star}}{(6.6)}}
\label{sub:fixed-cocycle-sewing}

We verify the quasi-periodicity relations used in the proof of Lemma~\ref{lem:correction_zak}.
Let $\ZakPoint=(\ZakPosition,\ZakFrequency)\in\R^2$,
let $e_1=(1,0)$ and $e_2=(0,1)$, and let
$\SewingMatrix_j(\ZakPoint)$ denote the matrix in the corresponding
quasi-periodicity relations \eqref{eq:x-sewing} or \eqref{eq:omega-sewing}.
Thus
\begin{equation} \SewingMatrix_1(\ZakPoint)=\FirstSewingMatrix(\ZakFrequency),
 \qquad
 \SewingMatrix_2(\ZakPoint)=\SecondSewingMatrix.
\end{equation}

Recall from Lemma~\ref{lem:sewing_Weyl_matrix}
that each matrix function $\LatticeWeylMatrix_{m,n}$ satisfies \eqref{eq:sewing}.
In view of \eqref{eq:Astar}, this shows that
also $\FixedMatrixField$ satisfies the quasi-periodicity relations \eqref{eq:sewing}.
Hence, we have
\begin{equation}\label{eq:Astar-sewing-appendix}
 \FixedMatrixField(\ZakPoint+e_j)
 =
 \SewingMatrix_j(\ZakPoint)
 \FixedMatrixField(\ZakPoint)
 \SewingMatrix_j(\ZakPoint)^{-1},
 \qquad j=1,2.
\end{equation}

For $j=1$, the definitions of $\FirstSewingMatrix$ and
$\TranslateZak$ give
\begin{equation}
 \FirstSewingMatrix(\TranslateZak\ZakPoint)
 =\FirstSewingMatrix(\ZakFrequency-\TranslationBeta)
 =\e^{-\pi\ii\TranslationBeta}
  \FirstSewingMatrix(\ZakFrequency),
\end{equation}
whereas \eqref{eq:EtaDefinition} gives
\begin{equation}
 \OffsetPhase(\ZakPosition+1)
 =\e^{\pi\ii\TranslationBeta}\OffsetPhase(\ZakPosition).
\end{equation}
Using \eqref{eq:Bstar} and \eqref{eq:Astar-sewing-appendix}, we
therefore obtain
\begin{align*}
 \FirstSewingMatrix(\ZakPoint)
  \FixedCocycle(\ZakPoint)
  \FirstSewingMatrix(\TranslateZak\ZakPoint)^{-1}
 &=
 \FirstSewingMatrix(\ZakFrequency)
 \FixedMatrixField(\ZakPoint)\OffsetPhase(\ZakPosition)
 \bigl(
   \e^{-\pi\ii\TranslationBeta}
   \FirstSewingMatrix(\ZakFrequency)
 \bigr)^{-1}\\
 &=
 \e^{\pi\ii\TranslationBeta}
 \FirstSewingMatrix(\ZakFrequency)
 \FixedMatrixField(\ZakPoint)
 \FirstSewingMatrix(\ZakFrequency)^{-1}
 \OffsetPhase(\ZakPosition)\\
 &=
 \FixedMatrixField(\ZakPoint+e_1)
 \OffsetPhase(\ZakPosition+1) \\
 &=\FixedCocycle(\ZakPoint+e_1).
\end{align*}
For $j=2$, the matrix $\SecondSewingMatrix$ is constant and
$\OffsetPhase$ is independent of $\ZakFrequency$.  Hence,
\begin{align*}
 \SecondSewingMatrix
  \FixedCocycle(\ZakPoint)
  \SecondSewingMatrix^{-1}
 =
 \SecondSewingMatrix
 \FixedMatrixField(\ZakPoint)
 \SecondSewingMatrix^{-1}
 \OffsetPhase(\ZakPosition)
 =
 \FixedMatrixField(\ZakPoint+e_2)
 \OffsetPhase(\ZakPosition)
 =\FixedCocycle(\ZakPoint+e_2).
\end{align*}
Combining the two cases proves
$
 \FixedCocycle(\ZakPoint+e_j)
 =
 \SewingMatrix_j(\ZakPoint)
 \FixedCocycle(\ZakPoint)
 \SewingMatrix_j(\TranslateZak\ZakPoint)^{-1},
 $ for $j=1,2.
$

\subsection{Proof of Equation\texorpdfstring{~\eqref{eq:matrixB0}}{(6.21)}}
\label{sub:ideal-local-matrix}

Recall from \eqref{eq:B0Definition} that
\begin{equation}
 \IdealCocycle(\ZakPoint)
 =\ReferenceSection(\ZakPoint)
  \ReferenceSection(\TranslateZak\ZakPoint)^*.
\end{equation}
Since the columns of $\UnitaryFrame(\ZakPoint)$ are
$\ReferenceSection(\ZakPoint)$ and $\NormalSection(\ZakPoint)$, which
form an orthonormal basis,
\begin{equation}
 \UnitaryFrame(\ZakPoint)^*\ReferenceSection(\ZakPoint)
 =\begin{pmatrix}1\\0\end{pmatrix},
 \qquad
 \ReferenceSection(\TranslateZak\ZakPoint)^*
 \UnitaryFrame(\TranslateZak\ZakPoint)
 =\begin{pmatrix}1&0\end{pmatrix}.
\end{equation}
Consequently,
\begin{align*}
 \UnitaryFrame(\ZakPoint)^*
 \IdealCocycle(\ZakPoint)
 \UnitaryFrame(\TranslateZak\ZakPoint)
 &=
 \bigl(
  \UnitaryFrame(\ZakPoint)^*\ReferenceSection(\ZakPoint)
 \bigr)
 \bigl(
  \ReferenceSection(\TranslateZak\ZakPoint)^*
  \UnitaryFrame(\TranslateZak\ZakPoint)
 \bigr)\\
 &=
 \begin{pmatrix}1\\0\end{pmatrix}
 \begin{pmatrix}1&0\end{pmatrix}
 =\begin{pmatrix}1&0\\0&0\end{pmatrix},
\end{align*}
which proves Equation~\eqref{eq:matrixB0}.

\subsection{Proof of Equation~\texorpdfstring{\eqref{eq:estimates_differences}}{(6.28)}}
\label{sub:local-projective-estimates}

We provide the details for the estimates used in the proof of
Lemma~\ref{lem:projective-map}. 
Set
\begin{equation}
 P_0:=\begin{pmatrix}1&0\\0&0\end{pmatrix},
 \qquad
 E(\ZakPoint):=\LocalMatrix(\ZakPoint)-P_0.
\end{equation}
Put also
$
 r_0:=1-\frac54\CertifiedError.
$
By \eqref{eq:local-matrix-error},
\begin{equation}\label{eq:E-local-bound-appendix}
 \norm{E(\ZakPoint)}_{\op}\leq\CertifiedError
 \qquad\text{for all }\ZakPoint\in\R^2.
\end{equation}
Writing $[v]_j$ for the $j$-th component of a vector
$v\in\C^2$, the numerator and denominator from
\eqref{eq:local-projective-map} can be written as
\begin{equation}
 N_{\ZakPoint}(\ProjectiveCoordinate)
 =\left[
   \LocalMatrix(\ZakPoint)
   \begin{pmatrix}1\\\ProjectiveCoordinate\end{pmatrix}
  \right]_2,
 \qquad
 D_{\ZakPoint}(\ProjectiveCoordinate)
 =\left[
   \LocalMatrix(\ZakPoint)
   \begin{pmatrix}1\\\ProjectiveCoordinate\end{pmatrix}
  \right]_1.
\end{equation}
Since the second row of $P_0$ vanishes, while
$P_0(0,\ProjectiveCoordinate-\widetilde{\ProjectiveCoordinate})^T=0$,
we have
\begin{align*}
 N_{\ZakPoint}(\ProjectiveCoordinate)
 -N_{\ZakPoint}(\widetilde{\ProjectiveCoordinate})
 &=\left[
 E(\ZakPoint)
 \begin{pmatrix}
  0\\
  \ProjectiveCoordinate-\widetilde{\ProjectiveCoordinate}
 \end{pmatrix}
 \right]_2,\\
 D_{\ZakPoint}(\ProjectiveCoordinate)
 -D_{\ZakPoint}(\widetilde{\ProjectiveCoordinate})
 &=\left[
 E(\ZakPoint)
 \begin{pmatrix}
  0\\
  \ProjectiveCoordinate-\widetilde{\ProjectiveCoordinate}
 \end{pmatrix}
 \right]_1.
\end{align*}
The modulus of a component is bounded by the Euclidean norm of
the vector.  Hence \eqref{eq:E-local-bound-appendix} yields
\begin{align*}
 \left|
 N_{\ZakPoint}(\ProjectiveCoordinate)
 -N_{\ZakPoint}(\widetilde{\ProjectiveCoordinate})
 \right|
 &\leq\CertifiedError
 |\ProjectiveCoordinate-\widetilde{\ProjectiveCoordinate}|,\\
 \left|
 D_{\ZakPoint}(\ProjectiveCoordinate)
 -D_{\ZakPoint}(\widetilde{\ProjectiveCoordinate})
 \right|
 &\leq\CertifiedError
 |\ProjectiveCoordinate-\widetilde{\ProjectiveCoordinate}|.
\end{align*}
Furthermore, since the second row of $P_0$ vanishes,
\begin{align*}
 |N_{\ZakPoint}(\widetilde{\ProjectiveCoordinate})|
 &=\left|
  \left[
   E(\ZakPoint)
   \begin{pmatrix}1\\\widetilde{\ProjectiveCoordinate}\end{pmatrix}
  \right]_2
 \right|
 \leq\CertifiedError
 \sqrt{1+|\widetilde{\ProjectiveCoordinate}|^2}
 \leq\frac54\CertifiedError
\end{align*}
whenever
$|\widetilde{\ProjectiveCoordinate}|\leq3/4$.  Finally,
\begin{align*}
 |D_{\ZakPoint}(\ProjectiveCoordinate)|
 &=\left|
  1+\left[
   E(\ZakPoint)
   \begin{pmatrix}1\\\ProjectiveCoordinate\end{pmatrix}
  \right]_1
 \right|
 \geq
 1-\CertifiedError\sqrt{1+|\ProjectiveCoordinate|^2}
 \geq1-\frac54\CertifiedError
 =r_0.
\end{align*}
The same estimate holds with
$\widetilde{\ProjectiveCoordinate}$ in place of
$\ProjectiveCoordinate$.  These are precisely the four estimates used
before the application of \eqref{eq:quotient-difference}.

\subsection{Lemma~\ref{lem:OrthogonalVectorTransformation}}

\begin{lemma}\label{lem:OrthogonalVectorTransformation}
    For $v \in \C^2$ define
    \[
    v^\perp := \begin{pmatrix} -\overline{v_2} \\ \overline{v_1} \end{pmatrix} \in \C^2.
    \]
    Then, given a unitary matrix $U \in \C^{2 \times 2}$, it holds
    $
    (Uv)^\perp = \overline{\det U} \cdot U(v^\perp).
    $
\end{lemma}

\begin{proof}
Write $U = \left(\begin{smallmatrix} a & b \\ c & d \end{smallmatrix}\right)$.
By definition of $(Uv)^\perp$, we have
\[
  \begin{aligned}
    (Uv)^\perp
    &= \begin{pmatrix} -\overline{(Uv)_2} \\ \overline{(Uv)_1} \end{pmatrix} 
    = \begin{pmatrix} -\overline{cv_1 + dv_2} \\ \overline{av_1 + bv_2} \end{pmatrix} 
    = \begin{pmatrix} \overline{d} & -\overline{c} \\ -\overline{b} & \overline{a} \end{pmatrix}
       \begin{pmatrix} -\overline{v_2} \\ \overline{v_1} \end{pmatrix}
     = \begin{pmatrix} \overline{d} & -\overline{c} \\ -\overline{b} & \overline{a} \end{pmatrix} v^\perp.
  \end{aligned}
\]
Thus, it remains to show that
\(\left(\begin{smallmatrix} \overline{d} & -\overline{c} \\ -\overline{b} & \overline{a} \end{smallmatrix}\right) = \overline{\det U} \cdot U\).
To see that this holds, first note that
\[
  1
  = \det(I_2)
  = \det(UU^*)
  = \det U \cdot \det(U^*)
  = \det(U) \cdot \overline{\det(U)}.
\]
Moreover, by the well-known formula for the inverse of a $2 \times 2$ matrix, we have
\[
  \begin{aligned}
    \begin{pmatrix} \overline{a} & \overline{c} \\ \overline{b} & \overline{d} \end{pmatrix}
    = U^*
     = U^{-1} = \frac{1}{\det(U)} \begin{pmatrix} d & -b \\ -c & a \end{pmatrix} 
    = \overline{\det(U)} \begin{pmatrix} d & -b \\ -c & a \end{pmatrix}.
  \end{aligned}
\]
Comparing the components of this identity implies that
\[
  \overline{\det(U)} \cdot U
  = \overline{\det(U)} \begin{pmatrix} a & b \\ c & d \end{pmatrix}
  = \begin{pmatrix} \overline{d} & -\overline{c} \\ -\overline{b} & \overline{a} \end{pmatrix}.
\]
This completes the proof.
\end{proof}

\subsection{Lemma~\ref{lem:RecursiveSequenceEstimate}}

\begin{lemma}\label{lem:RecursiveSequenceEstimate}
    Let $\lambda_0 \in [0,1)$ and $C > 0$. If  $(a_n)_{n \in \N_0}$
    is a nonnegative real-valued sequence satisfying
    \[
      a_{n+1}
      \leq \lambda_0 \, a_n + C, \qquad
      n \in \N_0
      ,
    \]
    then $(a_n)_{n \in \N_0}$ is a bounded sequence. Moreover, 
    \[
      a_n \leq a_0 + \frac{C}{1 - \lambda_0}, 
      \qquad  n \in \N_0
      .
    \]
\end{lemma}

\begin{proof}
  We inductively show that, for $n \in \N_0$,
  \begin{equation}
    a_n \leq \lambda_0^n \, a_0 + C \sum_{j=0}^{n-1} \lambda_0^j
    \label{eq:RecursiveSequenceEstimateInduction}
    . 
  \end{equation}
  For $n = 0$, this is trivial.
  For the induction step, note that
  \begin{align*}
      a_{n+1}
      &\leq \lambda_0 \, a_n + C 
      \leq \lambda_0 \, (\lambda_0^n \, a_0 + C \sum_{j=0}^{n-1} \lambda_0^j) + C \\
      &=    \lambda_0^{n+1} a_0 + C \sum_{j=1}^{n} \lambda_0^j + C 
      =    \lambda_0^{n+1} a_0 + C \sum_{j=0}^{(n+1) - 1} \lambda_0^j
      ,
  \end{align*}
  proving \eqref{eq:RecursiveSequenceEstimateInduction}.

  Finally, note that \eqref{eq:RecursiveSequenceEstimateInduction} implies
  \[
    a_n \leq a_0 + C \sum_{j=0}^\infty \lambda_0^j
        =    a_0 + \frac{C}{1 - \lambda_0}
    .
    \qedhere
  \]
\end{proof}

\section{Computational supplement}\label{sec:computational-supplement}

The computational supplement contains only \nolinkurl{hrt_fixed_12_point_arb_1d_audit.py}. The eleven exact dyadic coefficients from Table \ref{tab:coefficients} are duplicated directly in that script, so no external data file is needed. With Python~3 and python-flint $0.8.0$ installed, run
\begin{center}
    \ttfamily python hrt\_fixed\_12\_point\_arb\_1d\_audit.py
\end{center}
from the directory containing the script.

\section{Numerical illustration of linear dependence}
\label{sec:numerical-illustrations}

We show a floating-point reconstruction of the final
Schwartz window in Figure \ref{fig:gstar-approximation} and a direct numerical evaluation of \eqref{eq:fixed-dependence} in Figure \ref{fig:gstar-cancellation}. 
The purpose of this appendix is only to visualize the function produced by the abstract construction and
to act as a sanity check that cancellation in physical space happens.
We stress that none of the computations in this appendix are used in the proof above.

\subsection{Reconstruction of the window}

For $\GridSize=256$, we used a uniform
$\GridSize\times\GridSize$ grid on $[0,1]^2$, with
\begin{equation}
 \ZakPoint_{r,s}
 :=
 \left(\frac r{\GridSize},\frac s{\GridSize}\right),
 \qquad
 0\le r,s<\GridSize.
\end{equation} 
We then compute an approximation $\GridInvariantVector$ of
$\FixedInvariantVector$ on the grid using $150$ iterations.
Starting at a distant backward translate with the function $\ReferenceSection$, the
iteration repeatedly applies $\FixedCocycle$ and normalizes the result by its
$\ReferenceSection$-component: 
\begin{equation}
 v_{\mathrm{out}}(\ZakPoint)
 =\FixedCocycle(\ZakPoint)v(\TranslateZak\ZakPoint),
 \qquad
 v_{\mathrm{new}}(\ZakPoint)
 =
 \frac{
  v_{\mathrm{out}}(\ZakPoint)
 }{
  \ReferenceSection(\ZakPoint)^*v_{\mathrm{out}}(\ZakPoint)
 }.
\end{equation}

The resulting grid function $\GridInvariantVector$ approximates
$\FixedInvariantVector$. Its scalar
multiplier is then evaluated according to the formula
\begin{equation}
 \GridMultiplier(\ZakPoint)
 =
 \ReferenceSection(\ZakPoint)^*
 \FixedCocycle(\ZakPoint)
 \GridInvariantVector(\TranslateZak\ZakPoint),
\end{equation}
which results from its definition in \eqref{eq:vq} and \eqref{eq:denom}.

The correction of Proposition \ref{prop:cohomology} is
computed as follows: Put
$\GridLogMultiplier=\Log\GridMultiplier$, take its two-dimensional
discrete Fourier
transform, and set
\begin{equation}\label{eq:appendix-discrete-u}
 \widehat{\GridCohomologySolution}(k)=
 \begin{cases}
 \displaystyle
 \frac{\widehat{\GridLogMultiplier}(k)}
 {1-\e^{-2\pi\ii k\cdot\TorusShiftVector}},
   &k\ne0,\\[3mm]
 0,&k=0.
 \end{cases}
\end{equation}
The inverse discrete Fourier transform yields
$\GridCohomologySolution$, and we form
\begin{equation}
 \GridScalarCorrection=\e^{\GridCohomologySolution},
 \qquad
 \GridZakEigenfunction
 =\GridScalarCorrection\GridInvariantVector,
 \qquad
 \GridEigenvalue
 =\exp\left(
  \frac1{\GridSize^2}
  \sum_{r,s=0}^{\GridSize-1}
  \GridLogMultiplier(\ZakPoint_{r,s})
 \right).
\end{equation}

The computation yields,
\begin{equation}
 \GridEigenvalue
 =0.9773322606287828
   +6.8\mathbin{\cdot}10^{-18}\ii.
\end{equation}

The scalar Zak samples are recovered from the two components of
$\GridZakEigenfunction$ by following
\eqref{eq:recoveredZakFunction}. Thus, for each of the
$\GridSize$ sampled $\ZakPosition$ values, we recover
$\GridScalarZak(\ZakPosition,\ZakFrequency)$ on the one-dimensional
$2\GridSize$-point frequency grid
$$
\ZakFrequency_{\ell}
=\frac{\ell}{2\GridSize},
\qquad
0\le\ell<2\GridSize.
$$
Concretely,
\begin{equation}
 \GridScalarZak(\ZakPosition,\ZakFrequency)=\sqrt2
 \begin{cases}
(\GridZakEigenfunction)_1
    (\ZakPosition,2\ZakFrequency),
    &0\le\ZakFrequency<\frac12,\\
  (\GridZakEigenfunction)_2
    (\ZakPosition,2\ZakFrequency-1),
    &\frac12\le\ZakFrequency<1.
 \end{cases}
\end{equation}

For $\ZakPosition=r/\GridSize$ and $j\in\Z$, the inverse Zak
integral is approximated
by the inverse discrete Fourier transform
\begin{equation}\label{eq:appendix-inverse-zak}
 \GridWindow(\ZakPosition+j)
 =
 \frac1{2\GridSize}
 \sum_{\ell=0}^{2\GridSize-1}
 \GridScalarZak\left(
  \ZakPosition,\frac{\ell}{2\GridSize}
 \right)
 \e^{2\pi\ii j\ell/(2\GridSize)}.
\end{equation}
Finally, $\GridWindow$ is normalized to have discrete $L^2$-norm
one, and
one irrelevant global phase is chosen so that its largest displayed
sample is positive real.
The maximum residual of the discretized invariant-line equation
\begin{equation}R_{\mathrm{line},G}(\ZakPoint_{r,s})
 =
 \FixedCocycle(\ZakPoint_{r,s})
 \GridInvariantVector(\TranslateZak\ZakPoint_{r,s})
 -
 \GridMultiplier(\ZakPoint_{r,s})
 \GridInvariantVector(\ZakPoint_{r,s}),
\end{equation}
where
\begin{equation}
 \GridMultiplier(\ZakPoint_{r,s})
 =
 \ReferenceSection(\ZakPoint_{r,s})^*
 \FixedCocycle(\ZakPoint_{r,s})
 \GridInvariantVector(\TranslateZak\ZakPoint_{r,s})
\end{equation}
was $5.66\mathbin{\cdot}10^{-16}$.
The residuals of the normalization
$\ReferenceSection^*\GridInvariantVector=1$ and of
\eqref{eq:additive-cohomology} were, respectively,
\begin{equation}
 4.56\mathbin{\cdot}10^{-16},
 \qquad
 9.02\mathbin{\cdot}10^{-17}.
\end{equation}
We stress again that these values measure the consistency of the floating-point grid
calculation and are not bounds for the exact mathematical objects of the main argument of this manuscript.

\begin{figure}[htb]
\centering
\includegraphics[width=\textwidth]{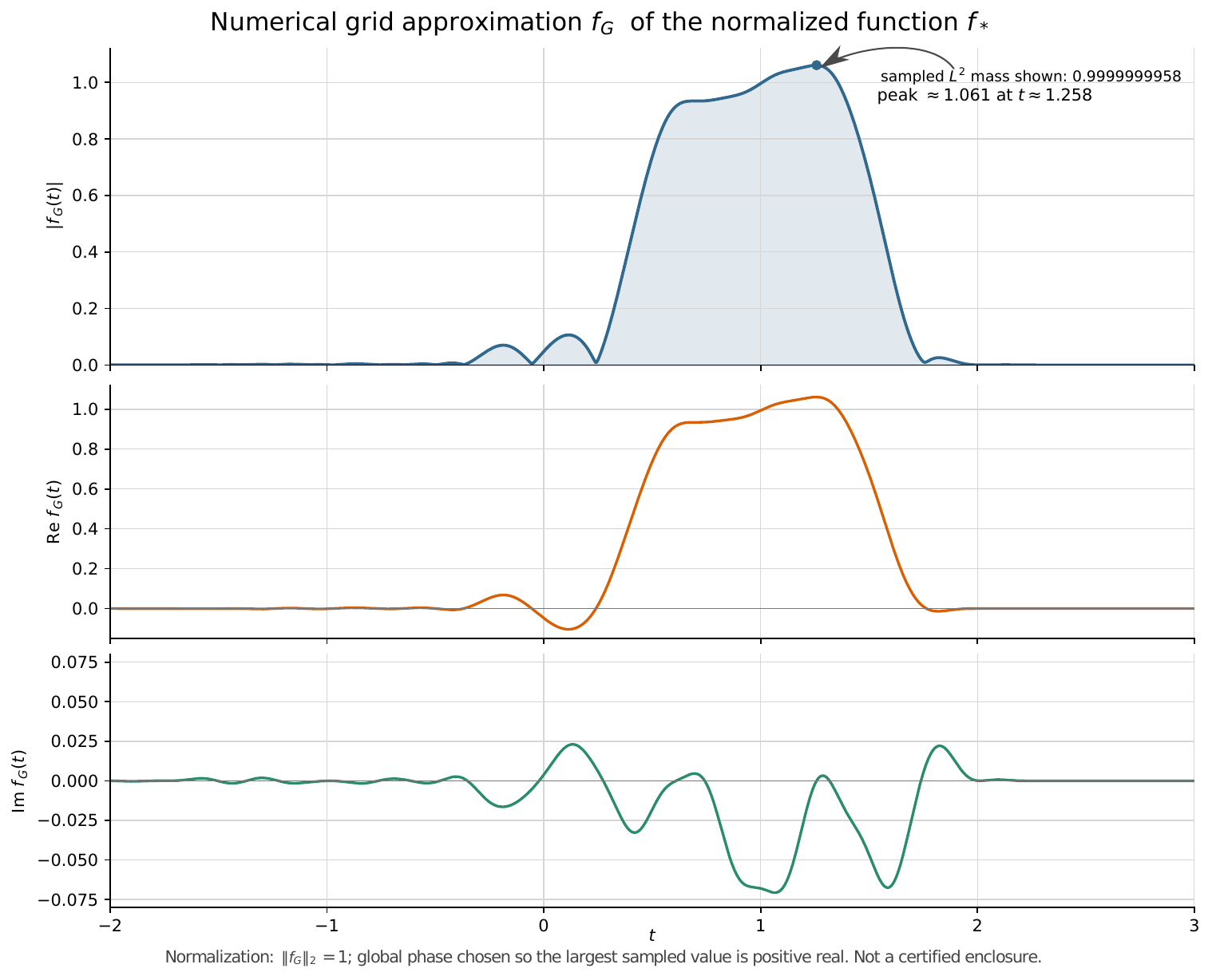}
\caption{Floating-point approximation $\GridWindow$ of the final window.
The three panels show its modulus, real part, and imaginary part.
With the stated normalization, the largest displayed sample has
$\abs{\GridWindow(\PhysicalVariable)}\approx1.061$ at
$\PhysicalVariable\approx1.258$. The sampled
$L^2$-mass in $[-2,3)$ is $0.9999999958$.}
\label{fig:gstar-approximation}
\end{figure}

\subsection{Direct evaluation of the twelve-term relation}
We next evaluate \eqref{eq:fixed-dependence} numerically with
$\FinalWindow$ replaced by $\GridWindow$:
\begin{align}
 \GridResidual(\PhysicalVariable)
 &=
 -\GridEigenvalue\GridWindow(\PhysicalVariable)\notag\\
 &\quad+
 \sum_{(m,n)\in\CoefficientIndexSet}
 \FixedWeylCoefficient_{m,n}\e^{\frac{\pi\ii}{2}(n\TranslationAlpha-m\TranslationBeta)}
 \e^{\pi\ii(n+\TranslationBeta)
       \left(
        \PhysicalVariable-\frac{m+\TranslationAlpha}{2}
       \right)}
 \GridWindow(\PhysicalVariable-m-\TranslationAlpha).
\label{eq:appendix-physical-residual}
\end{align}

The second exponential in
\eqref{eq:appendix-physical-residual} is precisely the symmetric Weyl
phase from \eqref{eq:Weyl}, with translation parameter
$m+\TranslationAlpha$ and modulation parameter
$(n+\TranslationBeta)/2$.
The residual was evaluated at the $\GridSize$ equally spaced points,
including the left endpoint, in every
interval $[j, j+1)$ for $j = -2,\dots, 3$.

The shifted values
$\GridWindow(\PhysicalVariable-m-\TranslationAlpha)$ were not
obtained by linearly
interpolating the plotted curve.  Instead, the approximate section
$\GridZakEigenfunction$ was evaluated at the shifted fractional
coordinate $\ZakPosition-\TranslationAlpha\pmod1$, after which the
same inverse Zak transform
\eqref{eq:appendix-inverse-zak} was applied.  This spectral evaluation
avoids introducing a separate low-order interpolation error.

Write
$\GridSummand_{0,\GridSize}
 =-\GridEigenvalue\GridWindow$, and let
$\GridSummand_{1,\GridSize},\ldots,
  \GridSummand_{11,\GridSize}$
denote the eleven summands in
\eqref{eq:appendix-physical-residual}.  On the displayed interval, the
largest individual term and the largest sum of term magnitudes were
\begin{equation}
 \max_{j,\PhysicalVariable}
 \abs{\GridSummand_{j,\GridSize}(\PhysicalVariable)}
 =1.037,
 \qquad
 \max_{\PhysicalVariable}
 \sum_{j=0}^{11}
 \abs{\GridSummand_{j,\GridSize}(\PhysicalVariable)}
 =2.368.
\end{equation}
In contrast, the directly computed residual satisfies
\begin{equation}
 \max_{\PhysicalVariable}
 \abs{\GridResidual(\PhysicalVariable)}
 =7.96\mathbin{\cdot}10^{-16},
 \qquad
 \frac{\norm{\GridResidual}_{L^2([-2,4))}}
 {\sum_{j=0}^{11}
  \norm{\GridSummand_{j,\GridSize}}_{L^2([-2,4))}}
 =1.21\mathbin{\cdot}10^{-16}.
\end{equation}

\begin{figure}[htb]
\centering
\includegraphics[width=\textwidth]{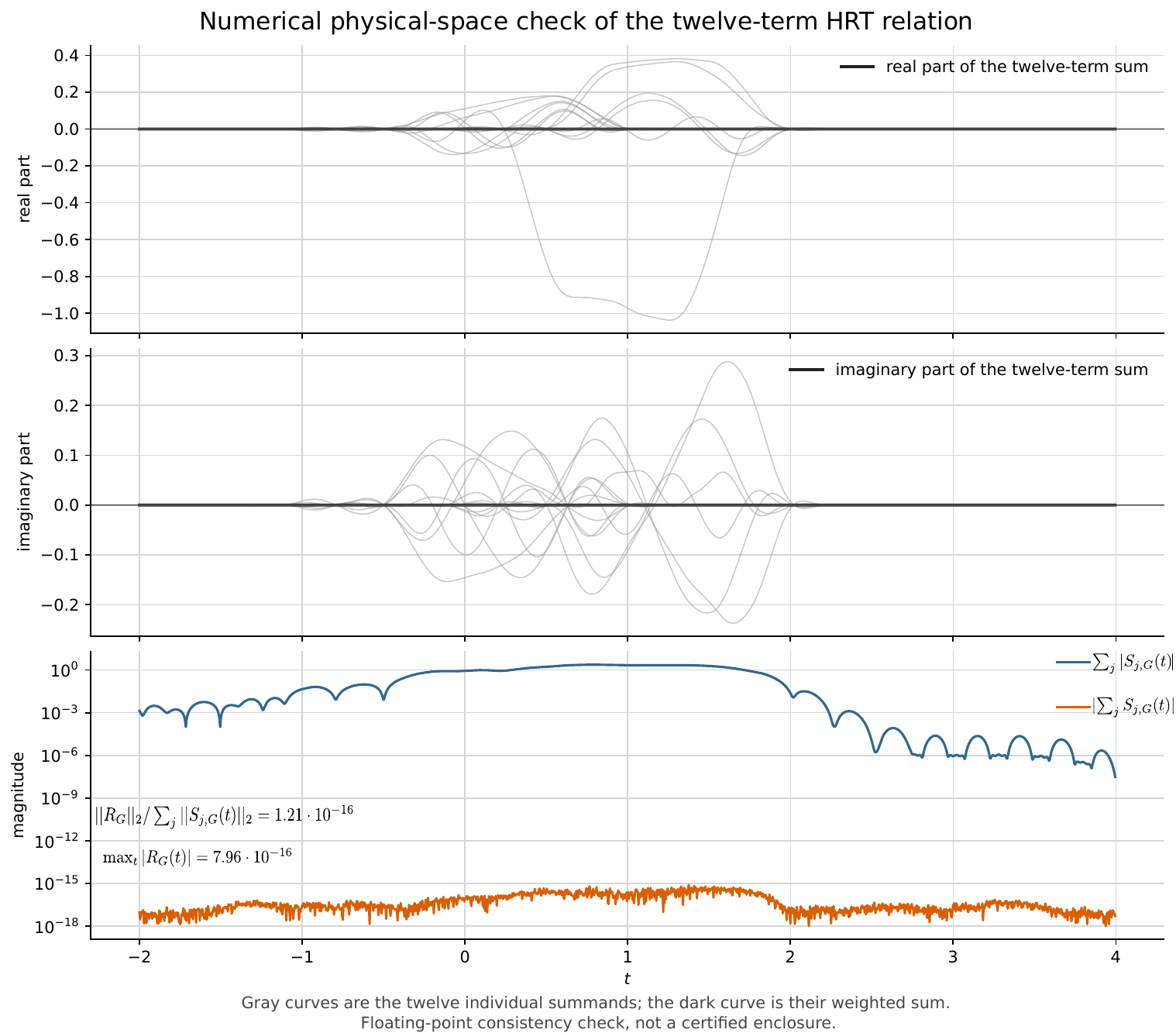}
\caption{Direct floating-point evaluation of the twelve summands in \eqref{eq:appendix-physical-residual}. In the first two panels, the gray curves are the real and imaginary parts of the individual summands, while the dark curve is their sum.  In the bottom panel, the blue curve is the sum of the magnitudes of the individual terms
$\sum_{j=0}^{11}
    \abs{\GridSummand_{j,\GridSize}(\PhysicalVariable)}
$, and the orange curve is the residual $\abs{\GridResidual(\PhysicalVariable)}$. The latter remains at approximately double-precision rounding level.}
\label{fig:gstar-cancellation}
\end{figure}

\newpage

\bibliographystyle{abbrv}
\bibliography{bib}

\end{document}